\documentclass[a4paper, 12pts]{article}
\usepackage{amssymb, url, amsthm, amsmath, undertilde, scalerel, stmaryrd, enumitem, lipsum, tikz , tikz-cd, sectsty, filecontents, enumitem}

\usepackage[nottoc,notlof,notlot]{tocbibind} 

\newcommand{\lusim}[1]{\smash{\underset{\raisebox{1.2pt}[0cm][0cm]{$\sim$}}
		{{#1}}}}

\chaptertitlefont{\LARGE}

\makeatletter
\DeclareRobustCommand\widecheck[1]{{\mathpalette\@widecheck{#1}}}
\def\@widecheck#1#2{%
	\setbox\z@\hbox{\m@th$#1#2$}%
	\setbox\tw@\hbox{\m@th$#1%
		\widehat{%
			\vrule\@width\z@\@height\ht\z@
			\vrule\@height\z@\@width\wd\z@}$}%
	\dp\tw@-\ht\z@
	\@tempdima\ht\z@ \advance\@tempdima2\ht\tw@ \divide\@tempdima\thr@@
	\setbox\tw@\hbox{%
		\raise\@tempdima\hbox{\scalebox{1}[-1]{\lower\@tempdima\box
				\tw@}}}%
	{\ooalign{\box\tw@ \cr \box\z@}}}
\makeatother

\theoremstyle{plain}
\newtheorem{thm}{Theorem}[section]
\newtheorem{claim}[thm]{Claim}
\newtheorem{lemma}[thm]{Lemma}
\newtheorem{corollary}[thm]{Corollary}
\newtheorem{remark}[thm]{Remark}

\newtheorem{defn}[thm]{Definition}

\newtheorem*{remark*}{Remark}
\newtheorem*{lemma*}{Lemma}
\newtheorem*{claim*}{Claim}
\newtheorem*{thm*}{Theorem}

\begin{document}

\title{Non-stationary support iterations of Prikry Forcings and
	Restrictions of Ultrapower Embeddings to the Ground Model}

\baselineskip=18pt
\author{ Moti Gitik\footnote{ The work was partially supported by ISF grant No. 1216/18.}  and Eyal Kaplan}

\date{\today}
\maketitle

\begin{abstract}
	We continue the study started in \cite{RestElm} and characterize $j\restriction_V$, where $j:V[G]\to M[H]$ is an ultrapower embedding by a normal ultrafilter after a non-stationary support iteration of Prikry forcings.
\end{abstract}

\section*{Introduction}

Let $ P $ be a forcing notion, and assume that $ G\subseteq P $ is generic over $ V $. Assume that a cardinal $ \kappa $ is measurable in $ V\left[G\right] $, and let $ W\in V\left[G\right] $ be a normal measure on $ \kappa $, with a corresponding ultrapower embedding $ j_W\colon V\left[G\right]\to M\left[H\right] $. We continue our study from \cite{RestElm}, and consider the embedding $ j_W\restriction_{V} $, focusing on the following questions:
\begin{enumerate}
	\item Is $ j_W\restriction_{V} $ an iteration of $ V $ (by its measures or extenders)?
	\item Is $ j_W\restriction_{V} $ definable over $ V $?
\end{enumerate}
The answers to both questions depend on the forcing notion $ P $ and the ground model $ V $. The first question is answered affirmatively, for every forcing notion $ P $, assuming that there is no inner model with a Woodin cardinal, and   $ V = K $ is the core model \cite{schindler2006iterates}. The answer to the second question can go both ways. For instance, if $ P $ has a gap below $ \kappa $, in the sense of \cite{hamkins2001gap}, $ j_W\restriction_{V} $ is a definable class of $ V $, no matter what is the ground model. On the other hand, the answer for $2$ can be negative, even when we force over the core model and $ \kappa $ is measurable there (see, e.g., section 5.2 in \cite{RestElm}).

In this paper we focus on iterations of Prikry forcings. Let $ \kappa $ be a measurable limit of measurables, and assume that $ \mbox{GCH} $ holds below $ \kappa $. We would like to perform an iterated forcing, destroying the measurability of every  measurable cardinal $ \alpha<\kappa $. Our main goal will be the following theorem:

\begin{thm}
	Assume $ \mbox{GCH} $ below $ \kappa $. Let $ P $ be a nonstationary support iteration of Prikry forcings below $ \kappa $. Let $ G\subseteq P $ be generic over $ V $, and $ W\in V\left[G\right] $ be a normal measure on $ \kappa $ in $ V\left[G\right] $. Then $ j_W\restriction_{V} $ is an iterated ultrapower of $ V $ by normal measures.
	\\Moreover, a description of such iteration is given. 
\end{thm}

We focus on the nonstationary-support iteration for sake of simplicity; the full-support iteration will be considered in a future work.

This paper is structured as follows: In section $ 1 $, we present the forcing and its basic properties. In section $ 2 $, we characterize all the normal measures $ W\in V\left[G\right] $ on $ \kappa $, using and extending results from  \cite{ben2017homogeneous} and \cite{RestElm}; More specifically, we prove that every such measure is the unique extension of some normal measure of Mitchell order $ 0 $ on $ \kappa $ in $ V $. In section $ 3 $, we present the structure of $ j_W\restriction_{V} $ as an iterated ultrapower, and provide a sufficient condition for its definability in $ V $. Finally, in section $ 4 $, we study iterated ultrapowers of $ V $ in general, developing tools for computation of cofinalities, in $ V $, of ordinals which become inaccessibles at some stage in an iteration; we apply those tools to simplify the  presentation of $ j_W\restriction_{V} $ as an iteration of $ V $.

We assume throughout this paper that $ \mbox{GCH}_{\leq \kappa} $ holds.

\section{The Forcing} 

\begin{defn}
An iteration $ \langle P_\alpha, \lusim{Q}_\beta \colon \alpha\leq\kappa\ ,\ \beta <\kappa \rangle $ is called a nonstationary support iteration of Prikry-type forcings if and only if, for every $ \alpha\leq\kappa $ and $p\in P_\alpha$,
\begin{enumerate}
\item  $ p $ is a function with domain $ \alpha $ such that for every $\beta <\alpha$, $p\restriction \beta  \in P_\beta$, and $p\restriction \beta \Vdash p(\beta) \in \lusim{Q}_\beta \mbox{ and }  \langle \lusim{Q}_{\beta}, \lusim{\leq}_{\lusim{Q}_{\beta} }, \lusim{\leq}^*_{\lusim{Q}_{\beta} } \rangle \ \mbox{ is a Prikry-type forcing.}$ 
\item If $ \alpha \leq \kappa $ is inaccessible, then  $\mbox{supp}(p)\cap \alpha$ is nonstationary in $\alpha$ (where $ \mbox{supp}(p)\subseteq \alpha$ is the complement of the set $\{ \beta< \alpha \colon p\restriction_{\beta}\Vdash p(\beta) \mbox{ is trivial} \}$). In other words, there exists a club $ C\subseteq \alpha $ such that for every $ \beta\in C $, $ p\restriction_{\beta}\Vdash p(\beta) \mbox{ is trivial} $.
\end{enumerate}

Suppose that $p,q \in P_\alpha$. Then $p\geq q$, which means that $p$ extends $q$, holds if and only if:
\begin{enumerate}
\item $\mbox{supp}(q)\subseteq \mbox{supp}(p)$.
\item For every $\beta \in \mbox{supp}(q)$, $p\restriction \beta \Vdash p(\beta)\geq_\beta q(\beta)$ (where $ \geq_\beta$ is the order of $ Q_\beta $).
\item There is a finite subset $b\subseteq \mbox{supp}(q)$, such that for every $\beta \in \mbox{supp}(q)\setminus b$, $p\restriction \beta \Vdash p(\beta)\geq^{*}_\beta q(\beta)$ (where $ \geq^{*}_{\beta} $ is the direct extension order of $ Q_\beta $).
\end{enumerate} 
If $b=\emptyset$, we say that $p$ is a direct extension of $q$, and denote it by $p\geq^* q$. 

\end{defn}

In this section we consider a nonstationary support iteration of Prikry forcings,  $ \langle P_\alpha, \lusim{Q}_\beta \colon \alpha\leq\kappa\ ,\ \beta <\kappa \rangle $. Denote by $ \Delta \subseteq \kappa $ the set of measurable cardinals below $ \kappa $ in $ V $. Assume that $ \alpha \in \Delta $ and $ P_{\alpha} $ has been defined. Assume that $ \lusim{U}^{*}_{\alpha} $ is a $ P_{\alpha} $-name for a normal measure on $ \alpha $ in $ V^{P_{\alpha}} $ (we will prove that at least one such measure exists). Let $ \lusim{Q}_{\alpha} $ be the Prikry forcing with $ \lusim{U}^*_{\alpha} $. If $ \alpha  $ is not measurable in $ V $,  $ \lusim{Q}_{\alpha} $ is the trivial forcing.

We did not specify the normal measure $ U^{*}_{\alpha} $ which is used at stage $ \alpha $. As we will prove, each such measure in $ V^{P_{\alpha}} $ is the unique extension of a normal measure $ U_{\alpha} $ of Mitchell order $ 0 $ in $ V $. Let $ \lusim{ \mathcal{U}} = \langle \lusim{U}_{\alpha} \colon \alpha\in \Delta \rangle $ be a sequence of names, such that, for every $ \alpha\in \Delta $, $ \lusim{U}_{\alpha} $ is forced by the weakest condition in $ P_{\alpha} $ to be $ \lusim{U}^{*}_{\alpha} \cap V $. Given $ G\subseteq P_{\kappa} $ generic over $ V $, let $ \mathcal{U} = \langle U_{\alpha} \colon \alpha\in \Delta \rangle $ be the interpretation of the names in $ \lusim{\mathcal{U}} $.  Then $ \mathcal{U} $ is a sequence of measures in $ V $, but $ \mathcal{U} $ itself does not necessarily belong to $ V $. Since $\mathcal{U} $ depends on $ G $, a more accurate notation would be $ \mathcal{U}_G $, but most of the time $ G $ will be clear from the context.

An iteration of Prikry-type forcings with nonstationary support was studied in \cite{ben2017homogeneous}. The following key property was proved:

\begin{lemma} \label{Lemma: NS, Prikry property}
$P= P_\kappa$ satisfies the Prikry property.
\end{lemma}

The proof relies on a fusion property which holds in our iteration. We will use the formulation of this  property as it is stated and proved in \cite{ben2017homogeneous}:

\begin{lemma} (Fusion Lemma) \label{Lemma: NS, Fusion Lemma}
Let $ \lambda \leq \kappa $ be a limit ordinal, and assume that $ p\in P_{\lambda} $. Suppose that $e\colon \lambda \to V$ is a function such that for every $\alpha <\lambda$, $ e(\alpha) $ is a $ P_{\alpha+1} $-name, such that,
\begin{align*}
p\restriction_{\alpha+1}\Vdash & \mbox{"} e(\alpha) \mbox{ is a dense open subset of  } P_{\lambda}\setminus \left(\alpha+1\right)  \mbox{ above } p\setminus \left(\alpha+1\right)  \mbox{, } \\
& \mbox{ with respect to the direct extension order."}
\end{align*}
assume also that $ \nu< \lambda $ is an ordinal. Then there exists $p^* \geq^{*} p$ which satisfies $ p^* \restriction_{ \nu } = p\restriction_{ \nu } $, and a club $ C\subseteq \lambda $,  such that for every $ \alpha\in C $,
$$ p^* \restriction_{\alpha+1} \Vdash p^* \setminus \left(\alpha+1\right)  \in e(\alpha)$$
\end{lemma}

The Fusion Lemma will be applied repeatedly in this paper, and is standard in nonstationary support iterations. For sake of completeness, we provide the proof.

\begin{proof}
As in \cite{ben2017homogeneous}, we focus first on the case where $ \lambda $ is an inaccessible cardinal. The other case is simpler since an inverse limit is taken at $ \lambda $.
	
We construct a sequence $\langle p_\xi\colon \xi< \lambda \rangle$ of conditions in $ P_{\lambda} $, a sequence $\langle \nu_\xi  \colon \xi< \lambda\rangle $ of ordinals below $ \lambda $ and a sequence of clubs $ \langle C_{\xi} \colon \xi <\eta \rangle $, such that, 
\begin{enumerate}
\item The sequence $\langle p_\xi\colon \xi< \lambda \rangle$ is increasing with respect to direct extensions.
\item The sequence $\langle \nu_\xi\colon \xi< \lambda \rangle$ is increasing, continuous and unbounded in $\lambda$.
\item For every $ \xi<\lambda $, $ C_{\xi} \cap \mbox{supp}\left( p_{\xi} \right) = \emptyset $.
\item For every $ \xi< \lambda $, $ \{ \nu_{\eta} \colon \eta<\lambda \} $ is disjoint from the support of $ p_\xi $.
\item For every $ \xi<\lambda $, $ p_{\xi}\restriction_{\left(\nu_{\xi}+1\right)} \Vdash p_{\xi}\setminus \left( \nu_{\xi}+1 \right) \in e\left( \nu_{\xi} \right) $.
\item Whenever $\eta<\xi< \lambda$,
\begin{enumerate}
\item $ \mu_{\xi}\in C_{\eta} $.
\item $p_\xi \restriction_{\nu_\eta+1} = p_\eta \restriction_{\nu_\eta +1}$.
\item $p_\xi \restriction_{\nu_\eta+1 }\Vdash p_\xi\setminus \left( \nu_\eta+1 \right) \geq^{*} p_\eta \setminus \left( \nu_\eta+1 \right)$.
\end{enumerate}
\end{enumerate}
Take $p_0 = p$, $ C_0 $ a club disjoint from $ \mbox{supp}\left( p_0 \right) $,  and $ \nu_0 > \nu$ in $ C_0 $.

\textbf{Successor stages:} Suppose that the construction is done up, and including, some $\xi<\lambda$, and let us construct $p_{\xi+1}$ and $ \nu_{\xi+1}$.  Define-- 
$$\nu_{\xi+1} = \min\left(\bigcap_{\eta<\xi+1} C_\eta \setminus \left(\nu_\xi+1\right)\right)$$
Let us construct $p_{\xi+1}$. First, we require $p_{\xi+1}\restriction_{ \nu_{\xi+1}+1} = p_{\xi}\restriction _{\nu_{\xi+1}+1}$. Now, there exists a $P_{\nu_{\xi+1}}$-name for a direct extension of $p_{\xi} \setminus \nu_{\xi+1}$ which is forced, by $p_{\xi}\restriction_{\nu_{\xi+1}+1}$, to belong to $e(\nu_{\xi+1})$. Let $\sigma$ be this name, and take $p_{\xi+1} \setminus  \nu_{\xi+1} = \sigma$. There exists a $P_{\nu_{\xi+1}}$-name $\lusim{C}$ for a club in $\lambda$ disjoint from $\mbox{supp}(\sigma)$; Since $\lambda$ is inaccessible,  $P_{\nu_{\xi+1}}$ is $\lambda$-c.c., so there exists a club in $\lambda$, $C' \in V$, which is forced to be a subset of $\lusim{C}$. Hence $ p_{\xi+1} $ has a club $ C_{\xi+1} \in  V $ disjoint from its support, and is a legitimate condition in $ P_{\lambda} $.

\textbf{Limit stages:} Suppose that $\xi< \lambda$ is a limit ordinal. Set $\nu_{\xi} = \cup_{\eta<\xi} \nu_{\eta} $. For every $ \eta<\xi $, $ \nu_{\xi} $ is a limit point of $ C_\eta $, and thus $ \nu_{\xi}\notin \mbox{supp}\left( p_{\eta} \right) $. Let us construct $p_{\xi}$. We construct it such that $\nu_\xi \notin \mbox{supp}(p_\xi)$. First, we set-- 
$$p_\xi \restriction_{\nu_{\xi}} = \bigcup_{\eta<\xi} p_{\eta} \restriction_{\nu_{\eta}+1}$$
note that $ \langle \nu_{\eta} \colon \eta<\xi \rangle $ is disjoint from the support of $ p_{\xi}\restriction_{\nu_{\xi}} $, so $ p_{\xi}\restriction_{\nu_{\xi}} \in P_{\nu_{\xi}} $ holds even if $ \nu_{\xi} $ is inaccessible. Also, $p_{\xi} \restriction_{\nu_{\xi}+1}$ forces that $\langle p_{\eta} \setminus \left(\nu_{\xi}+1\right) \colon \eta<\xi \rangle$ is an increasing sequence with respect to direct extension in $P \setminus \left(\nu_\xi+1\right) $, which is forced to be $\left|\nu_{\xi}\right|^{+}$-closed (so it's definitely more than $\xi$-closed). Thus, there exists an upper bound. Take $p_\xi \setminus \left(\nu_\xi+1\right)$ to be a name, which is forced, by $p_\xi \restriction_{\nu_\xi+1}$, to be a direct extension of the upper bound which belongs to $e(\nu_\xi)$. Pick $ C_{\xi}\subseteq \lambda $ as a club disjoint from $\mbox{supp}\left( p_{\xi} \right)$.

This finishes the construction. Finally, set--
$$p^* = \bigcup_{\xi<\lambda} p_{\xi} \restriction_{\nu_{\xi}+1}$$
Let $ C = \{ \nu_{\xi} \colon  \xi<\lambda  \} \subseteq \triangle_{\xi<\lambda} C_\xi$. Then, by our construction, $ C\subseteq \lambda $ is a club disjoint from $ \mbox{supp}\left( p ^* \right) $. Therefore, $ p^* $ is a legitimate condition in $ P_{\lambda} $. Also, given $ \alpha \in C $, let $ \xi<\lambda $ be such that $ \alpha = \nu_{\xi} $.  Then $ p^*\restriction_{\alpha+1} = p_{\xi}\restriction_{\alpha+1} $, and thus it forces that $ p^*\setminus \left(\alpha+1\right) \geq^* p_{\xi}\setminus \left( \alpha+1 \right) \in e(\alpha) $, as desired.

Now, let us adjust the proof to the case where $ \lambda $ is not inaccessible. Fix in advance an increasing, continuous and cofinal sequence $ \langle \nu_{\xi}{\xi} \colon \xi < \mbox{cf}(\lambda) \rangle $ in $ \lambda $, such that $ \nu_0> \mbox{cf}(\lambda) $. Now construct a $ \leq^* $-increasing sequence of conditions $ \langle p_{\xi} \colon \xi <\mbox{cf}(\lambda) \rangle$. In successor steps, assuming that $ p_{\xi} $ has been constructed, pick $ p_{\xi+1} $ such that--
$$ p_{\xi+1}\restriction_{\nu_{\xi+1} +1} = p_{\xi}\restriction_{ \left(\nu_{\xi+1}+1\right) } $$
and $ p_{\xi+1}\restriction_{\nu_{\xi+1}+1} \Vdash p_{\xi+1}\setminus \left(\nu_{\xi+1}+1\right) \in e\left( \nu_{\xi+1} \right) $ . In limit steps, say for limit $ \xi<\mbox{cf}(\lambda) $, choose $ p_{\xi} $ such that--
$$ p_{\xi}\restriction_{\nu_{\xi}} = \bigcup_{\eta<\xi} p_{\eta}\restriction_{\nu_{\eta}+1} $$
and $ p_{\xi}\restriction_{\nu_{\xi}} $ forces that $p_{\xi}\setminus \mu_{\xi} $ is a $ \leq^* $-upper bound of $\langle p_{\eta}\setminus \nu_{\xi} \colon \eta<\xi \rangle$ (this is the main difference from the case where $ \lambda $ is regular. Note that the direct extension order of $ P_{\lambda} \setminus\nu_{\xi} $ is more than $ \xi $-closed, since $ \nu_{\xi} > \xi $). Then, direct extend further above $ \nu_{\xi}+1 $ such that $ p_{\xi}\restriction_{\nu_{\xi}+1} \Vdash p_{\xi}\setminus \nu_{\xi}+1\in e(\nu_{\xi}) $. 

Finally, set $ p^* = \bigcup_{\xi< \mbox{cf}(\lambda)} p_{\xi}\restriction_{\nu_{\xi}+1}$.
 \end{proof}

The following claim takes care of dense open subsets of $ P_{\kappa} $ (not necessarily with respect to direct extensions).

\begin{claim} \label{Claim: NS, fusion for dense open sets}
	Let $ \lambda \leq \kappa $ be a limit ordinal and let $ D\subseteq P_{\lambda} $ be a dense open subset of $ P_{\lambda} $. Assume that $ p\in P_{\lambda} $ and $ \nu<\lambda $. Then there exists $p^* \geq^{*} p$ and a club $ C\subseteq \lambda $, such that $ p^*\restriction_{\nu} = p\restriction_{\nu} $, and, for every $p^* \leq q\in D $, 
	$${q\restriction_{\gamma+1} }^{\frown} p^* \setminus \left( \gamma+1 \right) \in D$$
	where $\gamma \in C $ is the first coordinate for which-- 
	$$ q\restriction_{\gamma+1}\Vdash \mbox{"} q\setminus \gamma \mbox{ is a direct extension of } p^*\setminus \gamma \mbox{"}$$
\end{claim}

\begin{proof}
	Fix a non-measurable $ \xi<\lambda $ and $ G_{\xi}\subseteq P_{\xi} $ generic over $ V $ such that $ p\restriction_{\xi}\in G_{\xi} $. Given $p\restriction_{\xi} \leq  q\in G_{\xi} $, we define a subset of $ P\setminus \xi $ which is $ \leq^* $-dense open above $ p\setminus \xi$:
	$$ e_q(\xi) = \{ r\in P\setminus \xi  \colon   q^{\frown} r\in D \mbox{ or }  \left( \forall r'\geq^* r,  \  q^{\frown} r' \notin D  \right)  \} $$
	
	Since $ \xi $ is non-measurable, the direct extension order of $ P\setminus \xi $ is more than $ \left| G_{\xi}\right|^{+} $-distributive. Let $ e(\xi) $ be a $ P_{\xi} $-name for the set--
	$$ e(\xi) = \bigcap_{q\in G_{\xi}}  e_{q}(\xi) $$
	then $ p\restriction_{\xi} $ forces that $ e(\xi) $ is $ \leq^* $-dense open above $ p\setminus \xi $.

	Apply lemma \ref{Lemma: NS, Fusion Lemma}. Let $ p^* \geq^* p $ be such that $ p^*\restriction_{ \nu } = p \restriction_{ \nu } $, and there exists a club $ C $ such that, for every $\alpha\in C$, 
	$$ p^* \restriction_{ \alpha+1 } \Vdash p^*\setminus \left(\alpha+1\right) \in e(\alpha) $$
	Assume now that $ p^* \leq q \in D $. Let $ \gamma\in C $ be as in the formulation of the claim. Then--
	$$ p^* \restriction_{ \gamma+1 } \Vdash p^*\setminus \left(\gamma+1\right) \in e(\gamma+1) $$
	In particular,
	$$ q \restriction_{ \gamma+1 } \Vdash p^* \setminus \left(\gamma+1\right) \in e(\gamma+1) $$
	
	Finally, since there exists a direct extension $ r' = q\setminus \left(\gamma+1\right) \geq^* p^*\setminus \left( \gamma+1\right) $ such that $ {q\restriction_{\gamma+1}}^{\frown} r' \in D $, it follows that ${q\restriction_{\gamma+1}}^{\frown}  p^* \setminus \left(\gamma+1\right)\in D$, as desired.
\end{proof}

\begin{lemma} \label{Lemma: NS, preservation of cardinals}
$P = P_{\kappa}$ preserves cardinals. It also preserves cofinalities $ \geq \kappa^{+} $.
\end{lemma}

\begin{proof}
$ P $ clearly preserves cardinals and cofinalities $ \geq \kappa^{++} $, since it has cardinality $ \kappa^{+} $. 

Let us prove by induction that every cardinal $\mu\leq \kappa^{+}$ is not collapsed. For limit $\mu$ it's clear. Suppose that $\mu = \lambda^{+}$ is a successor. 
Split $P = P_\lambda * \lusim{Q}_{\lambda} * P\setminus \left( \lambda+1 \right)$. The direct extension order of $P\setminus \left(  \lambda+1\right)$ is more than $\mu$-closed, so it preserves $\mu$. $Q_{\lambda}$ preserves cardinals, whether $\lambda$ is measurable or not. Thus, it suffices to prove that $P_\lambda$ preserves $\lambda^{+} = \mu $, for every $ \lambda\leq \kappa $. Suppose that $\lusim{f}$ is a $P_\lambda$-name for an increasing function from $\lambda$ to $\mu$, and this is forced by an arbitrary condition  $p\in P_\lambda$. We will prove that there exists an extension $p^*$ of $p$ in $P_\lambda$ which forces that the image of $\lusim{f}$ is bounded in $\mu $. 

For every $ \xi< \lambda $, define the following $P_{\xi+1}$-name for a dense open subset of $ P\setminus \left(\xi+1\right) $,
$$ e(\xi) = \{ r\in P\setminus \xi+1 \colon \exists \delta < \lambda^{+}, \ r\Vdash \lusim{f}(\xi)<\delta \} $$
We claim that $ e(\xi) $ is $ \leq^* $-dense open. First, let us argue that this suffices. Indeed, by fusion, there exists $ p^*\in G $ and a club $ C\subseteq \lambda $ such that for every $ \xi\in C $,
$$ p^*\restriction_{\xi+1}\Vdash \exists \delta_{\xi}<\lambda^{+},\ p^*\setminus \left( \xi+1 \right) \Vdash \lusim{f}(\xi)<\delta_{\xi} $$
and set-- 
$$ \delta^* = \sup\left( \bigcup_{\xi\in C} \{  \delta \colon \exists r \geq p^*\restriction_{\xi+1}, \ r\Vdash \lusim{\delta_{\xi}} = \delta \}\right) $$ Then $ \delta^* < \lambda^{+} $ and, since $ \lusim{f} $ is increasing, $ p^*\Vdash \mbox{Im}\left( \lusim{f} \right) \subseteq \delta^{*}+1 $.

Let us prove that $ e(\xi) $ is indeed $ \leq^* $-dense open. Fix $ \xi<\lambda $. Let $ G'\subseteq P_{\xi+1} $ be generic over $ V $, and work in $ V\left[G'\right] $. Denote $ P' = P\setminus \left( \xi+1 \right) $. Apply claim \ref{Claim: NS, fusion for dense open sets} for the dense open set $ D $ of conditions in $ P' $ which decide the value of $ \lusim{f}(\xi) $. Given a condition $ q\in P' $, there exists $ q^* \geq^* q $ and a club $ C\subseteq \lambda $ such that for every $ q^* \leq p\in D $, 
$$ {p\restriction_{\gamma+1}}^{\frown} q^*\setminus \left(\gamma+1\right) \in D $$
where $ \gamma$ is the least coordinate in $ C $ above the non-direct extensions. Let--
$$\delta^* = \sup\left(  \bigcup_{\gamma\in C}  \{ \delta \colon \exists s\in P'_{\gamma+1}, \ s^{\frown} q^*\setminus \left(\gamma+1 \right) \Vdash \lusim{f}(\xi) = \delta \} \right)$$
Then $ q^*\Vdash \lusim{f}(\xi)< \delta^* $.
\end{proof}

The following lemma is a minor modification of lemma 3.6 from \cite{ben2017homogeneous}.

\begin{lemma} \label{Lemma: NS, every new function is evaluated by F in V}
Let $ \lambda \leq \kappa $ be inaccessible. Let $ p\in P_{\lambda} $ and assume that $ \lusim{f}$ is a $ P_{\lambda} $-name for a function from $ \lambda $ to the ordinals. Then there exists $ p^*\geq^* p $, a club $ C\subseteq \lambda $ and a function $ F\colon \lambda \to \left[\mbox{Ord}\right]^{<\lambda} $ in $ V $, such that for every $ \xi\in C $, $ p^*\Vdash \lusim{f}(\xi)\in F(\xi) $. 

\end{lemma}

\begin{proof}
For each $ \xi < \lambda $, consider the $ P_{\xi+1} $-name for the following set--
$$ e(\xi) = \{ r\in P\setminus \xi \colon \exists A\in \left[\mbox{Ord}\right]^{<\lambda}, \ r\Vdash \lusim{f}(\xi)\in A  \} $$
It suffices to prove that for every $ \xi<\lambda $, $ e(\xi) $ is forced to be $ \leq^* $-dense open subset of $ P\setminus \left(\xi+1\right) $. Indeed, once we prove this, there exists $ p^*\in G $ above $ p $ and a club $ C\subseteq \lambda $ such that for every $ \xi\in C $,
$$ p^*\restriction_{\xi+1}\Vdash \exists A_{\xi}\in \left[ \mbox{Ord}\right]^{<\lambda}, \  p^*\setminus \left(  \xi+1\right)\Vdash \lusim{f}(\xi)\in A(\xi) $$
and then, for every $ \xi\in C $, we can define--
$$ F(\xi) = \{  \gamma \colon \exists q\geq p^*\restriction_{\xi+1}, \ q\Vdash \gamma \in \lusim{A}_{\xi} \} $$
Then $ \left|F(\xi)\right|<\lambda $ for every $ \xi\in C $, and $ p^*\Vdash \lusim{f}(\xi)\in F(\xi) $.

Let us prove that $ e(\xi) $ is $ \leq^* $ dense open.  Fix $ \xi<\kappa $. Let $ G'\subseteq P_{\xi+1} $ be generic over $ V $, and work in $ V\left[G'\right] $. Denote $ P' = P\setminus \left(\xi+1\right) $. It suffices to prove that given a condition $ q\in P' $, there exists a direct extension $ q^*\geq^* q $ and a set $ A \in \left[\mbox{Ord}\right]^{<\lambda}$ such that $ q^*\Vdash \lusim{f}(\xi)\in A $. 

Let $ D\subseteq P' $ be the dense open set of conditions $ r\in P' $ such that, for some $ A\in \left[Ord\right]^{<\lambda} $, $ r\Vdash \lusim{f}(\xi) \in A $. By claim \ref{Claim: NS, fusion for dense open sets}, there exists $ q^* \geq^*q $ and a club $ C\subseteq D $, such that for every $q^*\leq p\in D$, $ {p\restriction_{\gamma'+1}}^{\frown} q^*\setminus \left( \gamma'+1 \right)\in D $, where $ \gamma' = \min\left( C\setminus \left( \gamma+1\right) \right) $, and $ \gamma  $ is the maximal coordinate in which a non-direct extension is taken in the extension $ q^*\leq p $. 

%
Let us construct a direct extension $ q^{**}\geq^*  q^*$ with the same support as $ q^* $. Let $ \mu\in \mbox{supp}\left(q^*\right) $ be a measurable, and assume that $ q^{**}\restriction_{\mu} $ was constructed. Take an arbitrary generic $ G_{\mu}\subseteq P'_{\mu} $ with $ q^{**}\restriction_{\mu}\in G_\mu $. Denote $ \mu' = \min \left(C\setminus \left(\mu+1\right)\right) $. In $ V\left[G',G_{\mu}\right] $, shrink the set $ \lusim{A}^{q^*}_{\mu} $ to a set $ A $ such that, for each $ n<\omega $, exactly one of the following holds: Either for every $ s\in \left[A\right]^n $, there exists  direct extension $ r_{s} \geq^* q^*\restriction_{\left( \mu, \mu' \right]} $ and a set of ordinals $ A_{s} $ with $ \left|A_s\right|<\lambda $, such that--
$${ \langle {t^{q^*}_{\mu}}^{\frown} s  , A\setminus \max(s) \rangle}^{\frown} {r_s} ^{\frown} {q^*}\setminus \left(\mu'+1\right)  \Vdash  \lusim{f}(\xi) \in  \check{A}_s$$
or, there is no such $ s\in \left[A\right]^n $.

Let us prove now that $ q^{**} $ has a direct extension which belongs to $ e(\xi) $. Assume otherwise. Let $ p\geq q^{**} $ be a condition which decides the value of $ \lusim{f}(\xi) $, and is chosen with the least number of non--direct extensions. Let $ \gamma \in \mbox{supp}(q^*) $ be the maximal coordinate in which a non-direct extension is taken, and let $ \gamma' = \min\left( C\setminus \left(\gamma+1\right) \right) $. Clearly $p\geq q^*$, and by the choice of $ q^* $, 
$$ {p\restriction_{\gamma'+1}}^{\frown} q^*\setminus \left( \gamma'+1 \right)\in D $$
In particular, for some $ A\in \left[\mbox{Ord}\right]^{<\lambda} $, 
$$ p\restriction_{\gamma} \Vdash  p\restriction_{\left[\gamma, \gamma'\right]}\Vdash  q^{*}\setminus \left( \gamma'+1 \right)\Vdash \lusim{f}(\xi)\in A $$
Now, let $ G_{\gamma}\subseteq P'_{\gamma} $ be generic over $ V\left[G'\right] $ with $ p\restriction_{\gamma}\in G_{\gamma} $. Then in $ V\left[G',G_{\gamma}\right] $, there exists $ A\in \left[\mbox{Ord}\right]^{<\lambda} $ such that--
$$  {\langle t^{p}_{\gamma}, A^{p}_{\gamma} \rangle}^{\frown} {p\restriction_{\left(\gamma, \gamma'\right]}}^{\frown} {q^{*}\setminus \left( \gamma'+1 \right)} \Vdash \lusim{f}(\xi)\in A $$
Let $ n<\omega $ be such that  $\mbox{lh} \left( t^{p}_{\gamma} \right) = n +  \mbox{lh} \left( t^{q^*}_{\gamma} \right) $. Then $ p\restriction_{\gamma} $ extends $ q^{**}\restriction_{\gamma} $, and thus forces that for every $ s\in \left[ \lusim{A}^{p}_{\gamma} \right]^n $, there exists $ r_s \geq^* q^*\restriction_{\left( \gamma,\gamma' \right]} $ and a set $ A_{s} $ bounded in $\lambda$, such that--
$$  { \langle  { t^{q^*}_{\gamma} }^{\frown} s ,  \lusim{A}^{p}_{\gamma} \setminus \max(s)
	\rangle } ^{\frown} { r_s } ^{\frown} q^*\setminus \left(\gamma'+1\right) \Vdash \lusim{f}(\xi)\in A_{s}  $$
Let $ r $ be a $ P_{\gamma+1} $-name for the direct extension of $ q^* $ which is forced by-- 
$$ { \langle  { t^{q^*}_{\gamma} }^{\frown} s ,  \lusim{A}^{p}_{\gamma} \setminus \max(s)
	\rangle }  $$ 
to be $ r_s $, for every $ s $ of length $ n $. Then $ r\geq^* q^*\restriction_{\left( \gamma, \gamma' \right]} $, and by direct extending $ r $ inside the support of $ q^* $, we can assume that $ r\geq^* q^{**}\restriction_{\left( \gamma, \gamma' \right]} $ (note that the coordinates in which a non-direct extension is taken in the extension $ r\geq^* q^*\restriction_{\left( \gamma, \gamma' \right]} $ does not lie inside $ \mbox{supp}\left( q^* \right) $).

By taking a union of the sets $ A_{s} $ above, there exists a set of ordinals $ A\in V\left[G',G_{\mu}\right] $ with $ \left|A\right|<\lambda $ such that--
$$ { \langle  { t^{q^*}_{\gamma} } ,  \lusim{A}^{p}_{\gamma} 
	\rangle }^{\frown} r ^{\frown} q^*\setminus \left(\gamma+1\right) \Vdash \lusim{f}(\xi)\in A $$
$ G_{\gamma} $ was an arbitrary generic set with $ p\restriction_{\gamma}\in G_{\gamma} $; thus, in $ V\left[G'\right] $,
$$ p\restriction_{\gamma} \Vdash \exists A\in \left[Ord\right]^{<\lambda}, \ { \langle  { t^{q^*}_{\gamma} } ,  \lusim{A}^{p}_{\gamma} 
	\rangle }^{\frown} r ^{\frown} q^*\setminus \left(\gamma+1\right) \Vdash \lusim{f}(\xi)\in A $$
Let $\lusim{A} $ be a $ P'_{\gamma} $-name for the above set $ A $, and let $ A^*\in V\left[G'\right] $ be the set of all possible values of elements in $ \lusim{A} $ as forced by extensions of $ p\restriction_{\gamma} $. Then $ A^*\in \left[\mbox{Ord}\right]^{<\lambda} $, and--
$$  {p\restriction_{\gamma}}^ {\frown} {\langle  t^{q^*}_{\gamma}  ,  \lusim{A}^{p}_{\gamma} 
\rangle }^{\frown} r ^{\frown} q^{**}\setminus \left(\gamma+1\right) \Vdash \lusim{f}(\xi)\in A^* $$
This contradicts the minimality of the number of non-direct extensions in the choice of $ p\geq q^{**} $.
\end{proof}

Let us mention several immediate corollaries of the last lemma, all of them were introduced in \cite{ben2017homogeneous}:

\begin{corollary} \label{Corollary: NS, every name for an ordinal is decided by a direct extension up to boundedly many values}
	Let $ \lambda \leq \kappa $ be a regular cardinal and $ p\in P_{\lambda} $. Assume that $ \lusim{\alpha} $ is a $ P_{\lambda} $-name for an ordinal. Then there exist $ p^*\geq^* p $ and a set of ordinals $ A $ of cardinality $ \left|A\right|< \lambda $, such that $ p^*\Vdash \lusim{\alpha}\in \check{A} $.
\end{corollary}

\begin{proof}
	If $ \lambda $ is a limit of measurables, then it is inaccessible, and then the proof is included in the proof of lemma \ref{Lemma: NS, every new function is  evaluated by F in V}. Else, let $ \lambda' < \lambda $ be the supremum of the set of measurables below $ \lambda $. Then $ P_{\lambda} = P_{\lambda'} $. We can now repeat the argument in the proof of lemma \ref{Lemma: NS, every new function is evaluated by F in V} for the forcing $ P_{\lambda'} $, with minor changes: first define $ D = \{  r\in P_{\lambda'} \colon \exists A\in \left[\mbox{Ord}\right]^{<\lambda} \mbox{ such that } r\Vdash \lusim{\alpha}\in A \} $. Direct extend $ p^*\geq^* p $ and find a club $ C\subseteq \mu' $ such that for every $p^*\leq q \in D$, $ {q\restriction_{\gamma'+1}}^{\frown} {p^*\setminus \left( \gamma'+1 \right)} \in D $, where $ \gamma'\in C $ is a above the finite set of non-direct extensions taken in the extension $ q\geq p^* $. Then, direct extend $ p^{**}\geq^* p^* $, without changing the support, as in the previous lemma. Arguing as above, $ p^{**} $ has a direct extension which decides $ \lusim{\alpha} $ up to $ <\lambda $-many possibilities.
	
	We remark that if $ \lambda > {\lambda'}^{+} $, a simpler argument exists: by $ \mbox{GCH} $,  $ P_{\lambda'} $ is $ \lambda-c.c. $. let $ A\in V $ be the set--
	$$ A = \{ \xi  \colon \exists q\geq p, \ q\Vdash \xi = \lusim{\alpha} \} $$
	then $ \left|A\right|<\lambda $ and $ p\Vdash \lusim{\alpha}\in A $ (here a direct extension of $ p $ is not required).
\end{proof}

\begin{corollary} \label{Corollary: NS, every new function is dominated by an old one}
Let $ \lambda \leq \kappa $ be inaccessible, and assume that $ G_{\lambda}\subseteq P_{\lambda} $ is generic over $ V $. Then $ \lambda $ is still regular iv $ V\left[G_{\lambda}\right] $. Moreover, every function $ f\colon \lambda \to \lambda  $ in $ V\left[G_{\lambda}\right] $ is dominated by a function $ g\colon \lambda \to \lambda  $ in $ V $. 
\end{corollary}

\begin{proof}
Assume that $ \lambda $ is singular in $ V\left[G_{\lambda}\right] $. Let $ \mu = \mbox{cf}\left( \lambda \right) $. Let $ f\colon \mu \to \lambda $ be an increasing cofinal sequence in $ V\left[G_{\lambda}\right] $. Let $p\in  P_{\lambda} $ be a condition which forces this. We argue that there exists $ \delta< \lambda $ and $ p^*\geq p $ such that $ p^*\Vdash \mbox{Im}(f)\subseteq \delta $, which is a contradiction. Assume without loss of generality that $ p $ is the weakest condition in $ P_{\lambda} $.

Work in an arbitrary generic extension of $ V $ with the forcing $ P_{\mu+1} $. We argue that every condition $ q\in P\setminus \left(\mu+1\right) $ has a direct extension $ q^*\in P\setminus \left( \mu+1 \right) $ and function $ \alpha\mapsto F(\alpha)  $ such that for every $ \xi<\mu $, $ F(\xi) $ is a bounded subset of $ \lambda  $, and--
$$q^*\Vdash \lusim{f}(\xi)\in F(\xi) $$ 
Indeed, given $ \xi<\mu $ , $ \lusim{f}(\xi) $ is a $ P\setminus \left( \mu+1 \right) $-name for an ordinal below $ \lambda $. By corollary \ref{Corollary: NS, every name for an ordinal is decided by a direct extension up to boundedly many values}, every $ q\in P\setminus \left( \mu+1 \right)$ can be direct extended to $ q^* \in P\setminus \left( \mu+1 \right) $ such that for some set of ordinals $ A_{\xi}\subseteq \lambda$ with $ \left| A_{\xi} \right|<\lambda $, $ q^*\Vdash \lusim{f}(\xi) \in A_{\xi} $. Since the direct extension order of $ P\setminus \left(\mu+1\right) $ is more then $ \mu $-closed, we can find a single $ q^*\in P\setminus \left( \mu+1 \right) $, and, for every $ \xi<\mu $, a bounded subset $ A_{\xi}\subseteq \lambda $ such that $ q^* \Vdash \forall \xi< \mu, \ \lusim{f}(\xi) \in A_{\xi} $; then, set $ F(\xi) = A_{\xi} $ as desired.

Since we worked in an arbitrary generic extension above $ \left(\mu+1\right) $ and gave a density argument in $ P\setminus \left( \mu+1 \right) $, we can assume that there exists $ p^*\in G $ such that--
\begin{align*}
p^*\restriction_{\mu+1} \Vdash & \mbox{ there exists a function } \xi \mapsto F(\xi) \mbox{ such that, for every  } \xi< \mu, \\
&\ F(\xi) \mbox{ is a bounded subset of } \lambda  \mbox{ and } p^*\setminus \left(\mu+1\right)\Vdash \lusim{f}(\xi)\in F(\xi)
\end{align*}
Finally, define, in $V$, 
$$ \delta = \sup \left( \bigcup_{\xi<\mu}   \{ \beta< \lambda \colon \exists q\geq p^*\restriction_{\mu+1}, \ q\Vdash \check{\beta}\in \lusim{F}(\xi) \} \right)$$
and note that $ \delta < \lambda $ and $ p^*\Vdash \mbox{Im}\left(\lusim{f}\right) \subseteq \delta $. 

Let us argue now that every function $ f\colon \lambda\to \lambda $ in $ V\left[G\right] $ is dominated by a function $ g\colon \lambda\to \lambda $ in $ V $. First, in $ V\left[G\right] $, $ f $ is dominated by an increasing function $ f'\colon \lambda\to \lambda $. By \ref{Lemma: NS, every new function is evaluated by F in V}, $ f' $ is dominated on a club $ C\subseteq \lambda $ by a function $ g'\colon \lambda\to \lambda  $ in $ V $. Given $ \xi<\kappa $, let $ c_\xi = \min\left( C\setminus \xi+1 \right) $. Finally, define $ g\colon \lambda\to \lambda $, 
$$ g(\xi) = g'\left( c_{\xi} \right) $$
Then for every $ \xi<\kappa $, 
$ f(\xi) \leq f'(\xi) \leq f'\left( c_\xi \right)  < g'\left( c_\xi \right) = g(\xi)$.
\end{proof}

\begin{corollary}
Let $ \lambda \leq \kappa $ be inaccessible. The forcing $ P_{\lambda} $ preserves stationary subsets of $ \lambda $.
\end{corollary}

\begin{proof}
It suffices to prove that for every club in $ \kappa $, $ C\in  V\left[G\right] $, there exists a club in $ \kappa $, $ D\in V $, such that $ D\subseteq C $. In $ V\left[G\right] $, let $ f\colon \kappa\to \kappa $ be the increasing enumeration of $ C $. By corollary \ref{Corollary: NS, every new function is dominated by an old one}, there exists $ g\in V $ which dominates $ f $. Let $ D $ be the set of closure points of $ g $. Clearly, $ D $ is a club. Let us prove that $ D\subseteq C $. Given $ \alpha\in D $, $ \alpha $ is a closure point of $ f $, and thus a limit point of $ \mbox{Im}(f) =C $. Therefore $ \alpha\in C $.
\end{proof}

Recall that a set of ordinals $ A\in V\left[G\right] $ is called fresh if $ A\notin V$ and, for every ordinal $ \xi < \sup(A) $, $ A \cap \xi \in V $. Every old measurable $ \mu < \kappa $ clearly has a fresh unbounded subset: its Prikry sequence. So if $ \sup(A) $ was a measurable cardinal below $ \kappa $ in $ V $, $ A $ might be  fresh over $ V $. Let us address the case where $ \sup(A)$ is $\kappa$ or $ \kappa^{+} $.

\begin{lemma} \label{Lemma: NS, no fresh subsets of kappa and kappa+}
$ P=P_{\kappa} $ does not add new unbounded subsets of $ \kappa $ or $ \kappa^{+} $ which are fresh over $ V $.
\end{lemma}

The proof appears in \cite{RestElm}. Having no fresh subsets of $ \kappa, \kappa^{+} $, together with preservation of cardinals and $ 2^{\kappa} = \kappa^{+} $, leads to the following key property:

\begin{corollary} \label{Corollary: NS, every measure in generic extension extends a measure from V}
Let  $ W\in V\left[G\right] $ be a $ \kappa $-complete ultrafilter on $ \kappa $. Then $ W\cap V \in V $.
\end{corollary}

For the proof, see proposition $2.1$ in \cite{RestElm}.

\begin{corollary} \label{Corollary: NS, If W is normal in V[G], then W cap V has Mitchell order 0 in V}
	Let $ W\in V\left[G\right] $ be a normal measure. Then $ W\cap V\in V $ is a normal measure of Mitchell order $ 0 $ in $ V $.
\end{corollary}

\begin{proof}
	Denote $ U = W\cap V $.  By corollary \ref{Corollary: NS, every measure in generic extension extends a measure from V}, $ U \in V $. $ U $ inherits normality from $ W $, since it is closed under diagonal intersections. Finally, let us prove that $ U $ has Mitchell order $ 0 $. Assume  otherwise. Then $ U $ concentrates on the set $ \Delta $ of measurables below $ \kappa $ in $ V $. Hence, $\Delta \in  W $. However, in $ V\left[G\right] $, each cardinal in $ \Delta $ is singular and has cofinality $ \omega $, and by normality of $W$, it cannot concentrate on $ \Delta $.
\end{proof}

\section{Normal Measures in the Generic Extension}

Our goal in this section is to prove that there exists a bijection between  normal measures of Mitchell order $ 0 $ in $ V $, and  normal measures on $ \kappa $ in $ V\left[G\right] $. Let $ U\in V $ be a normal measure of Mitchell order $ 0 $. We will define a normal measure $ U^*\in V\left[G\right] $ which extends $ U $. We will prove the following:

\begin{thm} \label{Theorem: NS, characterization of normal measures in the generic extension}
Every normal measure $ W\in V\left[G\right] $ on $ \kappa $ has the form $ U^* $ for some normal measure $ U \in V $ of Mitchell order $ \ 0 $. Furthermore, $ U^* $ is the unique normal measure in $ V\left[G\right] $ which extends $ U $. 
\end{thm}

Let $ U \in V$ be any normal measure on $ \kappa $ of Mitchell order $ 0 $. After forcing an iteration of Prikry forcings, with any standard support, one can define, in the generic extension $V\left[G\right]$, a natural filter which extends $ U $: The filter consisting of sets $ \left(\lusim{A}\right)_G $, where $ \lusim{A} $ is a name for subset of $ \kappa $, such that, for some $ p\in G $,
$$  \{ \alpha<\kappa \colon p\Vdash \check{\alpha}\in A \} \in U $$
or simply $j_{U}(p)\Vdash \check{\alpha}\in j_{U}\left( \lusim{A} \right) $, in $ M_{U} $. 

Forcing with nonstationary support has the advantage, that this filter is actually a normal, $\kappa$-complete ultrafilter.

\begin{lemma} \label{Lemma: NS, extending ultrafilters of Mitchell order 0}
Let $ U $ be a normal measure of Mitchell order $ 0 $ on $ \kappa $. Define $ U^*\in V\left[G\right] $ as follows: For every $ P_{\kappa} $-name $ \lusim{A} $ for a subset of $ \kappa $, $ \left(\lusim{A}\right)_G \in U^* $ if and only if there exists $ p\in G $ such that $ j_U(p) \Vdash \check{\kappa} \in j_U\left( \lusim{A} \right)
 $. Then $  U^{*} $ is a normal, $ \kappa $-complete ultrafilter in $ V\left[G\right] $, which extends $ U$.
\end{lemma}

\begin{proof}
Let us check first that $ U^* $ is well defined. Assume that $ \lusim{A} , \lusim{B}$ are $ P$-names for subsets of $ \kappa $, and $ p\in G $ is a condition such that $ p\Vdash \lusim{A} = \lusim{B} $. Then $ j_{U}(p)\Vdash j_{U}(\lusim{A}) = j_{U}(\lusim{B})$, and thus  $ j_{U}(p)\Vdash \check{\kappa} \in j_{U}(\lusim{A}) $ if and only if $ j_{U}(p)\Vdash \check{\kappa} \in j_{U}(\lusim{B}) $, so $ \left( \lusim{A} \right)_{G} \in U^* $ if and only if $ \left( \lusim{B} \right)_{G} \in U^*  $, as desired.

It's not hard to verify that $ U^{*} $ is a filter. Let us prove that it's $ \kappa $-complete (thus, in particular, it's an ultrafilter). Assume that $ \gamma<\kappa  $ and $ \langle \lusim{A}_\beta \colon \beta<\gamma \rangle $ is forced by the weakest condition to be a partition of $ \kappa $.

For every $ \alpha \in \left( \gamma,\kappa \right) $, let $ e(\alpha)\subseteq P\setminus \left(\alpha+1\right) $ be the following $ \leq^* $-dense open subset:
$$ e(\alpha) = \{ r\in P\setminus \left(\alpha+1\right) \colon \exists \beta^*<\gamma \  \ r\Vdash \check{\alpha} \in \lusim{A}_{\beta^*} \} $$
By lemma \ref{Lemma: NS, Fusion Lemma}, there exists $ p\in G $ and a club $ C\subseteq \kappa $ such that for every $ \alpha\in C $, 
$$ p\restriction_{ \alpha+1} \Vdash p\setminus \alpha+1 \in e(\alpha)$$
$ C $ is a club, so $ C\in U $, and thus,
$$ p^{\frown} 0_{\lusim{Q}_{\alpha}} \Vdash  \exists \beta^*<\gamma  \ \ \ j_{U}(p)\setminus \left(\kappa+1\right)  \Vdash \check{\kappa} \in j_{U}(\lusim{A}_{\beta^*})  $$
therefore, for some $ q>p $, $ q\in G $ and $ \beta^*<\gamma $,
$$ q \Vdash j_{U}(p)\setminus \kappa  \Vdash \check{\kappa} \in j_{U}(\lusim{A}_{\beta^*})  $$
so $ j_{U}(q) \Vdash \check{\kappa} \in j_{U}(\lusim{A}_{\beta^*})   $.

Let us prove normality. Assume that $ \lusim{f} $ is a name for a regressive function from $ \kappa $ to $ \kappa $. Work in $ M_{U} $. $ j_{U}(\lusim{f}) $ is forced there to be a regressive function. There exists a $ P_{\kappa+1} $-name for a dense open subset $ D $ of $ j_{U}(P)\setminus \left( \kappa+1 \right) $, consisting of all the conditions which force that $ j_{U}(\lusim{f})(\kappa) =\beta^* $ for some $ \beta^*<\kappa $. Let $ \alpha\mapsto d(\alpha) $ be a function in $ V $ which represents $ D $ in the ultrapower construction. We can assume that for every $ \alpha<\kappa $, $ e(\alpha) $ is a $ P_{\alpha+1} $-name, forced by the weakest condition to be a $ \leq^* $ dense-open subset of $ P_{\kappa}\setminus \left(\alpha+1\right) $. Now we apply fusion just as before, and find $ p\in G $ such that--
$$ p^{\frown} 0_{\lusim{Q}_{\kappa}} \Vdash  \exists \beta^*<\kappa  \  j_{U}(p)\setminus \left( \kappa+1 \right)  \Vdash  j_{U}(\lusim{f})(\kappa) =\beta^*   $$
so for some $ q\in G $, $ q>p $, and for some $ \beta^*<\kappa $,
$$ j_{U}(q)\Vdash  j_{U}(\lusim{f})(\kappa) =\beta^*   $$
Therefore, $ \{ \xi<\kappa \colon f(\xi)=\beta^* \} \in U^{*}$ as desired.\end{proof}

Now, given a normal measure $ W\in V\left[G\right] $  on $ \kappa $, denote $ U = W\cap V $. By corollary \ref{Corollary: NS, every measure in generic extension extends a measure from V}, $ U\in V $. Our goal will be to prove that $ W = U^* $. We start with the following observation:

\begin{lemma} \label{Lemma: NS, k}
Let $ W\in V\left[G\right] $ be a normal measure on $ \kappa $. Let $ j_{W}\colon V\left[G\right] \to M\left[H\right]$ be the ultrapower embedding. Denote  $ U = W\cap V $, and define $ k\colon M_{U} \to M $ as follows: $ k\left( \left[f\right]_{U} \right) = \left[ f\right]_W $. Then:
\begin{enumerate}
\item $ k \colon M_{U} \to M $ is an elementary embedding.
\item $ k\circ j_{U}  = j_{W} \restriction_{V}$.
\item $ k\restriction_{ \mu } = id $, where $ \mu $ is the first measurable above $ \kappa $ in $ M_{U} $. In particular, for every $ \eta < \mu $, there exists $ f\in V $ such that $\left[f\right]_{U} = \left[f\right]_W = \eta$. 
\item $\mbox{crit}(k) = \mu$.
\item $\left(  V_{\mu} \right)^M = \left( V_{\mu} \right)^{M_{U}}$.
\end{enumerate}

\end{lemma}

\begin{proof}
\begin{enumerate}
\item $ k $ is well defined, since, if $ \left[f\right]_{U} = \left[g\right]_{U} $, then $ \{ x<\kappa \colon f(x)=g(x) \} \in U $, and thus this set belongs to $ W $. So $ \left[f\right]_W = \left[g\right]_W $. Similarly, $ k $ respects $ \in $. Finally, assume that $ \varphi(a_1,\ldots , a_n) $ is a formula and $ f_1,\ldots ,f_n $ are functions. Then:
\begin{align*}
M_{U}\vDash \varphi\left( \left[f_1\right]_{U},\ldots \left[f_n\right]_{U} \right)  & \iff \{ x<\kappa \colon V\vDash \varphi\left( f_1(x),\ldots ,f_n(x \right) \}\in U \\
& \iff \{ x<\kappa \colon V\vDash \varphi\left( f_1(x),\ldots ,f_n(x \right) \}\in W \\
& \iff M\left[H\right] \vDash M\vDash \varphi\left( \left[f_1\right]_{W},\ldots \left[f_n\right]_{W} \right) \\
& \iff M\vDash \varphi\left( \left[f_1\right]_{W},\ldots \left[f_n\right]_{W} \right)
\end{align*}
\item Clear from the definitions.
\item First, let us note that for every $ \eta \leq \kappa^{+} $, $ k\left(\eta\right) = \eta $, using the canonical function which represents $ \eta $. Also, $ k\left( \kappa^{+} \right) = \kappa^{+} $ since $ \kappa^{+} $ is represented by the successor cardinal function. Thus, $ \mbox{crit}(k)> \kappa^{+} $.

Now, Assume, for contradiction, that there exists $ \eta<\mu $ such that $ k(\eta) > \eta $. Take the minimal such $ \eta $. There exists $ g\in V\left[G\right] $ such that $ \left[g\right]_W = \eta $. Let $ h\colon \kappa \to \kappa $ be a function in $ V $ such that $ \left[h\right]_{U} = \eta $. So--
$$ \left[g\right]_W = \eta = \left[h\right]_{U} \leq \left[h\right]_W $$
and thus, by changing $ g $ on a set which doesn't belong to  $ W $, we can assume that, for every $ \xi<\kappa $, $$ g(\xi)\leq h(\xi) < \mbox{ the first measurable above }\xi $$
For every $ \xi<\kappa $, let $ e(\xi) $ be the $ P_{\xi+1} $-name for the following set--
$$ e(\xi) = \{  r\in P\setminus \left(\xi+1\right)  \colon  \exists \alpha\leq h(\xi), \  r\Vdash g(\xi)=\check{\alpha} \}  $$
this set is $\leq^*$-dense open, since the direct extension order is more than $ h(\xi) $-closed. Apply fusion. There exits $ p^* \geq^* p $ and a club $ C\subseteq \kappa $ such that, for every $ \xi\in C $, 
$$ p^* \restriction_{ \xi+1 } \Vdash p^*\setminus \left( \xi+1 \right) \in e(\xi) $$
in other words,
$$ p^* \restriction_{ \xi+1 } \Vdash  \exists \alpha\leq h(\xi), \  p^*\setminus \left( \xi+1 \right) \Vdash g(\xi)=\check{\alpha} $$

Define $ F\colon C \to V $ as follows: For every $ \xi\in C $, let $ \lusim{\alpha} $ be a $ P_{\xi+1} $-name for $ \alpha $ as above, and set--
$$ F(\xi) = \{  \alpha  \colon \exists a\in P_{\xi+1} \ \  a\geq p^* \restriction_{ \xi+1 } \mbox{ and } a^{\frown} p^*\setminus \left( \xi+1 \right) \Vdash g(\xi) = \check{\alpha}\}$$
Note that for every $ \xi\in C $,  $ p^* \Vdash g(\xi)\in \check{F}(\xi) $, and $ \left|F(\xi)\right|\leq \left|\xi\right|^{+} $. Also, $ C\in U $ and thus $ C\in W $. Therefore, in $ M\left[H\right] $,
$$ \eta = \left[g\right]_W \in \left[F\right]_W = j_{W}(F)(\kappa) = k\left( j_{U}(F)(\kappa) \right) $$
but, in $ M_{U} $, $ \left| j_{U}(F)(\kappa) \right| \leq \kappa^{+} $, which is strictly below the critical point of $ k $. So $ \eta \in \mbox{Im}(k) $, i.e., for some $ \alpha\leq \eta $, $ \eta = k(\alpha) $. But $ \eta $ was the minimal such that $ k(\eta)\neq \eta $, so $ \alpha<\eta $ and $ \alpha = k(\alpha)= \eta $, a contradiction.
\item It suffices to prove that $ k(\mu) \neq \mu $. Since $ \mu $ is measurable in $ M_{U} $, it suffices to prove that $ \mu $ is not measurable in $ M $. Assume otherwise. Then in $ V\left[G\right] $, $ \mbox{cf}(\mu) = \omega $ would hold. Therefore, in $ V $, $ \mbox{cf}(\mu)\leq \kappa $. By closure under $ \kappa $-sequences, this is true in $ M_{U} $ as well, a contradiction.
\item It suffices to prove that, for every $ \alpha<\mu $, $ \left(V_{\alpha} \right)^M =  \left(V_{\alpha} \right)^{M_{U}} $. Indeed, 
$$  \left(V_{\alpha} \right)^M  =  k\left( \left(V_{\alpha} \right)^{M_{U}} \right) = k'' \left( \left(V_{\alpha} \right)^{M_{U}} \right) = \left(V_{\alpha} \right)^{M_{U}}  $$
\end{enumerate}
\end{proof}

We now have all the tools necessary for the proof of theorem \ref{Theorem: NS, characterization of normal measures in the generic extension}

\begin{proof}[Proof of theorem \ref{Theorem: NS, characterization of normal measures in the generic extension}]
Assume that $ W\in V\left[G\right] $ is a normal measure. Denote $ U = W\cap V $. Let us use the notations of lemma \ref{Lemma: NS, k}: Assume that $ j_W \colon V\left[G\right] \to M\left[H\right] $ is the ultrapower embedding of $ W $, and let  $ k\colon M_{U} \to M $ be such that $ j_W\restriction_{V} = k\circ j_U $. 

Let us prove that $ W = U^* $. Since both are ultrafilters on $ \kappa $, it suffices to prove that $ U^* \subseteq W $. Assume that $ X\in W $. Let $ \lusim{X}\in V $ be a $ P_{\kappa} $-name such that $ \left( \lusim{X} \right)_G = X $. There exists $ p\in G $ such that $ j_U(p) \Vdash \check{\kappa}\in j_U\left( \lusim{X} \right) $. By applying $ k $,
$$ j_W(p) \Vdash \check{\kappa} \in j_W\restriction_{V}\left( \lusim{X} \right) $$
since $ \mbox{crit}(k) >\kappa$. But $ j_W(p)\in j_W(G) =  H $. Hence, in $ M\left[H\right] $, 
$$ \kappa \in \left(  j_{W}\restriction_{V}\left( \lusim{X} \right) \right)_H = j_W \left(  \left( \lusim{X} \right)_G \right) = j_W\left( X\right) $$
so $ X\in W $.
\end{proof}

\section{The Structure of $ j_W\restriction_{V} $}

As usual, let $ W\in V\left[G\right] $ be a normal measure, and denote $ U = W\cap V $. Let $ \kappa^* = j_{U}(\kappa) $. Given $ \alpha<\kappa $, recall that $ \lusim{U}^*_{\alpha} $ is a $ P_{\alpha} $-name, forced by the weakest condition in $ P_{\alpha} $ to be the normal measure on $ \alpha $ used in the Prikry forcing $ \lusim{Q}_{\alpha} $. Let $ \lusim{ \mathcal{U}} = \langle \lusim{U}_{\alpha} \colon \alpha\in \Delta \rangle $ be the sequence of names, such that, for every $ \alpha\in \Delta $, $ \lusim{U}_{\alpha} $ is forced by the weakest condition in $ P_{\alpha} $ to be $ \lusim{U}^{*}_{\alpha} \cap V $. Given $ G\subseteq P_{\kappa} $ generic over $ V $, let $ \mathcal{U} = \langle U_{\alpha} \colon \alpha\in \Delta \rangle $ be the interpretation of the names in $ \lusim{\mathcal{U}} $ with respect to the generic $ G $.

Our goal in this section is to factor $ j_W\restriction_{V} $ to an iterated ultrapower of $ V $, while revealing, simultaneously, more and more information about the generic set $ H = j_W(G) $.

By induction, we define for every $ \alpha < \kappa^*$ a model $ M_{\alpha} $, an embedding $ j_{\alpha} \colon V\to M_{\alpha} $, a measurable cardinal $ \mu_{\alpha} $ in $ M_{\alpha} $ and a measure $ U_{\mu_{\alpha}}\in M_{\alpha} $ on it. The definition goes by induction on $ \alpha<\kappa^* $, such that the sequence of models $ \langle M_{\alpha} \colon \alpha<\kappa^* \rangle $ is a linear iterated ultrapower of $ V $ with direct limit $ M_{\kappa^*} $. The iteration is continuous, namely, for every $ \alpha \leq \kappa^* $  limit, $ M_{\alpha} $ is the direct limit of the models $ \langle M_{\alpha'} \colon \alpha'<\alpha \rangle $. 

Given $ \alpha<\kappa^* $, we define $ \mu_{\alpha} $ to be the least measurable $ \mu $ in $ M_{\alpha} $, such that for every $ \alpha'<\alpha $, $ \mu_{\alpha'}< \mu $, and such that $ \left( \mbox{cf}(\mu) \right)^V > \kappa $. We will define a measure $ U_{\mu_{\alpha}} \in M_{\alpha} $ on $ \mu_{\alpha} $. We postpone the definition of $ U_{\mu_{\alpha} } $, but mention only that it will have Mitchell order $ 0 $. After $ U_{\mu_{\alpha}} $ is\ defined, we take $ M_{\alpha+1} =  \mbox{Ult}\left( M_{\alpha} , U_{\mu_\alpha} \right)$ and $ j_{\alpha+1} = \left(  j_{ U_{\mu_\alpha} } \right)^{M_{\alpha}} \circ j_{\alpha} $. 

Our goal in this section will be to prove the following:

\begin{thm} \label{Theorem: NS, full description of j_W restricted to V}
	$ M = M_{\kappa^*} $,  $ j_W\restriction_{V} = j_{\kappa^*} $ and $ \kappa^* = j_W(\kappa) $. If $ \mathcal{U} \in V $, then both $ M $ and $ j_W\restriction_{V} $ are definable classes of $ V $.
\end{thm}

\begin{remark}
Given $ \alpha<\kappa^* $, we will prove, in the next section, that every inaccessible $ \lambda $ of $ M_{\alpha} $ above $ \bar{\mu} = \sup\{ \mu_{\alpha'} \colon \alpha'<\alpha \} $ satisfies $ \left( \mbox{cf}(\lambda) \right)^V > \kappa $. So whenever $ \mu_{\alpha} $ is picked as the least measurable above $ \bar{\mu} $ with cofinality $ > \kappa $ in $ V $, it is simply the least measurable above $ \bar{\mu} $. The proof appears in lemma \ref{Lemma: NS, every measurable above mu bar has cofinility above kappa in V}, and a simpler characterization of $ \langle \mu_{\alpha} \colon \alpha<\kappa^* \rangle $ appears in corollary \ref{Corollary: NS, better definition of mu_alpha}. In order to avoid complications in the current section, we chose to provide those results, which involve a detailed study of the iteration $ \langle M_{\alpha} \colon \alpha\leq \kappa^* \rangle $, in the next section.
\end{remark}

The proof of theorem \ref{Theorem: NS, full description of j_W restricted to V}  goes as follows: By induction on $ \alpha \leq \kappa^* $, we define an elementary embedding $ k_{\alpha} \colon M_{\alpha} \to M $, as follows: 
$$  k_{\alpha} \left( j_{\alpha}(h)\left( \kappa,\mu_{\alpha_0},\ldots, \mu_{\alpha_k} \right)  \right) = j_{W}(h)\left( \kappa,\mu_{\alpha_0},\ldots, \mu_{\alpha_k} \right)   $$
for $ h\in V $, $ k<\omega $ and $ \alpha_0<\ldots<\alpha_k<\alpha $.

Note that for $ \alpha=0 $, $ k_{0} =k $ is the embedding defined in lemma \ref{Lemma: NS, k}. In general, it's not trivial that $ k_{\alpha} $ is a well defined elementary embedding. This will be proved in lemma \ref{Lemma: NS, k alpha is elementary}. We denote $ \lambda_{\alpha} = \mbox{crit}\left( k_{\alpha} \right) $. We will prove that for every $ \alpha<\kappa^* $, the following properties hold:

\begin{enumerate}[label = (\Alph*)]
	\item $ k_{\alpha} \colon M_{\alpha}\to M $ is an elementary embedding, and $ j_W\restriction_{V} = k_{\alpha} \circ j_\alpha $.
	\item $ \lambda_{\alpha}$ is measurable in $ M_{\alpha} $.
	\item $ \lambda_{\alpha} $ appears as an element in the Prikry sequence of $ k_{\alpha}\left( \lambda_{\alpha} \right) $ in $ M\left[H\right] $.
	\item $\lambda_{\alpha} =\mu_{\alpha}$.
	\item Let $ U_{\mu_{\alpha}} = \{  X\subseteq \mu_{\alpha} \colon \mu_{\alpha} \in k_{\alpha}(X) \}\cap M_{\alpha} $. Then $ U_{\mu_{\alpha}} \in M_{\alpha} $, and is a normal measure of Mitchell order $ 0 $ there. Moreover, $j_W\left( \mathcal{U} \right)\left( k_{\alpha}(\mu_{\alpha}) \right) = k_{\alpha}\left( U_{\mu_{\alpha}} \right)$, and, if $\mathcal{U} \in V $, then $  U_{\mu_{\alpha}} =  j_{\alpha}\left( \mathcal{U} \right) \left( \mu_{\alpha} \right) $.  
\end{enumerate}

After we prove that properties  (A)-(E)  above hold for every $ \alpha <\kappa^* $, we will show that $ k_{\kappa^*} $ is the identity function.

Let us assume now that the $ M_{\beta} $-ultrafilter $ U_{\mu_{\beta}} $ and the embedding $ k_{\beta} \colon M_{\beta}\to M $ have been defined for every $ \beta<\alpha $, such that properties (A)-(E) hold. We first prove that $ k_{\alpha} $ is indeed elementary.

\begin{lemma} \label{Lemma: NS, k alpha is elementary}
$ k_{\alpha} \colon M_{\alpha} \to M $ is an elementary embedding, and $ j_{W}\restriction_{V} = k_{\alpha} \circ j_{\alpha} $.
\end{lemma}

\begin{proof}
	We prove only that $ k_{\alpha} $ is a well defined injection (and the rest of elementarity follows similarly). Assume that $ a,a'\in M_{\alpha} $. Let $ k<\omega $, $ h,h'\in V $ and $ \alpha_0<\ldots<\alpha_k<\alpha $ be such that--
	$$ a= j_{\alpha}(h)\left( \kappa,\mu_{\alpha_0},\ldots, \mu_{\alpha_k} \right) \ ,\  a'= j_{\alpha}(h')\left( \kappa,\mu_{\alpha_0},\ldots, \mu_{\alpha_k} \right) $$
	If $ \alpha $ is limit, let $ \alpha'<\alpha $ be high enough such that $ \mu_{\alpha'} > \mu_{\alpha_k} $. By induction, $ j_W\restriction_{V} = k_{\alpha'}\circ j_{\alpha'} $, and thus--
	\begin{align*}
		& j_W(h) \left( \kappa,\mu_{\alpha_0},\ldots ,\mu_{\alpha_k} \right) =  j_W(h') \left( \kappa,\mu_{\alpha_0},\ldots ,\mu_{\alpha_k} \right) \\ \iff &
		j_{\alpha'}(h) \left( \kappa,\mu_{\alpha_0},\ldots ,\mu_{\alpha_k} \right) =  j_{\alpha'}(h') \left( \kappa,\mu_{\alpha_0},\ldots ,\mu_{\alpha_k} \right) \\ \iff &
		j_{\alpha}(h) \left( \kappa,\mu_{\alpha_0},\ldots ,\mu_{\alpha_k} \right) =  j_{\alpha}(h') \left( \kappa,\mu_{\alpha_0},\ldots ,\mu_{\alpha_k} \right)
	\end{align*}
	If $ \alpha = \alpha'+1 $ is successor, we can assume that $ \alpha_k = \alpha' $, and then--
	\begin{align*}
		& j_W(h) \left( \kappa,\mu_{\alpha_0},\ldots ,\mu_{\alpha_k} \right) =  j_W(h') \left( \kappa,\mu_{\alpha_0},\ldots ,\mu_{i_{k}} \right) \\ \iff &
		\mu_{\alpha'} \in k_{\alpha'}\left( \{ y<\mu_{\alpha'} \colon j_{\alpha'} (h)(\kappa,\mu_{\alpha_0},\ldots , \mu_{i_{k-1}},y) = j_{\alpha'}(h')(\kappa,\mu_{\alpha_0},\ldots , \mu_{i_{k-1}},y)  \}  \right)  \\ \iff &
		\{ y<\mu_{\alpha'} \colon j_{\alpha'} (h)(\kappa,\mu_{\alpha_0},\ldots , \mu_{i_{k-1}},y) = j_{\alpha'}(h')(\kappa,\mu_{\alpha_0},\ldots , \mu_{i_{k-1}},y)  \} \in U_{\mu_{\alpha'}} \\ \iff &
		j_{\alpha}(h) \left( \kappa,\mu_{\alpha_0},\ldots ,\mu_{\alpha_k} \right) =  j_{\alpha}(h') \left( \kappa,\mu_0,\ldots ,\mu_{i_{k}} \right)
	\end{align*}
	
	Finally, we argue that $ k_{\alpha}\circ j_{\alpha} = j_{W} \restriction_{V} $: For each $ x\in V $, let $ c_x\colon \kappa \to V $ be the function such that for every $ \xi<\kappa $, $ c_x(\xi) = x $. Then--
	$$ k_{\alpha}\left(  j_{\alpha}(x) \right) =  k_{\alpha}\left(  j_{\alpha}(c_x)(\kappa) \right) = j_{W}(c_x)(\kappa) = j_{W}(x) $$
\end{proof}

Since $ \alpha $ is fixed from now on, we denote simply $ \lambda = \lambda_{\alpha} =   \mbox{crit}\left( k_{\alpha} \right) $. Then $ \lambda $ is a regular uncountable cardinal. Our goal will be to prove that it is measurable in $ M_{\alpha} $, and moreover, $ \lambda = \mu_{\alpha} $. There are several straightforward limitations on the value of $ \lambda $:

\begin{claim}
$\sup\{\mu_{\alpha'}\colon \alpha'<\alpha \} \leq \lambda \leq \mu_{ \alpha}$.
\end{claim}

\begin{proof}
	By the definition of $ k_{\alpha} $, for every $ \alpha'<\alpha $,
	$$ k_{\alpha}\left( \mu_{\alpha'} \right) = k_{\alpha} \left( j_{\alpha}(id)(\mu_{\alpha'}) \right) = j_{W}(id)(\mu_{\alpha'}) = \mu_{\alpha'}$$
	Now, if $ x<\mu_{\alpha'} $ for some $ \alpha'<\alpha $, then  $ j_{\alpha',\alpha}(x) = x $. Thus, for some $ h\in V $, and $ \alpha_0<\ldots< \alpha_k<\alpha' $, $ x = j_{\alpha'}\left( h \right)\left( \kappa, \mu_{\alpha_0},\ldots, \mu_{\alpha_{k}} \right) $. Denote $ \vec{\mu} = \langle \mu_{\alpha_0},\ldots, \mu_{\alpha_{k}} \rangle $. Then--
	$$ k_{\alpha}(x) = k_{\alpha}\left(  j_{\alpha}(h)\left( \kappa,\vec{\mu}   \right)\right)  = j_{W}(h)\left(\kappa, \vec{\mu} \right) =  k_{\alpha'}\left(  j_{\alpha'}(h)\left( \kappa, \vec{\mu} \right) \right) = k_{\alpha'}\left( x \right) =x$$
	where the last equality holds since $ x<\mu_{\alpha'} $, and, by induction, $ \mbox{crit}\left( k_{\alpha'} \right) = \mu_{\alpha'} $.
	
	This shows that $ \mbox{crit}(k_{\alpha}) \geq  \mu_{\alpha'} $ for every $ \alpha'<\alpha $.
	
	For the second inequality, recall that $ \mu_{\alpha} $ is measurable in $ M_{\alpha} $ which satisfies $ \left( \mbox{cf}\left( \mu_{\alpha} \right) \right)^V> \kappa $. If $ k_{\alpha}\left( \mu_{\alpha} \right) = \mu_{\alpha} $, then, by elementarity, $ \mu_{\alpha} $ is measurable in $ M $. Therefore, in $ M\left[H\right] $, $ \mbox{cf}\left( \mu_{\alpha} \right) = \omega $, and thus in $ V\left[G\right] $, $ \mbox{cf}\left( \mu_{\alpha} \right) = \omega $. Therefore, in $ V $,  $ \mbox{cf}\left( \mu_{\alpha} \right) \leq \kappa $, a contradiction.
\end{proof}

Recall that for every $ \beta<\alpha $, $ \mu_{\beta} $ appears as an element in the Prikry sequence added to $ k_{\beta }\left( \mu_{\beta} \right) $ in $ M\left[H\right] $. Assume that it is the $ \left(n_{\beta} +1\right)$-th element in this Prikry sequence, and has an initial segment $ t_{\beta} $ of length $ n_{\beta} $ below it. Note that, by induction, $ k_{\beta}\left( t_{\beta} \right) = t_{\beta} $. 

We now provide a useful way to represent elements in the model $ M_{\alpha} $.

\begin{defn} \label{Definition: NS, Nice seqeunces of ordinals}
An increasing sequence $ \langle \alpha_0,\ldots, \alpha_k \rangle $ of ordinals below $\kappa^* $ is called \textbf{nice} if, for every $ 0\leq i \leq k $, there are functions $ g_i, t_i, F_i \in V $ such that--
$$ \mu_{\alpha_{i}} = j_{\alpha_i}\left( g_i \right) \left( \kappa, \mu_{\alpha_0}, \ldots, \mu_{\alpha_{i-1}} \right) $$
$$ t_{\alpha_{i}} = j_{\alpha_i}\left( t_i \right) \left( \kappa, \mu_{\alpha_0}, \ldots, \mu_{\alpha_{i-1}} \right) $$
$$ U_{\mu_{\alpha_{i}}} = j_{\alpha_i}\left( F_i \right) \left( \kappa, \mu_{\alpha_0}, \ldots, \mu_{\alpha_{i-1}} \right) $$
(for $ i=0 $, $ \mu_{\alpha_0} = j_{\alpha_0}\left( g_{0} \right)(\kappa)  $, $ t_{\alpha_0} = j_{\alpha_0}\left( t_0 \right)\left( \kappa \right) $ and $ U_{\mu_{\alpha_0}} = j_{\alpha_0}\left( F_0 \right)\left( \kappa \right)  $ ).
\end{defn}

(We remark that the functions $ F_i $ used to represent $ U_{\mu_{\alpha_i}} $ will be relevant only in the next section, so the third requirement, that includes them, can be omitted from the definition at the moment).
It's not hard to prove that, given a pair of nice sequences, the increasing enumeration of their union is nice.

\begin{lemma}
Every element in $ M_{\alpha} $ has the form--
$$j_{\alpha}(h)\left( \kappa,\mu_{\alpha_0},\ldots, \mu_{\alpha_k} \right)$$
for some $ k<\omega $, $ \left(k+1\right) $-ary function $ h\in V $ and a nice sequence $\langle \alpha_0,\ldots , \alpha_{k} \rangle$ of ordinals below $\alpha $.
\end{lemma}

\begin{proof}
We assume that the lemma holds for every $ \alpha'<\alpha $. Let $ x\in M_{\alpha} $.

If $ \alpha $ is limit: There exists $ \alpha'<\alpha $ and $ x'\in M_{\alpha'} $ such that $ x = j_{\alpha', \alpha}\left( x' \right) $. By induction, $ x' = j_{\alpha'}\left( h \right)\left( \mu_{\alpha_0},\ldots ,\mu_{\alpha_k} \right) $ for a nice sequence $ \langle \alpha_0,\ldots, \alpha_k \rangle $ below $ \alpha' $. Then $ x = j_{\alpha}\left( h \right)\left( \mu_{\alpha_0},\ldots ,\mu_{\alpha_k} \right) $, as desired.

If $ \alpha = \alpha'+1 $ is successor: Let $ f\in M_{\alpha}$ be a function such that $x= j_{\alpha', \alpha}(f)\left( \mu_{\alpha'} \right) $. Let $ h_1,h_2,h_3,h_4\in V $ be functions, and $ \langle \alpha_0,\ldots, \alpha_k \rangle $, $ \langle \beta_0,\ldots, \beta_l \rangle $, $ \langle \gamma_0, \ldots, \gamma_{s} \rangle $, $ \langle \delta_0,\ldots, \delta_{r} \rangle $ be nice sequences below $ \alpha' $ such that--
$$ f = j_{\alpha'}\left( h_1 \right)\left( \mu_{\alpha_0}, \ldots, \mu_{\alpha_k} \right)  \  ,  \   \mu_{\alpha'} = j_{\alpha'}\left( h_2 \right)\left( \mu_{\beta_0}, \ldots, \mu_{\beta_k} \right) $$
$$ t_{\alpha'} = j_{\alpha'}\left( h_3 \right)\left( \mu_{\gamma_0}, \ldots, \mu_{\gamma_s} \right)  \  ,  \   U_{\mu_{\alpha'}} = j_{\alpha'}\left( h_4 \right)\left( \mu_{\delta_0}, \ldots, \mu_{\delta_r} \right) $$
The increasing enumeration of-- 
$$\langle \alpha_0,\ldots, \alpha_k \rangle \cup \langle \beta_0,\ldots, \beta_l \rangle \cup \langle \gamma_0, \ldots, \gamma_{s} \rangle \cup \langle \delta_0,\ldots, \delta_{r} \rangle  \cup \langle \alpha' \rangle $$
is a nice sequence. Denote it by $ \langle \varepsilon_0, \ldots, \varepsilon_{m}, \alpha' \rangle$, where $ \varepsilon_m < \alpha' $.

By modifying the function $ h_1$ in $ V $, we can assume for simplicity that--
$$ f = j_{\alpha'}\left( h_1 \right)\left( \mu_{\varepsilon_0}, \ldots, \mu_{\varepsilon_m} \right)   $$
Define, in $ V $, a function $ h $, as follows: 
$$ h\left( \langle \nu_0 ,\ldots, \nu_m, \nu \rangle \right) =  h_1\left( \nu_0 , \ldots, \nu_{m} \right) \left( \nu \right)   $$
 Then $ j_{\alpha}(h)\left( \mu_{\varepsilon_0}, \ldots, \mu_{\varepsilon_{m}}, \mu_{\alpha'} \right) = x $.
\end{proof}

We now introduce several notations. We fix those notations throughout the proof that properties (A)-(E) hold at $ \alpha $. Recall that $ \mbox{crit}\left( k_{\alpha} \right) $ is denoted by $ \lambda  $. Let $ h\in V $ be a function such that-- 
$$\lambda = j_{\alpha}(h)\left( \kappa, \mu_{\alpha_0},\ldots, \mu_{\alpha_k} \right) $$
for a nice sequence $ \langle \alpha_0,\ldots, \alpha_k \rangle $ below $ \alpha $. Fix, for every $ 0\leq i \leq k $, functions $ g_i, t_i  \in V $ as in the definition of a nice sequence. In other words--
$$ \mu_{\alpha_{i}} = j_{\alpha_i}\left( g_i \right) \left( \kappa, \mu_{\alpha_0}, \ldots, \mu_{\alpha_{i-1}} \right) $$
$$t_{\alpha_{i}} = j_{\alpha_i}\left( t_i \right) \left( \kappa, \mu_{\alpha_0}, \ldots, \mu_{\alpha_{i-1}} \right) $$
\begin{remark} \label{Remark: NS, nice sequence is chosen such that h is strictly above gk}
\begin{enumerate}
	\item The functions $ g_i $ might be more or less the same. For instance,  set, for every $ \xi<\kappa $, $ g_0 = s(\xi) =$ the first measurable in $ V $ strictly above $ \xi $, and $ g_1(\xi, \nu) = s(\xi) $.	Then $ \mu_0 = j_{0}\left( g_0 \right)(\kappa) $ and $ \mu_1 =j_1(g_0)(\kappa) = j_1\left( g_1 \right)(\kappa) $. 
	\item It is not necessarily true that, given $ \xi, \vec{\nu} $, $ h\left( \xi, \vec{\nu} \right) \geq g_i\left( \xi, \nu_0, \ldots, \nu_{i-1} \right) $. For instance, take, $ \mu_{\alpha_k} $ to be a measurable of Mitchell order $ >0 $ in $ M_U $, and $ \lambda $ to be the first measurable above it in $  M_{ \mu_{\alpha_k} +1} = \mbox{Ult}\left( M_{\alpha_k}, U_{\mu_{\alpha_k}} \right) $. Then $ \lambda = j_{\alpha_{k}+1}(h)\left( \kappa, \mu_{\alpha_k} \right) $, where $ h(\xi, \nu)  = s\left(\nu\right)$. Assume that $ \mu_{\alpha_k} = j_U(f)(\xi) $ for some $ f\in V $. In $ M\left[H\right] $, $ k_{\mu_{\alpha_k}+1} (\lambda) < k_{\mu_{ \alpha_k } }( \mu_{\alpha_k} ) $, namely, $ h( \xi, \mu_{\alpha_k}(\xi) ) < f\left( \xi \right) $ for a set of $ \xi $-s in $ W $, where $ \xi \mapsto \mu_{\alpha_k}(\xi) $ is a function in $ V\left[G\right] $ represents $ \mu_{\alpha_k} $ in the ultrapower with $ W $.
\end{enumerate}
\end{remark}

Given $ \beta<\alpha $, recall that, by induction, $ \mu_{\beta} $ appears in the Prikry sequence of $ k_{\beta} (\mu_{\beta}) $. For every $ 0\leq i\leq k $, denote by $ n_{i}<\omega $ the length of the finite sequence $ t_{i} $, which is the initial segment of the Prikry sequence of $ k_{\alpha_i}\left( \mu_{\alpha_i} \right) $ below $ \mu_{\alpha_i} $. Then $ \mu_{\alpha_i} $ is the $ \left(n_i+1\right) $-th element in this Prikry sequence.

For every $ i\leq k $, we define, in $ V\left[G\right] $, a function $ \xi \mapsto \mu_{\alpha_i}(\xi) $ such that $ \left[ \xi \mapsto \mu_{\alpha_i}(\xi) \right]_{W} = \mu_{\alpha_i} $:
\begin{itemize}
	\item For $ i=0 $,  set the $ \left(n_0+1\right) $-th element in the Prikry sequence of $ g_0(\xi) $ to be $ \mu_{\alpha_0}(\xi) $.
	\item Assume that $ 0<i< k $, and the functions $ \xi\mapsto \mu_{\alpha_j}(\xi) $ have been defined for every $ j\leq i $. Let $ \mu_{\alpha_i}(\xi) $ be the $ \left(n_i+1\right) $-th element in the Prikry sequence of $ g_i\left(  \mu_{\alpha_0}(\xi), \ldots, \mu_{\alpha_{i-1} }(\xi) \right) $.
\end{itemize}
For every $ 0\leq i \leq k $, $ \left[ \xi\mapsto \mu_{\alpha_i}(\xi) \right]_W = \mu_{\alpha_i} $, and-- 
$$t_{\alpha_i} =  \left[ \xi\mapsto t_i\left( \xi, \mu_{\alpha_0}(\xi), \ldots, \mu_{\alpha_{i-1}}(\xi) \right) \right]_W $$ 
where the last equality follows since $ \mbox{crit}\left( k_{\alpha_i} \right) = \mu_{\alpha_i} $ and thus $ k_{\alpha_i}\left( t_{\alpha_i} \right) = t_{\alpha_i} $.

We fix an abbreviation, $  \xi \mapsto \vec{\mu}(\xi) $ for the function $ \xi \mapsto \langle \mu_{\alpha_0}(\xi),\ldots, \mu_{\alpha_k}(\xi) \rangle $.  Given $ \xi, \vec{\nu} = \langle \nu_0,\ldots, \nu_k \rangle $, denote--
$$ \vec{t}\left( \xi, \vec{\nu} \right) = \langle t_0(\xi), t_1(\xi, \nu_0), \ldots, t_k\left( \xi, \nu_0, \ldots, \nu_{k-1} \right) \rangle $$

Our next goal is lemma \ref{Lemma: NS, Multivariable Fusion}, which generalizes the Fusion Lemma \ref{Lemma: NS, Fusion Lemma}. We deal there with sets which are $ \leq^* $ dense open above conditions which decide the values of $ \langle \mu_{\alpha_0}(\xi), \ldots, \mu_{\alpha_{k}}(\xi) \rangle $. We first define the notion of a $ C $-tree, which consists of sequences $\langle \xi, \vec{\nu} \rangle =\langle \xi, \nu_0, \ldots, \nu_{k} \rangle $ which are possible candidates for the exact values of $ \langle \xi, \mu_{\alpha_0}(\xi), \ldots, \mu_{\alpha_{k}}(\xi) \rangle $. Then, we define in \ref{Definition: NS, admissible sequence for a given condition} whenever such a candidate is admissible for a given condition $ p\in G $, in the sense that $ p $ can be extended to force that $ \vec{\mu}(\xi) = \vec{\nu} $.

\begin{defn} \label{Definition: NS, C tree}
A tree $ T\subseteq \left[\kappa\right]^{k+1} $ is called a $ C $-tree (with respect to a fixed nice sequence $ \langle \alpha_0, \ldots, \alpha_k \rangle $) if $ \mbox{Succ}_{T}\left( \langle \rangle \right)  $ is a club in $ \kappa $, and for every $ i<k $ and $ \langle \xi, \nu_0, \ldots, \nu_i \rangle\in T $,
 $\mbox{Succ}_T\left( \xi, \nu_0, \ldots, \nu_i  \right)$ is a club in $ g_{i+1}\left( \xi, \nu_0, \ldots, \nu_i \right) $.
 
Given $ i<k $ and a sequence $ \langle \xi, \nu_0, \ldots, \nu_i \rangle $, a $ C $-tree above it is a tree $ T\subseteq \left[\kappa\right]^{n-i} $, such that $ \mbox{Succ}_T(\langle \rangle) $ is a club in $ g_{i+1}\left( \xi, \nu_0, \ldots, \nu_i \right) $ and, for every $ i+1 \leq j \leq  k -1 $ and $ \langle \nu_{i+1}, \ldots , \nu_{j} \rangle\in T$, $\mbox{Succ}_{T}\left(\nu_{i+1}, \ldots , \nu_{j}  \right)$ is a club in $ g_{j+1}\left( \xi, \nu_0, \ldots, \nu_{j} \right) $.
\end{defn}

\begin{claim} \label{Claim: NS, a C tree contains the real values}
Let $ T $ be a $ C $-tree. Then, in $ V\left[G\right] $,
$$ \{ \xi <\kappa \colon   \langle \xi, \mu_{\alpha_0}(\xi), \ldots , \mu_{\alpha_k}(\xi) \rangle \in T \}\in W $$
\end{claim}

\begin{proof}
Work in $ V\left[G\right] $. First, $ \{ \xi<\kappa \colon \mu_{\alpha_0}(\xi)\in \mbox{Succ}_{T}(\xi) \}\in W $. Indeed, for each $ \xi\in \mbox{Succ}_T(\langle \rangle)\in W $, $ \mbox{Succ}_T(\xi) $ is a club in $ g_0(\xi) $, and thus--
$$ \mu_0\in k_0 \left( \left[\xi \mapsto \mbox{Succ}_T\left( \langle \xi \rangle \right) \right]_{U} \right) $$
This holds since $ \left[\xi \mapsto \mbox{Succ}_T\left( \langle \xi \rangle \right) \right]_{U} $ is a club in $ \left[g_0\right]_U = \mu_{\alpha_0} $ and thus belongs to $ U_{\mu_{\alpha_0}} $. 

Now proceed by induction. For every $ i\leq k-1 $, 
$$ \{ \xi < \kappa \colon \mu_{\alpha_{i+1}}(\xi) \in  \mbox{Succ}_T\left( \langle \xi, \mu_{\alpha_0}(\xi), \ldots, \mu_{\alpha_i}(\xi)  \rangle\right) \}\in W $$
Indeed, denote--
$$C = j_{\alpha_{i+1}}\left( \langle \xi, \nu_{0},\ldots, \nu_i \rangle  \mapsto  \mbox{Succ}_{T}\left( \xi, \nu_0, \ldots, \nu_{i} \right) \right)\left( \kappa, \mu_{\alpha_0}, \ldots, \mu_{\alpha_i} \right) $$
Then $ C $ is a club in $ j_{\alpha_{i+1}}\left( g_{i+1} \right)\left( \kappa, \mu_{\alpha_0}, \ldots, \mu_{\alpha_i} \right) = \mu_{\alpha_{i+1}} $. Thus $ C\in U_{\mu_{\alpha_{i+1}}} $, and  $ \mu_{\alpha_{i+1}}\in k_{\alpha_{i+1}}(C) $, as desired.
\end{proof}

\begin{defn} \label{Definition: NS, admissible sequence for a given condition}
	Fix $ \alpha<\kappa $ and a nice sequence $ \langle \alpha_0, \ldots, \alpha_k \rangle $ below $ \alpha $. Let $ p\in P_{\kappa} $ be a condition and $ \langle \xi, \nu_0, \ldots, \nu_{k} \rangle $ be a sequence below $ \kappa $. Let us define whenever $ \langle \xi, \nu_0, \ldots, \nu_{k} \rangle $ is admissible for $ p $, and in that case, define as well an extension $ p^{\frown} \langle \xi, \nu_0, \ldots, \nu_k \rangle \geq p $ in $ P_{\kappa} $.
	\begin{enumerate}
		\item $ \langle \xi ,\nu_0 \rangle $ is admissible for $ p $ if-- 
		$$ p\restriction_{ g_0(\xi) } \Vdash \langle t_{0}\left( \xi \right)^{\frown} \langle \nu_0 \rangle, A^{p}_{g_0(\xi)}\setminus  \left(\nu_0+1\right) \rangle \mbox{ is compatible with } p(g_0(\xi)) $$
		if this holds, and $ t^{p}_{g_0(\xi)}  $ is an initial segment of $ t_0(\xi)^{\frown} \langle \nu_0 \rangle $, let-- 
		$$ p^{\frown} \langle \xi, \nu_0 \rangle = {p\restriction_{g_0(\xi) }}^{\frown} {\langle t_{0}\left( \xi \right)^{\frown} \langle \nu_0 \rangle, A^{p}_{g_0(\xi)}\setminus  \left(\nu_0+1\right) \rangle}^{\frown} p\setminus \left(g_0(\xi)+1\right) $$
		otherwise, let $ p^{\frown} \langle \xi, \nu_0 \rangle = p $.
		\item Let $  0 \leq m <k$. Assume that $ \langle \xi, \nu_0, \ldots, \nu_{m} \rangle $ is admissible for $ p $ and $ p^{\frown} \langle \xi, \nu_0, \ldots, \nu_{m} \rangle $ has been defined. Denote-- 
		$$ g_{m+1} = g_{m+1}\left( \xi, \nu_0, \ldots, \nu_{m} \right) $$
		$$ t_{m+1} = t_{m+1}\left( \xi, \nu_0, \ldots, \nu_{m} \right) $$
		We say that $ \langle \xi, \nu_0, \ldots, \nu_{m+1} \rangle $ is admissible for $ p $ if--
		\begin{align*}	
			p^{\frown} \langle \xi, \nu_0, \ldots, \nu_{m} \rangle \restriction_{ g_{m+1} } \Vdash & \langle {t_{m+1}}^{\frown} \langle \nu_{m+1} \rangle, A^{p}_{g_{m+1}}\setminus  \left(\nu_{m+1}+1\right) \rangle  \mbox{ is }\\
			&\mbox{compatible with }\left( p^{\frown} \langle \xi, \nu_0, \ldots, \nu_{m} \rangle  \right)  (g_{m+1})
		\end{align*}
		if this holds, and $ t^{p^{\frown} \langle \xi, \nu_0, \ldots, \nu_{m} \rangle}_{g_{m+1}} $ is an initial segment of $ {t_{m+1}}^{\frown} \langle \nu_{m+1} \rangle $, let--
		\begin{align*}
			p^{\frown} \langle \xi, \nu_0, \ldots, \nu_{m+1} \rangle = &   	{  \left({p^{\frown}\langle  \xi, \nu_0, \ldots, \nu_{m} \rangle  }   \restriction_{g_{m+1} }\right)}^{\frown}\\
			& {\langle {t_{m+1}}^{\frown} \langle \nu_{m+1} \rangle, A^{p}_{g_{m+1}}\setminus  \left(\nu_{m+1}+1\right) \rangle}^{\frown} \\
			& \left({p^{\frown} \langle \xi, \nu_0, \ldots, \nu_{m} \rangle }\right) \setminus \left(g_{m+1}+1\right) 
		\end{align*}
		else, set $ p^{\frown} \langle \xi, \nu_0, \ldots, \nu_{m+1} \rangle = p^{\frown} \langle \xi, \nu_0, \ldots, \nu_{m} \rangle $.
	\end{enumerate}
Finally, assume that $ p $ is a condition, $ \xi<\kappa $, $ i<k $ and $ \langle  \nu_0, \ldots, \nu_i \rangle $ is a sequence such that $ p\Vdash \langle \lusim{\mu}_{\alpha_0}(\xi), \ldots, \lusim{\mu}_{\alpha_i}(\xi) \rangle = \langle \nu_0, \ldots, \nu_i \rangle $. Given a sequence $ \langle \nu_{i+1}, \ldots, \nu_k \rangle $, we can define similarly whether it is admissible for $ p $; if it is, we say that $ \langle  \nu_{i+1}, \ldots, \nu_k\rangle $ is admissible for $ p $ above $ \langle \xi, \nu_0, \ldots, \nu_i \rangle $, and define, in a similar way as above, the condition $ p^{\frown} \langle \nu_{i+1}, \ldots, \nu_k \rangle $.
\end{defn}

\begin{lemma} [Multivariable Fusion] \label{Lemma: NS, Multivariable Fusion}
	Fix $ \alpha<\kappa $ and a nice sequence $ \langle \alpha_0, \ldots, \alpha_k \rangle $ below $ \alpha $. Let $ p\in P_{\kappa} $ be a condition. Assume that for every $ \langle \xi, \nu_0, \ldots, \nu_k \rangle $ below $ \kappa $ there exists a subset $ e\left( \xi, \vec{\nu} \right) \subseteq P_{\kappa}\setminus \left(\nu_k+1\right) $ which is $ \leq^* $-dense open above every condition $ q\in P\setminus \left(\nu_k+1\right) $ which forces that $ \vec{\lusim{\mu}}(\xi) = \vec{\nu} $.
	Then there exists $ p^*\geq^* p $ and a $ C $-tree $ T $, such that for every $ \langle \xi, \nu_0, \ldots, \nu_k \rangle\in T $ which is admissible for $ p^* $,
	$$ \left({p^*}^{\frown} \langle \xi, \vec{\nu} \rangle\right)\restriction_{\nu_{k}+1}  \Vdash \left({p^*}^{\frown} \langle \xi, \vec{\nu} \rangle\right) \setminus \left(\nu_k+1\right) \in e\left( \xi, \vec{\nu} \right)$$
\end{lemma}

\begin{proof}
	For every $ i<k  $ and  $ \langle \xi, \nu_0,\ldots, \nu_i \rangle $, we define a subset $ e\left( \xi, \nu_0, \ldots, \nu_i \right)\subseteq P\setminus \left( \nu_i+1 \right) $ which is $ \leq^* $-dense open above every condition $ q\in P\setminus \left(\nu_i+1\right) $ which forces that-- 
	$$ \langle \lusim{\mu}_{\alpha_0}(\xi), \ldots, \lusim{\mu}_{\alpha_i}(\xi) \rangle = \langle \nu_0, \ldots, \nu_i \rangle $$ 
	as follows:
	
	\begin{align*}
		e\left( \xi, \nu_0, \ldots, \nu_i \right) = &\{  q\in P\setminus \left( \nu_i+1 \right)  \colon \mbox{there exists a C-tree } T \mbox{ above } \langle \xi, \nu_0, \ldots, \nu_i \rangle\\
		& \mbox{such that, for every } \langle \nu_{i+1}, \ldots, \nu_{k} \rangle\in T, \mbox{ which is admissible for  }\\
		&q \mbox{ above } \langle \xi, \nu_0, \ldots, \nu_{i} \rangle,\\
		&\left(q^{\frown} \langle \nu_{i+1}, \ldots, \nu_{k} \rangle\right) \restriction_{\nu_{k}+1} \Vdash \left(q^{\frown} \langle \nu_{i+1}, \ldots, \nu_{k} \rangle\right)\setminus \left(\nu_{k}+1\right)\in e\left( \xi, \vec{\nu} \right)
		\}
	\end{align*}
	
	The lemma now follows by applying, repeatedly, the following claim:
	\begin{claim}
		Let $ 0\leq i<k $ and fix an increasing sequence $\langle \xi, \nu_0, \ldots, \nu_{i}, \nu_{i+1}\rangle $. Assume that $ e\left( \xi, \nu_0, \ldots, \nu_i, \nu_{i+1} \right) $
		is $ \leq^* $-dense open above every condition in $ P\setminus \left(\nu_{i+1}+1\right) $ which forces that $ \langle \lusim{\mu}_{\alpha_0}(\xi), \ldots, \lusim{\mu}_{\alpha_{i+1}}(\xi) \rangle = \langle \nu_0, \ldots, \nu_{i+1} \rangle $. Then $ e\left( \xi, \nu_0, \ldots, \nu_i \right) $ is $ \leq^* $-dense open above every condition in $ P\setminus \left(\nu_{i}+1\right) $ which forces that $ \langle \lusim{\mu}_{\alpha_0}(\xi), \ldots, \lusim{\mu}_{\alpha_{i}}(\xi) \rangle = \langle \nu_0, \ldots, \nu_{i} \rangle $.
	\end{claim}
	
	\begin{proof}
		Let $ p\in P\setminus \left(\nu_i+1\right) $ be a condition which forces that $ \langle \lusim{\mu}_{\alpha_0}(\xi), \ldots, \lusim{\mu}_{\alpha_{i}}(\xi) \rangle = \langle \nu_0, \ldots, \nu_{i} \rangle $. Denote for simplicity $ g_{i+1} = g_{i+1}\left( \xi, \nu_0, \ldots, \nu_{i} \right) $. First, direct extend $ p\restriction_{g_{i+1}} $ such that it decides the length of $ t^{p}_{g_{i+1}} $, and whether $ t^{p}_{g_{i+1}}, t_{i+1}\left( \xi, \nu_0, \ldots, \nu_{i} \right) $ are compatible:
		\begin{enumerate}
			\item If $ p\restriction_{g_{i+1}} $ decides that $ t^{p}_{g_{i+1}} $ and $ t_{i+1}\left( \xi, \nu_0, \ldots, \nu_{i} \right) $ are incompatible, do nothing.
			\item If $ p\restriction_{g_{i+1}} $ decides that the length of $ t^{p}_{g_{i+1}} $ is at least $ n_{i+1}+1 $, direct extend it further, such that for some $ \gamma< g_{i+1} $, $ p\restriction_{g_{i+1}}\Vdash t^{p}_{g_{i+1}}\left( n_{i+1}+1 \right) <\gamma  $ (namely, $ \gamma $ bounds the $ \left(n_{i+1}+1\right) $-th element in the Prikry sequence of $ g_{i+1} $).
			\item If $ p\restriction_{g_{i+1}} $ decides that $ g_{i+1}\notin \mbox{supp}\left(  p\right) $, direct extend $ p $ such that $ t^{p}_{g_{i+1}} = t_{i+1}\left( \xi, \nu_0, \ldots, \nu_i \right)$.
			\item If $ p\restriction_{g_{i+1}} $ decides that the length of $ t^{p}_{g_{i+1}} $ is less or equal than $ n_{i+1} $, direct extend by shrinking $ \lusim{A}^{p}_{g_{i+1}} $ to $ A^{p}_{g_{i+1}}\setminus \left( \max\left(  t_{i+1}\left( \xi ,\nu_0,\ldots, \nu_{i} \right) \right)  +1 \right) $.
		\end{enumerate}
		Assume that $ p$ is already direct extended as described above. Let us direct extend $ p^*\restriction_{g_{i+1}} \geq^* p\restriction_{g_{i+1}} $ using the Fusion lemma in the forcing $ P\restriction_{  \left( \nu_{i} , g_{i+1}  \right)} $. For every $ \nu\in \left( \nu_{i},g_{i+1}  \right) $, consider the following $ \leq^* $-dense open subset of $ P\restriction_{  \left( \nu+1 , g_{i+1}  \right)}  $:
		\begin{align*}
			E(\nu) = &\{ r\in P\restriction_{  \left( \nu+1 , g_{i+1}  \right)} \colon \mbox{if } r\Vdash t^{r}_{g_{i+1}} = t_{i+1}\left( \xi ,\nu_0, \ldots, \nu_{i} \right) \mbox{ and } \nu\in \lusim{A}^{r}_{g_{i+1}},\\
			&\mbox{there exists a direct extension-- }\\
			&q = q(\nu)\geq^* {\langle {t^{p}_{g_{i+1}}}^{\frown}\langle \nu \rangle , \lusim{A}^{r}_{g_{i+1}}\setminus \left( \nu+1 \right)  \rangle}^{\frown} p\setminus \left( g_{i+1}+1 \right) \\
			&\mbox{such that } r^{\frown} q\in e\left( \xi, \nu_0, \ldots, \nu_{i}, \nu \right)  \}
		\end{align*} 
		The $ \leq^* $-density of $ E(\nu) $ follows from the $ \leq^* $-density of $ e\left( \xi, \nu_0, \ldots, \nu_{i}, \nu \right) $ above any condition which forces that $ \langle \lusim{\nu}_{\alpha_0}(\xi), \ldots, \lusim{\mu}_{\alpha_{i+1}}(\xi) \rangle = \langle \xi, \nu_0, \ldots, \nu_{i}, \nu \rangle $. \\
		Apply Fusion, and let $ p^*\restriction_{g_{i+1}} \geq^* p\restriction_{g_{i+1}} $ be a direct extension, such that for some club $ C = C\left( \xi, \nu_0 ,\ldots, \nu_{i}\right) \subseteq g_{i+1} $, and for every $ \nu\in C $, 
		$$  {p^*\restriction_{\nu+1}} \Vdash p^*\setminus \left( \nu+1 \right)\in e\left( \xi, \nu_0, \ldots, \nu_{i}, \nu \right) $$
		Shrink $ C $ such that $ C\cap \left(\gamma+1\right) =\emptyset $ (if necessary, namely, if $ \gamma  $ was defined and $ C $ contains ordinals below it). \\
		Let us define now $ p^*\left( g_{i+1} \right) $. For every $ \nu\in C $, such that $ p^*\restriction_{g_{i+1}}\Vdash t^{p^*}_{g_{i+1}} = t_{i+1}\left( \xi, \nu_0, \ldots, \nu_{i} \right) $ and $ \nu\in A^{p^*}_{g_{i+1}} $, let $ q(\nu) $ be the condition as in the definition of $ E(\nu) $. For every other value of $ \nu $, let $ q(\nu) = p\setminus g_{i+1} $. Now, direct extend $ p(g_{i+1}) $ to-- 
		$$ p^*\left( g_{i+1} \right)  = \langle t^{p}_{g_{i+1}} , \lusim{A}^{p}_{g_{i+1}} \cap \left( \triangle_{\nu<g_{i+1}} { \lusim{A}^{q(\nu)}_{g_{i+1} }  }  \right)\cap C \rangle$$
		Finally, we define $ p^*\setminus \left( g_{i+1}+1 \right) = q\left( \nu \right) $, where $ \nu $ is the $ \left(n_{i+1}+1\right) $-th element in the Prikry sequence of $ g_{i+1} $.
		
		Let us argue now that $ p^*\in e\left( \xi, \nu_0, \ldots, \nu_i \right) $. We first define a $ C $-tree $ T $ above $ \langle \xi, \nu_0, \ldots, \nu_{i} \rangle $. Let $ \mbox{Succ}_{T}\left( \langle \rangle \right) =C =  C\left( \xi, \nu_0, \ldots, \nu_{i} \right) $ . Fix $ \nu_{i+1} =\nu\in C $, and let us define $T_{\langle \nu\rangle}$, which is the tree $ T $ above the node $\langle \nu \rangle $.
		
		If $ p^*\restriction_{g_{i+1}} $ forces that $ t^{p^*}_{g_{i+1}} \neq  t_{i+1}\left( \xi, \nu_0, \ldots, \nu_i \right) $ or $ \nu\notin \lusim{A}^{p^*}_{g_{i+1}} $, let $ T_{\langle \nu \rangle} $ be any $ C $-tree above $ \langle \xi, \nu_0 ,\ldots, \nu_{i},  \nu \rangle $ (we will prove that any branch starting from $ \nu $ in $ T $ is not admissible for $ p^* $). 
		
		Else, note that--
		$$ {p^*}^{\frown} \langle \nu \rangle ={ {p^*}\restriction_{g_{i+1}}}^{\frown} {\langle {t^{p^*}_{g_{i+1}}}^{\frown} \langle \nu \rangle, A^{p^*}_{g_{i+1}}\setminus \left( \nu+1 \right) \rangle }^{\frown} p^*\setminus \left( g_{i+1}+1 \right) \geq^* {p^*\restriction_{g_{i+1}}}^{\frown} q(\nu) $$
		since $ \lusim{A}^{p^*}_{g_{i+1}}\setminus \left( \nu_{i+1}+1 \right) \subseteq \lusim{A}^{q\left( \nu_{i+1} \right)}_{g_{i+1}} $. Thus, $ {p^*}^{\frown} \langle \nu \rangle $ belongs to $ e\left( \xi, \nu_0, \ldots, \nu_{i}, \nu \right) $. This is witnessed by a $ P\restriction_{\nu+1} $-name for a $ C $-tree $ \lusim{T}\left( \nu \right) $ above $ \langle \xi, \nu_0, \ldots, \nu_{i}, \nu \rangle $. We construct $ T_{\langle \nu \rangle} $ in $ V^{P_{\nu+1} } $ to be a $ C $-tree which is forced, by $ {{p^*}^{\frown}\langle \nu \rangle}\restriction_{ \nu+1} = p^*\restriction_{\nu+1} $ to be contained in $ \lusim{T}\left( \nu \right) $. The definition is inductive: First, let $ \mbox{Succ}_{T}\left( \nu \right)\subseteq g_{i+1} $ be a club in $ V^{P_{\nu_i+1}} $ which is forced by $  p^*\restriction_{\nu+1} = \left({p^*}^{\frown} \langle \nu \rangle \right)\restriction_{\nu+1} $ to be contained in $ \mbox{Succ}_{\lusim{T}(\nu)} \left( \langle \rangle \right) $; Such a club exists since the forcing $ P\restriction_{\left( \nu_i, \nu+1 \right)} $ has cardinality strictly below $ g_{i+1} $. Now, given $ \nu_{i+2}\in \mbox{Succ}_{T}\left(  \nu \right) $, let $ \mbox{Succ}_{T}\left( \nu, \nu_{i+2} \right) \subseteq g_{i+2}\left( \xi, \nu_0, \ldots, \nu_i, \nu, \nu_{i+2} \right) $ be a club which is forced by $ p^*\restriction_{\nu+1} $  to be contained in $ \mbox{Succ}_{\lusim{T}(\nu)  }\left( \langle \nu_{i+2} \rangle \right) $. Continue in this fashion.
		
		This finishes the definition of $ T $. Finally, assume that $ \langle \nu_{i+1}, \ldots, \nu_{k} \rangle $ belongs to $ T $ and is admissible for $ p^* $ above $ \langle \xi, \nu_0, \ldots, \nu_{i} \rangle $. Then $ p^*\restriction_{\nu_{i+1}+1} $ forces that $ \langle \nu_{i+2}, \ldots, \nu_{k} \rangle \in \lusim{T}\left( \nu \right) $. By admissibility of $ \langle \nu_{i+1}, \dots, \nu_{k} \rangle $ for $ p^* $, $ \nu_{i+1}\in \lusim{A}^{p^*}_{g_{i+1}} $, and $ t^{p^*}_{g_{i+1}}  $ is compatible with, but not a strict initial segment of $ t_{i+1}\left( \xi, \nu_0 ,\ldots, \nu_{i} \right)$. Since $ \nu_{i+1} $ belongs to  $C\left( \xi, \nu_0, \ldots, \nu_{i} \right) $, and in particular is above $ \gamma $, 
		$ t^{p^*}_{g_{i+1}}  = t_{i+1}\left( \xi, \nu_0 ,\ldots, \nu_{i} \right)$. Thus, 
		$$ p^*(g_{i+1}) = \langle t_{i+1}\left( \xi, \nu_{0}, \ldots ,\nu_{i} \right), \lusim{A}^{p^*}_{g_{i+1}}  \rangle $$
		and as before,  $ {p^*}^{\frown} \langle \nu_{i+1} \rangle \geq^* q\left( \nu_{i+1} \right)$. Thus $ {p^*}^{\frown} \langle \nu_{i+1} \rangle $ forces that $ \langle \nu_{i+2}, \ldots, \nu_{k} \rangle \in \lusim{T}\left( \nu_{i+1} \right) $ and therefore, 
		$$\left({p^*}^{\frown} \langle \nu_{i+1} , \nu_{i+2}, \ldots, \nu_{k} \rangle\right) \restriction_{ \nu_{k}+1 } \Vdash \left({p^*}^{\frown} \langle \nu_{i+1} , \nu_{i+2}, \ldots, \nu_{k} \rangle\right)\setminus \left( \nu_k+1 \right) \in e\left( \xi, \nu_0, \ldots, \nu_{k} \right)  $$
		as desired.
	\end{proof}
	Let us prove that the above claim completes the proof of the Multivariable Fusion Lemma. Let $ \xi<\kappa $. By applying the claim repeatedly, the set $ e(\xi) $ is $ \leq^* $-dense open, where $ e\left( \xi \right) $ is defined as follows:
	
	\begin{align*}
		e\left( \xi\right) = &\{  q\in P\setminus \left( \xi+1 \right)  \colon \mbox{there exists a C-tree } T \mbox{ above } \langle \xi\rangle  \mbox{ such that, for every }\\
		&\langle \nu_{0}, \ldots, \nu_{k} \rangle\in T, \mbox{ which is admissible for  } q \mbox{ above } \langle \xi\rangle,\\
		&\left(q^{\frown} \langle \nu_{0}, \ldots, \nu_{k} \rangle\right) \restriction_{\nu_{k}+1} \Vdash \left(q^{\frown} \langle \nu_{0}, \ldots, \nu_{k} \rangle\right)\setminus \left(\nu_{k}+1\right)\in e\left( \xi, \vec{\nu} \right)
		\}
	\end{align*}
	Thus, given a condition $ p\in P_{\kappa} $, there exists $ p^*\geq^* p $ and a club $ C\subseteq \kappa $, such that, for every $ \xi\in C $,
	$$ p^*\restriction_{\xi+1}\Vdash p^*\setminus \left(\xi+1\right)\in e(\xi) $$
	In particular, $ p^*\restriction_{\xi+1} $ forces that there exists a $ P_{\xi+1} $-name for a $ C $-tree $ \lusim{T}(\xi) $ above $ \langle \xi \rangle $, such that for every $ \langle \nu_0, \ldots, \nu_{k} \rangle\in \lusim{T}(\xi) $ which is admissible for $ p^*\setminus \left( \xi+1 \right) $ above $ \xi $, 
	$$  {\left({p^*\setminus \left( \xi+1 \right)}\right)^{\frown} \langle \nu_0, \ldots, \nu_{k} \rangle  }\restriction_{\nu_{k}+1} \Vdash {\left({p^*\setminus \left( \xi+1 \right)}\right)^{\frown} \langle \nu_0, \ldots, \nu_{k} \rangle  }\setminus \left( \nu_k +1 \right) \in e\left( \xi, \nu_0, \ldots, \nu_{k} \right) $$
	Now, we can construct in $ V $ the $ C $-tree $ T $ as desired in the formulation of the lemma, such that $ \mbox{Succ}_{T}\left( \langle \rangle \right)  =C $, and, for every $ \xi\in C $, $ T_{\langle \xi \rangle} $ is a tree in $ V $ which is forced by $ p^*\restriction_{\xi+1} $ to be contained in $ \lusim{T}(\xi) $. Then $ p^*, T $ are a desired.
\end{proof}

\begin{remark} \label{Remark: NS, if p* in G then }
The condition $ p^* $ and the $ C $-tree $ T $, obtained from the Multivariable Fusion Lemma, can be assumed to satisfy the following property: For every $ i<k $, $ \langle \xi, \nu_0, \ldots, \nu_i \rangle\in T $ which is admissible for $ p^* $ , and for every $ \nu_{i+1}\in \mbox{Succ}_{T}\left(  \xi, \nu_0, \ldots, \nu_i \right) $,
$$ {p^*}^{\frown} \langle \xi, \nu_0, \ldots, \nu_i \rangle\restriction_{g_{i+1}\left( \xi, \nu_0, \ldots, \nu_{i} \right) } \parallel \langle \xi, \nu_0, \ldots, \nu_i, \nu_{i+1} \rangle \mbox{ is admissible for } p^*  $$
this requires a minor change in the definition of the set $ e\left( \xi, \nu_0, \ldots, \nu_{i}  \right) $, which is adding the above as requirement (the same proof provided shows that this additional requirement holds).

Thus, if we apply the standard density argument and choose the condition $ p^* $ provided by the Multivariable Fusion Lemma inside $ G $, it follows that--
$$ \{ \xi <\kappa \colon \langle \xi, \vec{\mu}(\xi) \rangle\in T \ \mbox{is admissible for } p^* \mbox{ and } {p^*}^{\frown}\langle \xi, \vec{\mu}(\xi) \rangle\in G   \}\in W $$
Indeed, note first that $ X = \{ \xi <\kappa \colon \langle \xi, \vec{\mu}(\xi) \rangle\in T  \}\in W $ by claim \ref{Claim: NS, a C tree contains the real values}. Note that if $ Y = \{ \xi\in X \colon \langle \xi, \vec{\mu}(\xi) \rangle \mbox{ is admissible for } p^*  \}\in W $ then $  \{ \xi\in Y \colon {p^*}^{\frown} \langle \xi, \vec{\mu}(\xi) \rangle\in G  \}\in W $, since $ p^*\in G $, by the definition of the functions $ \vec{\mu}(\xi) $.

Thus, it's enough to argue that--
$$  \{ \xi <\kappa \colon \langle \xi, \vec{\mu}(\xi) \rangle\in T \ \mbox{is admissible for } p^*  \}\in W $$
Indeed, we proceed by induction on $ i\leq k $. Assume that--
\begin{align*}
	\{ \xi\in X \colon  {{p^*}^{\frown} \langle \xi, \mu_{\alpha_0}(\xi),\ldots, \mu_{\alpha_i}(\xi) \rangle} \mbox{ is admissible for } p^*   \} \in W
\end{align*}
For every such $ \xi<\kappa $, 
$${{p^*}^{\frown} \langle \xi, \mu_{\alpha_0}(\xi),\ldots, \mu_{\alpha_i}(\xi) \rangle}\restriction_{ g_{i+1}\left( \xi,  \mu_{\alpha_0}(\xi), \ldots, \mu_{\alpha_i}(\xi) \right) }   \parallel  \langle \xi, \mu_{\alpha_0}(\xi), \ldots, \mu_{\alpha_{i+1}}(\xi) \rangle \mbox{ is admissible for } p^* $$
and the decision must be positive for a set of $ \xi $-s in $ W $, since $ {t_{i+1}\left( \mu_{\alpha_0}(\xi), \ldots, \mu_{\alpha_i}(\xi) \right)}^{\frown} \langle \mu_{\alpha_{i+1}}(\xi) \rangle $ is an initial segment of $ g_{i+1}\left( \xi, \mu_{\alpha_0}(\xi), \ldots, \mu_{\alpha_{i}}(\xi) \right) $ for a set of $ \xi $-s in $ W $.  Therefore,
\begin{align*}
	\{ \xi\in X \colon  {{p^*}^{\frown} \langle \xi, \mu_{\alpha_0}(\xi),\ldots, \mu_{\alpha_{i+1}}(\xi) \rangle} \mbox{ is admissible for } p^*   \} \in W
\end{align*}
\end{remark}

We are now ready to prove that $ \lambda $ is measurable in $ M_{\alpha} $, which is property (B) above. 

\begin{lemma} \label{Lemma: NS, lambda is measurable}
$ \lambda $ is measurable in $ M_{\alpha} $.
\end{lemma}

\begin{proof}
Assume otherwise. Then it can be assumed that for every $\xi$ and $ \vec{\nu}  $, $ h\left( \xi, \vec{\nu} \right) $ is a non-measurable regular cardinal.
Let $ f \in V\left[G\right] $ be a function such that $ \left[ f \right]_{W} = \lambda $. Let $ \lusim{f} \in V $ be a $ P_{\kappa} $-name such that $ \left(\lusim{f}\right)_G = f $. Similarly, let $ \lusim{\vec{\mu}} \in V$ be the sequence of $ P $--names $ \langle \lusim{\mu}_{\alpha_0}(\xi), \ldots, \lusim{\mu}_{\alpha_k}(\xi) \rangle $ described above. In $ M\left[H\right] $, 
$$ \left[f\right]_{W} < j_{W}\left( h \right)\left(  \kappa, \vec{\mu} \right) $$
and thus we can assume that there exists a condition $ p\in G $ such that, for every $ \xi<\kappa $, 
$$p\Vdash \lusim{f}(\xi)< h\left( \xi, \lusim{\vec{\mu}}(\xi) \right) $$
From now on we work above $ p $. We can also assume that $ p $ forces, for every $ 0\leq i \leq k $, that $ \lusim{\mu}_{\alpha_i}(\xi) $ is the $ n_i+1 $-th element in the Prikry sequence of $ g_i\left( \xi, \lusim{\mu}_{\alpha_0}(\xi),\dots, \lusim{\mu}_{\alpha_{i-1}}(\xi) \right) $.

Apply the Multivariable Fusion Lemma. For every $ \langle \xi, \vec{\nu} \rangle = \langle \xi, \nu_0, \ldots, \nu_{k} \rangle $, let--
\begin{align*}
	e\left( \xi, \vec{\nu} \right) = &\{  r\in P\setminus \left(\nu_k+1\right)  \colon \exists \alpha< h\left( \xi, \vec{\nu} \right) , \ r\Vdash \lusim{f}(\xi) < \alpha \}
\end{align*}
Since $ h\left( \xi, \vec{\nu} \right) $ is regular and non-measurable, and by corollary \ref{Corollary: NS, every name for an ordinal is decided by a direct extension up to boundedly many values}, $ e\left( \xi, \vec{\nu} \right) $ is $ \leq^* $-dense open above conditions which force that $ \vec{\mu}(\xi)  = \vec{\nu} $.

 Let $ p^*\geq^* p $ and $ T $ be a $ C $-tree, such that, for every $ \langle \xi, \nu_0, \ldots, \nu_k \rangle\in T $ which is admissible for $ p^* $,
$$ \left({p^*}^{\frown} \langle \xi, \vec{\nu} \rangle \right)\restriction_{\nu_k+1} \Vdash \exists \alpha< h\left( \xi, \vec{\nu}\right) , \ \left( {p^*}^{\frown} \langle \xi, \vec{\nu} \rangle \right)\setminus \left( \nu_k+1 \right) \Vdash \lusim{f}(\xi) < \alpha $$
We can assume that $ p^* \in G $, by applying the same argument above any condition which extends $ p $. For every $ \langle \xi, \vec{\nu} \rangle\in T $, pick a $ P_{\nu_k+1} $-name $ \lusim{\alpha}\left( \xi, \vec{\nu} \right) $ for the above $ \alpha $.

Given $ \langle \xi, \nu_0,\ldots, \nu_k \rangle\in T $ which is admissible for $ p^* $, let-- 
$$ \delta\left( \xi, \vec{\nu} \right) = \sup\{ \gamma < h\left( \xi, \vec{\nu} \right) \colon \exists r\geq {p^*}^{\frown} \langle \xi, \vec{\nu} \rangle \restriction_{\nu_k+1} , \ r\Vdash \lusim{\alpha}\left( \xi, \vec{\nu} \right) = \gamma \} $$

Note that $
\delta\left( \xi, \vec{\nu} \right) < h\left( \xi, \vec{\nu} \right) $ since the forcing $ P\restriction_{\nu_k+1} $ has cardinality strictly below $ h\left( \xi, \vec{\nu} \right) $ (we can assume that $ h\left( \xi, \vec{\nu} \right) > \left|\nu_{k}\right|^{+}$ since $ \lambda > \mu_{\alpha_k}^{+} $. The latter can be easily verified since $ k_{\alpha} $ maps $ \mu_{\alpha_k} $, and its successor, to themselves, and $ \lambda = \mbox{crit}\left( k_{\alpha} \right) $ ). It follows that for every $ \langle \xi, \vec{\nu} \rangle\in T $ which is admissible for $ p^* $,
$$ {p^*}^{\frown} \langle \xi, \vec{\nu} \rangle \Vdash \lusim{f}(\xi)< \delta\left( \xi, \vec{\nu} \right) $$
and the mapping $ \langle \xi, \vec{\nu} \rangle \mapsto \delta\left( \xi, \vec{\nu} \right) $ lies in $ V $. 

Apply remark \ref{Remark: NS, if p* in G then }, and let us assume that--
$$ \{ \xi <\kappa \colon \langle \xi, \vec{\mu}(\xi) \rangle\in T \mbox{ is admissible for } p^* \mbox{ and } {p^*}^{\frown} \langle \xi, \vec{\mu}(\xi) \rangle\in G \}\in W $$
For every $ \xi $ in the above set, $ f(\xi)< \delta \left( \xi, \vec{\mu}(\xi) \right) $ holds in $ V\left[G\right] $. Thus, in $ M\left[H\right] $,
$$ \lambda = \left[f\right]_W < \left[ \xi \mapsto \delta\left( \xi, \vec{\mu}(\xi) \right) \right]_W = k_{\alpha}\left(  j_{\alpha}  
\left(  \langle \xi, \vec{\nu} \rangle\mapsto \delta\left( \xi, \vec{\nu} \right) \right) \left( \kappa, \mu_{\alpha_0}, \ldots, \mu_{\alpha_k} \right)
   \right) $$
But this is a contradiction since $ \lambda = \mbox{crit}\left(k_\alpha\right) $ and--
$$ j_{\alpha}  \left(  \langle \xi, \vec{\nu} \rangle\mapsto \delta\left( \xi, \vec{\nu} \right) \right) \left( \kappa, \mu_{\alpha_0}, \ldots, \mu_{\alpha_k} \right) < j_{\alpha}\left( h \right)\left( \kappa, \vec{\mu} \right)= \lambda $$
\end{proof}

\begin{lemma} \label{Lemma: NS, lambda appears in the Prikry sequence of it's image}
Denote $ \lambda^* = k_{\alpha}(\lambda) $. Then $ \lambda $ appears in the Prikry sequence of $ \lambda^* $ in $ M\left[H\right] $.
\end{lemma}

\begin{proof}
In $ M\left[H\right] $, denote by $ t_{\lambda} $ the finite initial segment of the Prikry sequence of $ \lambda^* $, which contains all the elements strictly below $ \lambda $. By modifying the nice sequence $ \langle \alpha_0, \ldots, \alpha_k \rangle $, we can assume that there exists a function $\langle \xi, \vec{\nu} \rangle \mapsto t_{\lambda}\left( \xi, \vec{\nu} \right) $ in $ V $, such that $ t_{\lambda} = j_{\alpha}\left(\langle \xi, \vec{\nu} \rangle \mapsto t_{\lambda}\left( \xi, \vec{\nu} \right)\right)\left( \kappa, \vec{\mu} \right) $. Assume that $ t_\lambda $ has length $ n^*<\omega $.

Define (in $ V\left[G\right] $) a function $\xi \mapsto \lambda(\xi) $ with domain $ \kappa $, such that for each $ \xi<\kappa $, $ \lambda(\xi) $ is the $ \left(n^*+1 \right)$-th element in the Prikry sequence of $ h\left( \xi, \mu_{\alpha_0}(\xi),\ldots, \mu_{\alpha_k}(\xi) \right) $. Clearly $ \left[\xi\mapsto \lambda(\xi)\right]_{W} \geq \lambda $, as it is the first element which appears after $ t_{\lambda} $ in the Prikry sequence of $ \lambda^* $. Thus, it suffices to prove that for every $ \eta < \left[\xi \mapsto \lambda(\xi) \right]_{W} $, $ \eta <\lambda $. 

Assume that $ f\in V\left[G\right] $ is a function such that $  \eta = \left[ f \right]_{W} < \left[\xi \mapsto \lambda\left({\xi}\right)\right]_{W} $. Assume that for every $ \xi<\kappa $, $ f(\xi) < \lambda(\xi) $. Let $ p\in G $ be a condition which forces this. 

Let us apply the Multivariable Fusion Lemma. For every $ \langle \xi, \vec{\nu} \rangle = \langle \xi, \nu_0, \ldots, \nu_{k} \rangle $, let--
\begin{align*}
e\left( \xi, \vec{\nu} \right) = \{&  r\in P\setminus \left(\nu_k+1\right) \colon \exists \alpha< h\left( \xi, \vec{\nu} \right), \ r\Vdash  \mbox{ if } t_{\lambda}\left( \xi, \vec{\nu} \right) \mbox{ is an initial segment} \\ &\mbox{of the Prikry sequence of } h\left( \xi, \vec{\nu} \right), \mbox{ then }  \lusim{f}(\xi) < \alpha  \} 
\end{align*}
We argue that $ e\left( \xi, \vec{\nu} \right) $ is $ \leq^* $ dense above every condition which forces that $ \lusim{\vec{\mu}}(\xi) = \vec{\nu} $. Let $ p\in P\setminus \left( \nu_k+1 \right) $ be such a condition. Denote for simplicity $ h = h\left( \xi, \vec{\nu} \right) $. First, direct extend $ p\restriction_{\left( \nu_{k}+1, h \right)} $ such that it decides whether $ t_{\lambda}\left( \xi, \vec{\nu} \right) $ and $ t^{p}_{h} $ are compatible:
\begin{enumerate}
	\item If $ h\notin \mbox{supp}\left( p \right) $, direct extend $ p $ such that $ t^{p}_{h} = t_{\lambda}\left( \xi, \vec{\nu} \right) $.
	\item If $ t_{\lambda}\left( \xi, \vec{\nu} \right) $ and $ t^{p}_{h} $ are incompatible, pick $ \alpha =0 $.
	\item If $ t^{p}_{h} $ is a strict initial segment of $ t_{\lambda}\left( \xi, \vec{\nu} \right) $, direct extend by replacing $ \lusim{A}^{p}_{h} $ with $ \lusim{A}^{p}_{h}\setminus \max\left( t_{\lambda}\left( \xi, \vec{\nu} \right)   \right)+1 $. Then take $\alpha =0 $.
	\item If $ t_{\lambda}\left( \xi, \vec{\nu} \right) $ is strictly an initial segment of $ t^{p}_{h} $, direct extend $ p^*\restriction_{\left( \nu_{k}+1, h \right)} \geq^* p\restriction_{\left( \nu_{k}+1, h \right)} $ such that for some $ \alpha <h $, $ p^*\restriction_{\left( \nu_{k}+1, h \right)} $ forces that the $ \left(n^*+1\right) $-th element of $ t^{p}_{h} $ is bounded by $ \alpha$. It will follow that $ {p^*\restriction_{h}}^{\frown} p\setminus h \Vdash \lusim{\lambda}(\xi) < \alpha $.
\end{enumerate}
Let us assume that $ p $ has already been direct extended as above, and $ p\restriction_{h}\Vdash t^{p}_{h} = t_{\lambda}\left( \xi, \vec{\nu} \right) $. Direct extend $ p^*\setminus \left(h+1\right) \geq^* p\setminus \left(h+1\right) $ such that--
$$ p\restriction_{h+1} \Vdash \exists \delta< h, \ p^*\setminus \left(h+1\right) \Vdash \lusim{f}(\xi) = \delta$$
Since it is forced that $ f(\xi)< \lambda(\xi) $, $ p\restriction_{h} $ forces that for every $ \alpha\in \lusim{A}^{p}_{h} $ there exists an ordinal $ \delta_{\alpha} < \alpha $ and a set $ B_{\alpha} $ such that-- 
$$ \langle { t_{\lambda}\left( \xi, \vec{\nu} \right)  }^ {\frown} \langle \alpha \rangle , B_{\alpha}  \rangle^{\frown} p^*\setminus \left( h+1 \right) \Vdash \lusim{f}(\xi) = \delta_{\alpha}  $$
Thus, there exists a set $ B\in \lusim{U}^{*}_{h} $, $ B\subseteq \lusim{A}^{p}_{h} \cap \left( \triangle_{ \alpha< h }  B_{\alpha} \right) $,  and an ordinal $ \delta<h $, such that for every $ \alpha\in B $, $ \delta_{\alpha} = \delta $. Direct extend $ p^*(h)\geq^* p(h) $ such that $ \lusim{A}^{p^*}_{h} = B $. Finally, direct extend $ p^*\restriction_{h} \geq^* p\restriction_{h} $ such that, for some $ \alpha < h $ (in $ V^{P_{\nu_{k}+1}}$), $p^*\restriction_{h} \Vdash \lusim{\delta} < \alpha$. Thus, $ p^*\geq^* p $ forces that $ \lusim{f}(\xi) < \alpha $.

Now, fix $ p^*\in G $ and a $ C $-tree $ T $ such that for every $ \langle \xi, \vec{\nu}\rangle\in T $,
\begin{align*}
{\left({ p^*}^{\frown} \langle \xi, \vec{\nu} \rangle\right)}\restriction_{\nu_k+1} \Vdash & \exists \alpha<h\left( \xi, \vec{\nu} \right), \mbox{ if } t_{\lambda}\left( \xi, \vec{\nu} \right) \mbox{ is an initial segment of the Prikry} \\ & \mbox{sequence of } h\left( \xi, \vec{\nu} \right) \mbox{ then }  \left({ p^*}^{\frown} \langle \xi, \vec{\nu} \rangle\right)\setminus \left( \nu_k+1 \right)\Vdash \lusim{f}(\xi)<\alpha
\end{align*}
Let $ \lusim{\alpha}\left( \xi, \vec{\nu} \right) $ be a name for the above $ \alpha $, and define--
$$ \delta\left( \xi, \vec{\nu} \right) = \sup\{  \gamma< h\left( \xi, \vec{\nu} \right) \colon \exists r\geq {p^*}^{\frown} \langle \xi, \vec{\nu} \rangle \restriction_{\nu_k+1} , \ r\Vdash \lusim{\alpha}\left( \xi, \vec{\nu} \right) = \gamma \} $$
as before, $ \delta\left( \xi, \vec{\nu} \right) < h\left( \xi, \vec{\nu} \right) $. 

Finally, work in $ V\left[G\right] $. As before,
$$ \{ \xi < \kappa \colon \langle \xi, \vec{\mu}(\xi) \rangle\in T \mbox{ is admissible for } p^* \mbox{ and }  {p^*}^{\frown} \langle \xi, \vec{\mu}(\xi) \in G \rangle \} \in W $$
Moreover, 
$$ \{ \xi<\kappa \colon t_{\lambda}\left( \xi, \vec{\mu}(\xi) \right)  \mbox{ is an initial segment of the Prikry sequence of } h\left( \xi, \vec{\mu}(\xi) \right)\}\in W $$
Thus, in $ M\left[H\right] $,
$$ \left[f\right]_W < \left[ \xi \mapsto \delta\left( \xi, \vec{\mu}(\xi) \right) \right]_W = k_{\alpha} \left( j_{\alpha} \left(  \langle \xi, \vec{\nu} \rangle \mapsto \delta\left( \xi, \vec{\nu} \right) \right)\left( \kappa, \vec{\mu} \right) \right)$$
but $ j_{\alpha} \left(  \langle \xi, \vec{\nu} \rangle \mapsto \delta\left( \xi, \vec{\nu} \right) \right)\left( \kappa, \vec{\mu} \right)< \lambda $ since $ \delta\left( \xi, \vec{\nu} \right) < h\left( \xi, \vec{\nu} \right) $ for every $ \xi, \vec{\nu} $. Thus, in $ M\left[H\right] $, $ \eta = \left[f\right]_W < \lambda $, as desired.\\
\end{proof}

Let us denote  $U_{\lambda} = \{  X\subseteq \lambda \colon \lambda \in k_{\alpha}\left( X \right) \}\cap M_{\alpha} $. This is an $ M_{\alpha} $-ultrafilter. We will eventually prove that $ \lambda = \mu_{\alpha} $, and then $ U_{\lambda} = U_{\mu_{\alpha}} $ will be the $ M_{\alpha} $-ultrafilter which is used to form $ M_{\alpha+1} $  in the iterated ultrapower.

\begin{lemma} \label{Lemma: NS, U lambda in M alpha even if U notin V}
	$ U_{\lambda} \in M_{\alpha} $. Moreover, it is a normal measure of Mitchell  order $ 0 $ there.
\end{lemma}

\begin{proof}
	The proof follows from a pair of claims.
	\begin{claim}
		There exist $ p\in G $ and a set $ \mathcal{F}\in M_{\alpha} $ of normal measures on $ \lambda $, each of Mitchell order $ 0 $, such that $ \left| \mathcal{F} \right|< \lambda $ and 
		$ {j_{\alpha}(p)}^{\frown}\langle \kappa, \vec{\mu} \rangle \restriction_{\lambda} \Vdash  j_{\alpha}\left( \lusim{\mathcal{U}} \right)\left( \lambda \right) \in \mathcal{F} $.
	\end{claim}
	
	\begin{proof}
		In $ V $, for every measurable $ x<\kappa $, let $ S_x $ be an enumeration of all the normal measures on $ x $ of order $ 0 $.
		
		We claim that there exists $ p\in G $ and a $ C $-tree $ T $, such that for every $ \langle \xi, \vec{\nu} \rangle\in T $ which is admissible for $ p $, there exists a set of ordinals $ A\left( \xi, \vec{\nu} \right) $ with $ \left|A\left( \xi, \vec{\nu} \right)\right|< h\left( \xi, \vec{\nu} \right) $, such that--
		$$ p^{\frown} \langle \xi, \vec{\nu} \rangle\restriction_{h\left( \xi, \vec{\nu} \right)}  \Vdash \lusim{U}_{h\left( \xi, \vec{\nu} \right) } \in {\left(S_{h\left( \xi , \vec{\nu} \right)}\right)}'' A\left( \xi, \vec{\nu} \right) $$ 
		This follows from the Multivariable Fusion Lemma. Fix $ \langle \xi, \vec{\nu} \rangle = \langle \xi, \nu_0, \ldots, \nu_k \rangle $ and denote for simplicity $ h =h\left( \xi, \vec{\nu} \right) $. Consider--
		\begin{align*}
		e\left( \xi, \vec{\nu} \right)  =  \{  &r\in P\setminus \nu_k+1 \colon \mbox{ there exists a set of ordinals } A \mbox{ with } \left|A\right| < h \\
		&\mbox{ such that } r\restriction_{h} \Vdash \lusim{U}_{h}\in {S_h}'' A \}
		\end{align*}
		Let us argue that $ e\left( \xi, \vec{\nu} \right) $ is $ \leq^* $-dense open above conditions which force that $ \lusim{\mu}(\xi) = \vec{\nu} $. Let $ p $ be such a condition. Note that every condition in $ P_h $, and $ p\restriction_{h} $ in particular, forces that there exists an ordinal $ \alpha $ such that $ \lusim{U}_h = S_h(\alpha) $; Now, direct extend $ p^*\restriction_{h} \geq^* p\restriction_{h} $ such that for some $ A $ of cardinality less then $ h $, $ p^*\restriction_{h} \Vdash \lusim{\alpha}\in A $.
		
		Now pick $ p\in G $ and a $ C $-tree $ T $ as above. Then for every $ \langle \xi, \vec{\nu} \rangle\in T $ which is admissible for $ p $, 
		\begin{align*}
		 &{{p}^{\frown} \langle \xi, \vec{\nu} \rangle}\restriction_{\nu_k+1}  	\Vdash \mbox{ there exists a set of ordinals } A \mbox{ with } \left| A \right|< h\left( \xi, \vec{\nu} \right), \\   
		 &\mbox{such that } {{p}^{\frown} \langle \xi, \vec{\nu} \rangle}\restriction_{ \left( \nu_k, h\left( \xi, \vec{\nu} \right) \right) } \Vdash \lusim{U}_{h\left( \xi, \vec{\nu} \right)} \in {S_{h\left( \xi, \vec{\nu} \right)}}''A 
		\end{align*}  
		For every such $ \langle \xi, \vec{\nu} \rangle\in T $, let $ \lusim{A}\left( \xi, \vec{\nu} \right) $ be a $ P_{\nu_k+1} $-name for $ A $ above, and let--
		$$ A^*\left( \xi, \vec{\nu} \right) = \{ \gamma \colon \exists r\geq {p}^{\frown} \langle \xi, \vec{\nu} \rangle, \  r\Vdash \gamma \in \lusim{A}\left( \xi, \vec{\nu} \right) \} $$
		Then $ \left|A^*\left( \xi, \vec{\nu} \right)\right|< h\left( \xi, \vec{\nu} \right) $, and--
		$$ p^{\frown} \langle \xi, \vec{\nu} \rangle\restriction_{ h\left( \xi, \vec{\nu} \right) } \Vdash \lusim{U}_{h\left( \xi, \lusim{\vec{\mu}}(\xi) \right)} \in {S_{h\left( \xi, \lusim{\vec{\mu}}(\xi) \right)}}'' A^*\left( \xi, \lusim{\vec{\mu}}(\xi) \right)  $$ 
		
		Let  $A^* = j_{\alpha}\left( \langle \xi, \vec{\nu} \rangle \mapsto A^*\left( \xi ,\vec{\nu} \right)  \right) $. Denote $ \mathcal{F} = \left(\left(j_{\alpha}(S)\right)_{\lambda}\right)''{A^*}$. Then $ \left|A^*\right| <\lambda$ and thus $ \left|\mathcal{F}\right| <\lambda$.  $ {{j_W(p)}^{\frown} \langle \kappa, \vec{\mu} \rangle}\restriction_{k_{\alpha} \left(\lambda\right)} $ forces that $ j_W\left( \lusim{\mathcal{U}} \right)\left( k_{\alpha}\left( \lambda \right) \right) \in k_{\alpha}\left( \mathcal{F} \right) $.  Thus, by elementarity of $ k_{\alpha} $, $ {{j_{\alpha}(p)}^{\frown} \langle \kappa, \vec{\mu} \rangle}\restriction_{\lambda} $ forces that $ j_{\alpha}\left( \lusim{\mathcal{U}} \right)(\lambda)\in \mathcal{F} $.
	\end{proof}
	
	\begin{claim}
		Assume that $ B\in U_{\lambda} $. Then there exists $ p\in G $ such that $ {j_{\alpha}(p)}^{\frown} \langle \kappa, \vec{\mu} \rangle \restriction_{\lambda} \Vdash B\in j_{\alpha}\left( \lusim{\mathcal{U}} \right)(\lambda) $.
	\end{claim}
	
	\begin{proof}
		Let $ \langle \xi, \vec{\nu} \rangle\mapsto B\left( \xi, \vec{\nu} \right) $ be a function in $ V $ such that-- 
		$$B = j_{\alpha}\left(  \langle \xi, \vec{\nu} \rangle\mapsto B\left( \xi, \vec{\nu} \right) \right)\left( \kappa, \vec{\mu} \right) $$ 
		(we assumed, without loss of generality, that $ B $ can be represented using $ \vec{\mu} $; else, change $ \vec{\mu} $). 
		Let $ n^*<\omega $ be the coordinate in which $ \lambda $ appears in the Prikry sequence of $ \lambda^* $. In $ V\left[G\right] $, denote by $ \lambda(\xi) $ the $ n^* $-th element in the Prikry sequence of $ h\left( \xi, \vec{\mu}(\xi) \right) $, so that $ \left[ \xi \mapsto \lambda(\xi) \right]_W = \lambda $. 
		
		As usual, we apply the Multivariable Fusion Lemma. Given $ \langle \xi, \vec{\nu} \rangle $, let-- 
		\begin{align*}
		e\left( \xi, \vec{\nu} \right) =\{ &r\in P\setminus \nu_k+1 \colon 
		r\restriction_{ h\left( \xi, \vec{\nu} \right) } \mbox{ decides whether }B\left( \xi, \vec{\nu} \right)\in \lusim{U}_{h\left( \xi, \vec{\nu} \right)}, \mbox{ and there}\\
		&\mbox{exists a bounded subset } A\subseteq h\left( \xi, \vec{\nu} \right) \mbox{ such that the following holds:}\\
		&\mbox{If } r\restriction_{  h\left( \xi, \vec{\nu} \right)}\Vdash B\left( \xi, \vec{\nu} \right)\in \lusim{U}_{h\left( \xi, \vec{\nu} \right)}, \  \lusim{\lambda}\left( \xi \right) \in A\cup B\left( \xi, \vec{\nu} \right); \mbox{ else, } \\
		&\lusim{\lambda}\left( \xi \right) \in A\cup \left(h\left( \xi, \vec{\nu} \right)  \setminus B\left( \xi, \vec{\nu} \right) \right) \}
		\end{align*}
		$ e\left( \xi, \vec{\nu} \right) $ is $ \leq^* $ dense open above any condition which forces that $ \lusim{\vec{\mu}}(\xi) = \vec{\nu} $. Indeed, let $ p\in P\setminus \nu_k+1 $ be such a condition. Denote $ h = h\left( \xi, \vec{\nu} \right) $. Direct extend $ p^*\restriction_{h} \geq^* p\restriction_{h} $ such that it decides the length of $ t^{p}_{h} $ and which of the sets $B\left( \xi, \vec{\nu} \right)$, $   h\setminus B\left( \xi, \vec{\nu} \right) $ belongs to $ \lusim{U}_{h} $:
		\begin{enumerate}
			\item If the length of $ t^{p}_{h} $ is $\geq n^* $, direct extend $p^*\restriction_{h} \geq^* p\restriction_{h} $ such that for some bounded subset $ A\subseteq \lambda $, $ p^*\restriction_{h} $ forces that the $ n^* $-th element in the Prikry sequence of $ h $ belongs to $ A $.
			\item Otherwise, $ t^{p}_{h} = t_{\lambda}\left( \xi, \vec{\nu} \right) $. In this case, direct extend and shrink $ \lusim{A}^{p}_{h} $ such that it is entirely contained in exactly one of the sets $B\left( \xi, \vec{\nu} \right)$, $ h\setminus  B\left( \xi, \vec{\nu} \right) $.
		\end{enumerate}
	The condition $ p^* $ obtained this way is as desired.
	
	Now pick $ p\in G $ and a $ C $-tree $  T $ such that for every $ \langle \xi, \vec{\nu} \rangle\in T $ which is admissible for $ p $,
	\begin{align*}
	&{{p}^{\frown} \langle \xi, \vec{\nu} \rangle}\restriction_{\nu_k+1}   	\Vdash  {{p}^{\frown} \langle \xi, \vec{\nu} \rangle}\restriction_{ \left( \nu_k, h\left( \xi, \vec{\nu} \right) \right) } \mbox{ decides whether } B\left( \xi, \vec{\nu} \right)\in \lusim{U}_{h\left( \xi, \vec{\nu} \right)}, \\
	&\mbox{and there exists a bounded subset } A\subseteq h\left( \xi, \vec{\nu} \right) \mbox{such that} \\
	&{{p}^{\frown} \langle \xi, \vec{\nu} \rangle}\restriction_{ \left( \nu_k, h\left( \xi, \vec{\nu} \right) \right) } \Vdash \lusim{\lambda}(\xi) \mbox{ belongs to exactly one of the sets } A\cup B\left( \xi, \vec{\nu} \right)  \\ &\mbox{or } A\cup \left( h\left( \xi, \vec{\nu} \right) \setminus B\left( \xi, \vec{\nu} \right)  \right), \mbox{ according to the above decision.}
	\end{align*}   
	For every such $ \langle \xi, \vec{\nu} \rangle\in T $, let $ \lusim{A}\left( \xi, \vec{\nu} \right) $ be a $ P_{\nu_k+1} $-name for $ A $ above, and let--
	$$ A^*\left( \xi, \vec{\nu} \right) = \{ \gamma \colon \exists r\geq {p}^{\frown} \langle \xi, \vec{\nu} \rangle, \  r\Vdash \gamma \in \lusim{A}\left( \xi, \vec{\nu} \right) \} $$
	Then $ A^*\left( \xi, \vec{\nu} \right)$ is a bounded subset of $ h\left( \xi, \vec{\nu} \right) $.
	
	We argue that $ {j_{\alpha}(p)}^{\frown} \langle \kappa, \vec{\mu} \rangle\Vdash B\in j_{\alpha}\left( \lusim{\mathcal{U}} \right)(\lambda) $. Work in $ V\left[G\right] $. Then for a set of $ \xi $-s in $ W $, $ \langle \xi, \vec{\mu}(\xi) \rangle\in T $ is admissible for $ p $. Thus, 
		$$ p^{\frown} \langle \xi, \vec{\mu}(\xi) \rangle \restriction_{ h\left( \xi, \vec{\mu}(\xi) \right) } \parallel B\left( \xi, \vec{\mu}(\xi) \right)\in \lusim{U}\left( h\left( \xi, \vec{\mu}(\xi)\right) \right) $$
		We argue that for a set of $ \xi $-s in $ W $, 
		$$p^{\frown} \langle \xi, \vec{\mu}(\xi) \rangle \restriction_{ h\left( \xi, \vec{\mu}(\xi) \right) } \Vdash B\left( \xi, \vec{\mu}(\xi) \right)\in \lusim{U}\left( h\left( \xi, \vec{\mu}(\xi)\right) \right) $$
		Assume otherwise. Then-- 
		$$ \{ \xi<\kappa \colon \lambda(\xi)\in A\left( \xi, \vec{\mu}(\xi) \right) \cup \left( h\left( \xi, \vec{\nu} \right) \setminus B\left( \xi, \vec{\mu}(\xi) \right)\right)  \}\in W $$
		However, this cannot hold:
		\begin{enumerate}
		\item If $ \{ \xi<\kappa \colon \lambda(\xi)\in A\left( \xi, \vec{\mu}(\xi) \right) \}\in W $, then, since $ \left| A(\xi, \vec{\nu}) \right| < h\left( \xi, \vec{\nu} \right) $ for every $\xi, \vec{\nu} $, it follows that $ \lambda \in \mbox{Im}\left( k_{\alpha} \right) $, which is a contradiction.
		\item Else, $ \{ \xi<\kappa \colon \lambda(\xi) \in h\left( \xi, \vec{\mu}(\xi) \setminus B\left( \xi, \vec{\mu}(\xi) \right) \right) \} $. But then $ \lambda \notin k_{\alpha}\left( B \right) $, contradicting the fact that $ B\in U_{\lambda} $. 
		\end{enumerate}	
		Thus, $ j_W(p)^{\frown} \langle \kappa, \vec{\mu} \rangle \restriction_{ k_{\alpha}\left( \lambda \right) } \Vdash k_{\alpha}(B)\in j_W\left( \lusim{\mathcal{U}} \right)\left( k_{\alpha}\left( \lambda \right) \right) $ and by elementarity of $ k_{\alpha} $, $ {j_{\alpha}(p)}^{\frown} \langle \kappa, \vec{\mu} \rangle \restriction_{ \lambda } \Vdash B\in j_\alpha\left( \lusim{\mathcal{U}} \right)\left( \lambda  \right) $ , as desired.
	\end{proof}

	Fix now a set $ \mathcal{F} $ and a condition $ p\in G $ as in the first claim. Since $\mathcal{F}$ is a sequence of normal measures on $ \lambda $ of cardinality $ <\lambda $, there exists a partition $ \langle B_F \colon F\in \mathcal{F} \rangle $ of $ \lambda $ such that for every $ F\in \mathcal{F} $, $ B_{F}\in F $. $ \left|\mathcal{F}\right|<\lambda $, and thus there exists a unique $ F^*\in \mathcal{F} $  such that $ \lambda \in k_{\alpha}\left( B_{F^*} \right) $. We denote for simplicity $ B^* = B_{F^*} $.
	
	By second claim, applied for the set $ B^*\in U_{\lambda} $, there exists  $ p^*\in G $ above $ p $ such that $ {j_{\alpha}(p^*)}^{\frown} \langle \kappa, \vec{\mu} \rangle  \Vdash j_{\alpha}\left( \lusim{\mathcal{U}} \right)(\lambda) = F^* $. 
	
	Finally, $ F^* = U_{\lambda} $ follows. Indeed, let $ X\in F^* $. By the second claim, for every $ X\in U_{\lambda} $, there exists $ p\in G $ such that $ {j_{\alpha}(p)}^{\frown} \langle \kappa, \vec{\mu} \rangle  \Vdash X\in j_{\alpha}\left( \lusim{\mathcal{U}} \right)(\lambda)  $. Without loss of generality, $ p $ extends $ p^* $ which was chosen in the previous paragraph, and thus $j_{\alpha}(p)\Vdash X\in F^* $. Since $ X $ and $ F^* $ are elements of $ M_{\alpha} $ (and not names), it follows that $ X\in F^* $.
\end{proof}

\begin{corollary} \label{Corollary: NS, connection between measures along the way and measures in the iteration}
In $ M\left[H\right] $, $j_W\left( \mathcal{U} \right)\left( k_{\alpha}(\lambda) \right) = k_{\alpha}\left( U_{\lambda} \right)$. In particular, if $ \mathcal{U}\in V $, then $ j_{\alpha}\left( \mathcal{U} \right)(\lambda) = U_{\lambda} $.
\end{corollary}

\begin{proof}
This follows since, by the proof of the previous lemma, there exists $ p\in G $  such that $ {j_{\alpha}(p)}^{\frown} \langle \kappa, \vec{\mu} \rangle\Vdash j_{\alpha}\left( \lusim{\mathcal{U}} \right)(\lambda) = U_{\lambda} $. Now apply $ k_{\alpha} \colon M_{\alpha}\to M $ and use the fact that $ {j_W(p)}^{\frown} \langle \kappa, \vec{\mu} \rangle\in H $.
\end{proof}

\begin{lemma}
In $ V $, $\mbox{cf}(\lambda) \geq \kappa^{+}$.
\end{lemma}

\begin{proof}
Denote $ M' = \mbox{Ult}\left( M_{\alpha}, U_{\lambda} \right) $, and let $ j' \colon V \to  M'  $, be defined as follows:
$$ j' =  j^{M_{\alpha}}_{  U_{\lambda} }  \circ j_{\alpha} $$
There exists an elementary embedding $ k'\colon M' \to M $, defined as follows:
$$ k'\left( j'(f)\left( \kappa, \mu_{i_0},\ldots, \mu_{i_m}, \lambda \right)  \right) = j_{W}\left( f \right) \left( \kappa, \mu_{i_0},\ldots, \mu_{i_m}, \lambda \right) $$
for every $ f\in V $ and $ i_0<\ldots < i_m < \alpha $.

Since $ U_{\lambda} $ was derived from $ k_{\alpha} $, $ k'\colon M' \to M $ is elementary (the proof is the same as in lemma \ref{Lemma: NS, k alpha is elementary}). It's not hard to verify that $ \mbox{crit}\left( k' \right) > \lambda $. Therefore, $ \lambda $, which is a non-measurable inaccessible cardinal in $ M' $, is still a non-measurable inaccessible cardinal in $ M $.

Let us argue that $ \lambda $ is regular in $ M\left[H\right] $. Split $ H = H_{\lambda} * H' $, where $ H_{\lambda} \subseteq j_{W}(P)\restriction_{\lambda} $. If $ \lambda $ changes its cofinality in $ M\left[H\right] $, then it changes its cofinality in $ M\left[H_{\lambda}\right] $ (since the upper forcing has a direct extension order which is more than $ \lambda $--closed). However, by corollary \ref{Corollary: NS, every new function is dominated by an old one}, $\lambda$ is regular in $ M\left[H_{\lambda}\right] $.

It follows that, in $ V\left[G\right] $, $ \mbox{cf}(\lambda) \geq \kappa^{+} $. Thus, in $ V $, $ \mbox{cf}\left( \lambda  \right) \geq \kappa^+ $.
\end{proof}

\begin{corollary}
$\mbox{crit}\left( k_{\alpha} \right) =  \mu_{\alpha}$. 
\end{corollary}

\begin{proof}
It suffices to prove that $\mbox{crit}\left( k_{\alpha} \right) \geq   \mu_{\alpha} $. Denote $ \bar{\mu} = \sup\{ \mu_\beta\colon \beta<\alpha \} $. We already argued that $ \mbox{crit}(k_{\alpha}) \geq \bar{\mu} $.

By all the properties proved so far, $ \mbox{crit}\left(k_{\alpha} \right) $ is a measurable cardinal in $ M_{\alpha} $, with cofinality $ >\kappa $ in $ V $. By the definition, $ \mu_{\alpha} \geq \bar{\mu} $ is the least such cardinal. Thus, $ \mbox{crit}\left( k_{\alpha} \right) \geq \mu_{\alpha} $.
\end{proof}

This finishes the inductive proof of properties (A)-(E). We are now prepared to finish the proof of Theorem \ref{Theorem: NS, full description of j_W restricted to V}:

\begin{proof}[Proof of theorem \ref{Theorem: NS, full description of j_W restricted to V}]
Recall that $ \kappa^* = j_{U}(\kappa) $. It's not hard to prove by induction that, for every $ \alpha<\kappa^* $, $ \mu_{\alpha}<\kappa^* $. Note that $ j_{\kappa^*}(\kappa) = \kappa^* $, since $ j_{\kappa^*} $ is an iterated ultrapower with measures on measurables below $ \kappa^*$. Since $ \kappa^* $ is measurable in each step, it does not move in $ j_{1,\kappa^*} \colon M_U \to M_{\kappa^*} $.

Recall the embedding $ k_{\kappa^*} \colon M_{\kappa^*} \to M $, defined as follows:
$$ k_{\kappa^*} \left(  j_{\kappa^*}(f)\left( \kappa, \mu_{i_0},\ldots, \mu_{i_m} \right) \right) =  j_{W} \left( f \right) \left( \kappa, \mu_{i_0},\ldots, \mu_{i_m} \right)  $$
for every $ f\in V $, $ m<\omega $ and $ i_0,\ldots, i_m < \kappa^* $. As in lemma \ref{Lemma: NS, k alpha is elementary}, $ k_{\kappa^*} $ is elementary, $ k_{\kappa^*}\circ j_{\kappa^*} = j_{W}\restriction_{V} $ and $ \mbox{crit}\left( k_{\kappa^*} \right) \geq \kappa^* $.

In order to prove that $ M =M_{\kappa^*} $, $ j_{\kappa^*} = j_{W}\restriction_{V} $ and $ k^* = j_{W}(\kappa) $, it suffices to prove that $ k_{\kappa^*} \colon M_{\kappa^*} \to M $ is the identity. Thus, it suffices to prove that for every ordinal $ \eta $, $ \eta \in \mbox{Im}\left( k_{\kappa^*} \right) $. Assume that $ g\in V\left[G\right] $ is a function such that $ \eta = \left[g\right]_{W} $. Let $ p\in G $ be a condition. By lemma \ref{Lemma: NS, every new function is evaluated by F in V}, there exists a condition $p\leq  p^*\in G $, a function $ \xi \mapsto A_{\xi} $ in $ V $ and a club $ C\subseteq \kappa $ such that, for every $ \xi\in C $, $ \left|A_{\xi}\right|< \kappa $, and $ p^* \Vdash \lusim{g}(\xi) \in A_{\xi} $. Then $ j_{W}\left(p^*\right) \in H $ forces that $ \eta = \left[\xi \mapsto \lusim{g}(\xi)\right]_{W} \in \left[\xi \mapsto A_{\xi}\right]_{W} = k_{\kappa^*}\left( j_{\kappa^*} \left( \xi\mapsto A_{\xi} \right) (\kappa) \right) $;  but $ \left|j_{\kappa^*}\left( \xi\mapsto A_{\xi} \right)(\kappa)\right| < j_{\kappa^*}(\kappa) = \kappa^* \leq \mbox{crit}\left( k_{\kappa^*} \right) $. Therefore, $ \eta \in \mbox{Im}\left( k_{\kappa^*} \right) $ as desired.

Finally, note that if $ \mathcal{U}\in V $, then by corollary \ref{Corollary: NS, connection between measures along the way and measures in the iteration}, $ U_{\mu_{\alpha}} = j_{\alpha}\left( \mathcal{U} \right)\left( \mu_{\alpha} \right) \in M_{\alpha} $ for every $ \alpha<\kappa $, and thus the iteration $ j_{\kappa^*} $ is definable over $ V $. Also, $ M = M_{\kappa^*} $ is a class of $ M $. 
\end{proof}

We finish this section with several remarks about definability of $ j_W\restriction_{V} $ in $ V $. 

The condition $ \mathcal{U}\in V $ is sufficient but not necessary for the definability of $ j_W\restriction_{V} $. For instance, let $ \eta\in \Delta $ be the first measurable. Assume that, in $ V $, there are infinitely many measurables which carry $\eta$ measures of Mitchell order $ 0 $. Take in $ V $ an enumeration $ \langle \alpha_n \colon n<\omega \rangle $ of the first $ \omega $ such measurables above $ \eta $. For every $ n<\omega $, let $ \langle F^{\xi}_{\alpha_n} \colon \xi< \eta \rangle $ be an enumeration of $ \eta $-many measures of Mitchell order $ 0 $ on $ \alpha_n $. Let $ P $ be the forcing notion which uses, at stage $ \alpha_n $, the unique normal measure which extends $ F^{ \eta_n }_{\alpha_n} $, where $ \eta_n <\eta $ is the $ n $-th element in the Prikry sequence of $ \eta $. For every other measurable $ \alpha $, use a measure which extends the least measure on $ \alpha $ of Mitchell order $ 0 $ with respect to a prescribed well order of $ V_\kappa $. So $ \mathcal{U}\notin V $, since it codes the Prikry sequence of $ \eta $. However, $ j_W\restriction_{V} $ is definable in $ V $, by repeating the argument of corollary \ref{Corollary: NS, connection between measures along the way and measures in the iteration}, replacing $ \mathcal{U} $ with $ \mathcal{U}\restriction_{ \Delta \setminus \{ \alpha_n \colon n<\omega \} } = \langle U_{\alpha} \colon \alpha\in \Delta\setminus \{ \alpha_n \colon n<\omega \} \rangle \in V $. More generally, the following holds, ind is proved similarly to corollary \ref{Corollary: NS, connection between measures along the way and measures in the iteration}:

\begin{lemma} \label{Lemma: NS, a weaker sufficient condition for definability}
Assume that for some $ \xi<\kappa $, $ \mathcal{U}\setminus \xi =  \langle U_{\alpha} \colon \alpha\in \Delta\setminus \xi \rangle \in V $. Then $ j_W\restriction_{V} $ is definable in $ V $.
\end{lemma}

\begin{remark}
Let $ A\subseteq \Delta $ be a set such that, for every $ \alpha<\kappa^* $, $ \mu_{\alpha}\in j_{\alpha}\left(A\right) $. If $ \mathcal{U}\restriction_{A} =  \langle U_{\alpha} \colon \alpha\in A \rangle \in V $, then $ j_W\restriction_{V} $ is definable in $ V $, and again, this is proved by repeating the argument of corollary \ref{Corollary: NS, connection between measures along the way and measures in the iteration}, replacing $ \mathcal{U} $ with $ \mathcal{U}\restriction_{A} $. This seems like an improvement of the previous lemma; however, we will prove in lemma \ref{Lemma: NS, every measurable of M is the image of some mu alpha}, that a set $ A $ satisfies that $ \mu_{\alpha}\in j_{\alpha}\left(A\right) $ for every $ \alpha<\kappa^* $,  if and only if, for some $ \xi<\kappa $, $ \Delta\setminus \xi \subseteq A $.
\end{remark}

By lemma \ref{Lemma: NS, a weaker sufficient condition for definability}, definability of $ j_W\restriction_{V} $ in $ V $ follows from the assumption that $ j_W\left( \mathcal{U} \right)\setminus \kappa \in M $. In the next section we will prove that the other direction is not necessarily true.

%
%
%

\section{A General Analysis Of Iterated Ultrapowers}

Our main goal in this section is to simplify the presentation of $ j_W\restriction_{V} $ provided in the previous section; for instance, we will provide a simpler characterization of the critical points $ \mu_{\alpha} $. Simultaneously, we describe in detail how the Prikry sequences, added to measurables of $ M $ above $ \kappa $, look like: up to a finite initial segment, those are sequences of critical points of an iterated ultrapower, generated over some finite sub-iteration of $ \langle M_{\alpha} \colon \alpha<\kappa^* \rangle $, using a single measure. It will follow that every Prikry sequence, added in $ M\left[H\right] $ for a measurable cardinal above $ \kappa $, already belongs to $ V $. 

Our goals are lemma \ref{Lemma: NS, every measurable above mu bar has cofinility above kappa in V} and corollaries \ref{Corollary: NS, better definition of mu_alpha}, \ref{Corollary: NS, how Prikry sequences look} and \ref{Corollary: NS, best definition of mu_alpha}.

We start by studying linear iterations of $ V $ in more general settings. Let us assume that $ \kappa^* $ is an ordinal, and $ \langle M_{\alpha } \colon \alpha< \kappa^* \rangle $ is a linear iteration of $ V $, by normal measures of Mitchell order $ 0 $. More specifically, we assume that $ M_0 =\mbox{Ult}\left( V, U \right) $ where $ U $ is a measure of Mitchell order $ 0 $ on some measurable $ \kappa $; in successor steps, $ M_{\alpha+1} = \mbox{Ult}\left( M_{\alpha}, U_{\mu_{\alpha} } \right) $, where $ U_{\mu_{\alpha}}\in M_{\alpha} $ is a normal measure of order $ 0 $ on some measurable $ \mu_{\alpha} $; at limit steps a direct limit is taken. We assume also that the iteration is normal in the sense that $ \langle \mu_{\alpha} \colon \alpha<\kappa^* \rangle $ is increasing. We do not assume that the entire iteration is definable in $ V $. Finally, we denote $ M = M_{\kappa^*} $.

We begin by observing that every finite nice sequence corresponds to a finite iteration of $ V $ which naturally embeds in $ M $. Assume that $ \langle \alpha_0, \ldots, \alpha_m \rangle $ is a nice sequence below some ordinal $ \alpha<\kappa^* $. Recall that this means that, every $ 0\leq k \leq m $, there are functions $ g_k, F_k \in V $ such that--
$$ \mu_{\alpha_{k}} = j_{\alpha_k}\left( g_k \right) \left( \kappa, \mu_{\alpha_0}, \ldots, \mu_{\alpha_{k-1}} \right) $$
$$ U_{\mu_{\alpha_{k}}} = j_{\alpha_k}\left( F_k \right) \left( \kappa, \mu_{\alpha_0}, \ldots, \mu_{\alpha_{k-1}} \right) $$
(for $ k=0 $, $ \mu_{\alpha_0} = j_{\alpha_0}\left( g_{0} \right)(\kappa)  $  and $ U_{\mu_{\alpha_0}} = j_{\alpha_0}\left( F_0 \right)\left( \kappa \right)  $ ). 

We define a finite iteration $ \langle N_k \colon k\leq m+1 \rangle $ of $ V $, for each $ k\leq m $ an embedding $ i_k \colon V\to N_k $, a cardinal $ \lambda_k $ measurable in $ N_k $ and a measure $ W_k\in N_k $ on it of order $ 0 $.

First, let $ N_0 \simeq \mbox{Ult}\left( V,U \right) $, $ i_0 \colon V\to M_0 $ the ultrapower embedding,  $ \lambda_0 = i_0\left( g_0 \right)(\kappa) $ and $ W_0 = i_0\left( F_0 \right)\left( \kappa \right) $.

Assume that $ k\leq m $ and $ N_k $ has been defined. Let $ N_{k+1} \simeq \mbox{Ult}\left( N_k, W_k \right) $, $ i_{k+1} \colon V\to N_{k+1} $, $ i_{k+1} = j^{N_k}_{W_k} \circ i_k $, $\lambda_{k+1} = i_{k+1}\left( g_{k+1} \right)\left( \kappa, \lambda_0, \ldots, \lambda_k \right)$ and $ W_{k+1} = i_{k+1}\left( F_{k+1} \right)\left( \kappa, \lambda_0, \ldots, \lambda_k \right) $.

\begin{lemma}
Fix a nice sequence $ \langle \alpha_0, \ldots, \alpha_m \rangle $ below some $ \alpha< \kappa^* $. In the above notations, define $ k_{m+1} \colon N_{m+1}\to M_{\alpha}  $ as follows:
$$ k_{m+1}\left(  i_{m+1}\left( f \right)\left( \kappa, \lambda_0, \ldots, \lambda_m  \right) \right) = j_{\alpha}(f)\left( \kappa, \mu_{\alpha_0}, \ldots, \mu_{\alpha_m} \right) $$
for every $ f\in V $. Then $ k_{m+1} \colon N_{m+1}\to M_{\alpha}  $ is an elementary embedding, and--
$$ k_{m+1} = \left(j_{\alpha_m+1, \alpha} \circ \ldots \circ j_{\alpha_0+1, \alpha_1}  \circ  j_{0,\alpha_0}\right)\restriction_{ N_{m+1} } $$
\end{lemma}

\begin{remark}
The iteration $j_{\alpha_m+1, \alpha} \circ \ldots \circ j_{\alpha_0+1, \alpha_1}  \circ  j_{0,\alpha_0}$ above is not necessarily internal to $ N_{m+1} $; this means that the sub-iterations $ j_{\alpha_i+1, \alpha_{i+1}} $ participating in it are iterated ultrapowers as defined over $ M_{\alpha_{i}+1} $. In the proof of the lemma we will show that the external iteration $ j_{\alpha_m+1, \alpha} \circ \ldots \circ j_{\alpha_0+1, \alpha_1}  \circ  j_{1,\alpha_0} $ is well defined over $ N_{m+1} $, in the sense that for every $x\in  N_{m+1} $, $ \left(j_{\alpha_{i-1}+1, \alpha_i} \circ \ldots \circ j_{\alpha_0+1, \alpha_1}  \circ  j_{0,\alpha_0}\right)(x) $ belongs to $ M_{\alpha_{i}+1} $. Later in this section, we will prove that such an iteration might be an internal iteration of $ N_{m+1} $, provided that the initial nice sequence is chosen more carefully.
\end{remark}

\begin{proof}
We proceed by induction on $ m $. The induction basis is given for $ "m=-1" $, namely, the case where the given nice sequence below $ \alpha $ is empty. In this case, $ N_0 = \mbox{Ult}\left( V, U \right) $, $ i_0 = j_U \colon V\to N_0 $ and $ k_0\left( i_0(f)(\kappa)  \right) = j_{\alpha}(f)(\kappa) $, and clearly $ k_0 = j_{0,\alpha} \colon M_0 \to M_{\alpha} $. 

Assume now that $ m<\omega $ and $ k_{m+1} \colon N_{m+1}\to M_{\alpha_{m+1}} $ has been constructed (here, the embedding $ k_{m+1} $ corresponds to the nice sequence $ \langle \alpha_0 ,\ldots, \alpha_m \rangle $ below $ \alpha_{m+1} $). Let us argue that $ k_{m+2}  =  j_{\alpha_{m+1}+1, \alpha} \circ k_{m+1} \restriction_{N_{m+2}} $. Indeed, given an arbitrary element $ i_{m+2}(f)\left( \kappa, \lambda_0, \ldots, \lambda_m, \lambda_{m+1} \right) $ of $ N_{m+2} $, 
\begin{align*}
&k_{m+1}\left(  i_{m+2}(f)\left( \kappa, \lambda_0, \ldots, \lambda_m, \lambda_{m+1} \right) \right) \\
=\  &k_{m+1}\left(  j^{N_{m+1}}_{W_{m+1}  } \left( i_{m+1} (f) \right) \left(  \kappa, \lambda_0, \ldots, \lambda_{m+1} \right) \right) \\
=\  & j_{U_{\mu_{\alpha_{m+1}}}}\left(  j_{\alpha_{m+1}} (f) \right) \left( \kappa, \mu_{\alpha_0}, \ldots, \mu_{\alpha_m} , \mu_{\alpha_{m+1}} \right)\\
=\  & j_{\alpha_{m+1}+1}(f)\left( \kappa, \mu_{\alpha_0}, \ldots, \mu_{\alpha_m} , \mu_{\alpha_{m+1}}  \right)
\end{align*}
where we used the fact that $ W_{m+1} = i_{m+1}\left( F_{m+1} \right)\left( \kappa, \lambda_0, \ldots, \lambda_m \right) $ and $ \lambda_{m+1} = i_{m+1}\left( g_{m+1} \right)\left( \kappa, \lambda_0, \ldots, \lambda_m \right) $ for the computation on their values under $ k_{m+1} $. Finally, apply $ j_{\alpha_{m+1}+1, \alpha} $ on both sides.
\end{proof}

If the sequence $ \langle \alpha_0, \ldots, \alpha_m \rangle $ below $ \alpha $ is clear from the context, we denote $ N^* = N_{m+1} $, $ i^* = i_{m+1} \colon V\to N^* $ and $ k^* = k_{m+1} \colon N^*\to M_{\alpha} $. Note that $ k^* $ is not necessarily an internal iteration of $ N^* $. Indeed, assume that $ \lambda_0 < \mu_{\alpha_0} $ (this happens, e.g., if $ \alpha_0 = 1 $. In this case, $ \lambda_0 = \mu_{0} $ and $ \mu_{1}  = j_{U_{\mu_0}}\left( \mu_{0} \right) > \mu_{0}  $). If $ k^* $ was an internal iteration of $ N^* $, then $ \lambda_0 $ would have to be one of the critical points participating in the iteration, since $ \lambda_0 $ is inaccessible in $ N^* $ and $ k^*\left( \lambda_0 \right) = \mu_{\alpha_0} $. However, this is not possible because $ \lambda_0 $ is not measurable in $ N^* $.

Our goal is lemma \ref{Lemma: NS, every measurable above mu bar has cofinility above kappa in V}. In the proof, it will be useful to consider a nice sequence $ \langle \alpha_0, \ldots, \alpha_{m} \rangle $ below $ \alpha $ and its associated iteration $ N^* $, such that the embedding $ k^* \colon N^*\to M_{\alpha} $ is an internal iteration of $ N^* $. This will require a more sophisticated choice of the initial nice sequence. The example from the last paragraph offers a lead: we would like $ \lambda_k = \mu_{\alpha_k} $ to hold for every $ 0\leq k \leq m $.

\begin{lemma}\label{Lemma: NS, nice sequence factors to finite iteration and iterated ultrapwer}
$k^*$ is an internal iteration of $ N^* $ if and only if, for every $ 0\leq k \leq m $, $ \lambda_k = \mu_{\alpha_k} $.
\end{lemma}

\begin{proof}
Let us assume first that $ k^* $ is an iteration of $ N^* $. Then $ \lambda_k $ is a non-measurable inaccessible in $ N^* $, and thus $ \lambda_k $ cannot move by $ k^* $.  So $ \mu_{\alpha_k} =  k^*\left( \lambda_{k} \right) = \lambda_{k} $.

Let us concentrate on the other direction. Assume that $ \lambda_k = \mu_{\alpha_k} $ for every $ 0\leq k\leq m $. $ \lambda_0 = \mu_{\alpha_0} $ is measurable in $ M_U = N_0$, and thus $ j_{0,\alpha_0}\left( \mu_{\alpha_0} \right) = \mu_{\alpha_0} $. Also,   $ j_{0,\alpha_0}\left( W_0 \right) = U_{\mu_{\alpha_0}} $. Note that--
$$ j_{\alpha} = j_{\alpha_0+1, \alpha} \circ j_{ U_{\mu_{ \alpha_0 }  } } \circ j_{\alpha_0} = j_{\alpha_0+1,\alpha} \circ j^{N_1}_{0,\alpha_0}  \circ  j_{W_0} \circ j_U $$
where $ j^{N_1}_{0,\alpha_0} $ is the iterated ultrapower consisting of the same measures as $ j_{0,\alpha_0} $, but acting on $ N_1 $. $ j^{N_1}_{0,\alpha_0} \colon N_1 \to M_{\alpha_0+1} $ is internal to $ N_1 $, and so is $ j_{\alpha_0+1,\alpha} \circ  j^{N_1}_{0,\alpha_0} $.

We proceed now by induction on $ m $. Assume that $ k_{m+1} \colon N_{m+1}\to M_{\alpha_{m+1}} $ is an internal iteration of $ N_{m+1} $ (with respect to the nice sequence $ \langle \alpha_0, \ldots, \alpha_m \rangle $ below $ \alpha_{m+1} $). Then--
$$ U_{\mu_{ \alpha_{ m+1 } }   } = k_{m+1}\left( W_{m+1} \right) $$
and thus--
\begin{align*}
 j_{ \alpha_{m+1}+1 }=\  & j_{  U_{\mu_{ \alpha_{ m+1 } }   } } \circ k_{m+1} \circ i_{m+1} \\
=\ & k_{m+1}\restriction_{ N_{m+2} }  \circ i_{m+2} \\
=\ & \left( j_{ \alpha_m+1, \alpha_{m+1} } \circ \ldots \circ j_{\alpha_0+1, \alpha_1} \circ j_{0,\alpha_0}  \right)\restriction_{N_{m+2}} \circ i_{m+2}
\end{align*}
Where $ \left( j_{ \alpha_m+1, \alpha_{m+1} } \circ \ldots \circ j_{\alpha_0+1, \alpha_1} \circ j_{0,\alpha_0}  \right)\restriction_{N_{m+2}}  $ above is an internal iteration of $ N_{m+2} $, since $ W_{m+1} $ is a measure over $ \lambda_{m+1} = \mu_{\alpha_{m+1}} $, and lies strictly above all the participating critical points. Thus, the embedding $ k_{m+2} $, obtained by applying $ j_{\alpha_{m+1}+1, \alpha} $ on $ \left( j_{ \alpha_m+1, \alpha_{m+1} } \circ \ldots \circ j_{\alpha_0+1, \alpha_1} \circ j_{0,\alpha_0}  \right)\restriction_{N_{m+2}}  $, is an internal iteration of $ N_{m+2} $.
\end{proof}

\begin{lemma} \label{Lemma: NS, completing sequences}
Every nice sequence $ \langle \alpha_0, \ldots, \alpha_m \rangle $ below $ \alpha $ can be completed to a nice sequence-- $$ \langle \alpha^{0}_0, \ldots, \alpha^{n_0}_{0}, \alpha^{0}_1, \ldots, \alpha^{n_1}_{1}, \ldots \ldots , \alpha^{0}_m, \ldots , \alpha^{n_m}_{m} \rangle $$
where $ \alpha^{n_0}_{0} = \alpha_0, \alpha^{n_1}_{1} = \alpha_1 , \ldots, \alpha^{n_m}_{m} = \alpha_m $, such that the embedding $ k^* $ associated to the latter sequence is an iteration of $ N^* $. 
\end{lemma}

\begin{proof}
We begin with an arbitrary nice sequence $ \langle \alpha_0, \ldots, \alpha_m \rangle $, and complete it to a nice sequence--
$$ \langle \alpha^{0}_0, \ldots, \alpha^{n_0}_{0}, \alpha^{0}_1, \ldots, \alpha^{n_1}_{1}, \ldots \ldots , \alpha^{0}_m, \ldots , \alpha^{n_m}_{m} \rangle $$
where $ \alpha^{n_0}_{0} = \alpha_0, \alpha^{n_1}_{1} = \alpha_1 , \ldots, \alpha^{n_m}_{m} = \alpha_m $.

We first extend the sequence below $ \alpha_0 $, namely define $ \alpha^{0}_{0}, \ldots , \alpha^{n_0}_{0} $.

Denote $ N^{0}_{0} = \mbox{Ult}\left(  V ,U\right) $ and $ i^{0}_{0} = j_U \colon V \to  N^{0}_{0} $. Let $ \lambda^{0}_0 = i^{0}_0\left( g_0 \right)(\kappa) $. Let $ \alpha^{0}_0 \leq \alpha_0 $ be the first such that $ \lambda^{0}_0 \leq \mu_{\alpha^{0}_0} $. $ \left( \mbox{cf}\left( \lambda^{0}_0 \right) \right)^V > \kappa $, so actually $ \lambda^{0}_0 = \mu_{\alpha^{0}_0} $. If $ \alpha^{0}_0 = \alpha_0 $, we set $ n_0=0 $ and we are done extending the sequence below $ \alpha_0 $. Assume otherwise. 

Work in $ N^{0}_{0} $ and define there $ W^{0}_{0} = i^{0}_{0}\left( \mathcal{U} \right)\left( \lambda^{0}_{0} \right) $. Let $ N^{1}_{0} = \mbox{Ult}\left( N^{0}_{0}, W^{0}_{0} \right) $ and $ i^{1}_{0} = j^{N^{0}_{0}}_{W^{0}_{0}} \circ i^{0}_{0} \colon V\to N^{1}_{0} $. Define $ k^{1}_{0} \colon  N^{1}_{0} \to M_{\alpha^{0}_0+1} $ to be such that for every $ f\in V $, 
$$ k^{1}_{0}\left(  i^{1}_{0} \left( f \right)\left( \kappa, \lambda^{0}_{0} \right) \right) = j_{\alpha^{0}_0+1} \left( f \right)\left( \kappa, \mu_{\alpha^{0}_{0} }\right)$$
by lemma \ref{Lemma: NS, nice sequence factors to finite iteration and iterated ultrapwer}, $ k^{1}_{0} $ is an iterated ultrapower of $ N^{1}_{0} $. The measures participating in this iteration lie on measurables below $ \alpha^{0}_{0} $ (actually, $ k^{1}_{0} = j^{N^{1}_{0}}_{1, \alpha^{0}_{0}} $).
In $ N^{1}_{0} $, let $ \lambda^{1}_{0}= j_{W^{0}_{0}}\left( \lambda^{0}_{0} \right)  $, and note that $ \lambda^{1}_{0} $ is measurable in $ N^{1}_{0} $ above $ \lambda^{1}_{0} $. Thus, $ \lambda^{1}_{0} $ does not participate in the iteration $ k^{1}_{0} $, namely $ k^{1}_{0}\left( \lambda^{1}_{0} \right) = \lambda^{1}_{0} $. So $ \lambda^{1}_{0} $ is a measurable cardinal in $ M_{\alpha^{0}_{0} +1} $, and $ \left( \mbox{cf}\left( \lambda^{1}_{0} \right)  \right)^V > \kappa $. Thus, there exists an index $ \alpha^{1}_{0}  $, such that $ \mu_{ \alpha^{1}_{0}} = \lambda^{1}_{0} $ and $ \alpha^{0}_{0} < \alpha^{1}_{0} \leq \alpha_{0} $. If $ \alpha^{1}_{0} = \alpha_0 $, we finish extending the sequence below $ \alpha_0 $ and set $ n_0 = 1 $. Assume otherwise. Define in $ N^{1}_{0} $ the measure $ W^{1}_{0} = i^{1}_{0}(\mathcal{U})\left( \lambda^{1}_{0} \right) $. Let $ N^{2}_{0} = \mbox{Ult}\left( N^{1}_{0}, W^{1}_{0} \right) $ and $ i^{2}_0 = j^{N^{1}_{0}}_{W^{1}_{0}} \circ i^{1}_{0} \colon V\to N^{2}_{0} $. Define $ k^{2}_{0} \colon N^{2}_{0} \to M_{\alpha^{1}_{0} +1} $ in the natural way, namely, for every $ f\in V $,
$$  k^{2}_{0}\left( i^{2}_{0}(f)\left( \kappa, \lambda^{0}_{0} , \lambda^{1}_{0} \right) \right) = j_{\alpha^{1}_{0}+1}(f)\left( \kappa, \lambda^{0}_{0} , \lambda^{1}_{0} \right) $$
and by \ref{Lemma: NS, nice sequence factors to finite iteration and iterated ultrapwer}, $ k^{2}_{0} $ is an iterated ultrapower of $ N^{2} $ with measurables below $\mu_{\alpha^{1}_{0} }$. Denote $ \lambda^{2}_{0} = j^{N^{1}_0}_{W^{1}_{0} }\left( \lambda^{1}_{0} \right) > \lambda^{1}_{0} $. Arguing as before, $ \lambda^{2}_{0} $ is measurable in $ M_{\alpha^{1}_{ 0}+1 } $  with cofinality above $ \kappa $, and thus, there exists $ \alpha^{2}_0 $ such that $ \lambda^{2}_{0} = \mu_{ \alpha^{2}_0 } $ and $ \alpha^{0}_{0} < \alpha^{1}_{0} < \alpha^{2}_0 \leq \alpha_0 $.

Continue in this fashion, and construct an increasing sequence $ \alpha^{0}_{0}< \alpha^{1}_{0}< \ldots \leq \alpha_0$. We argue that the construction stops after finitely many steps. Assume otherwise, and let $ \langle \alpha^{n}_{0} \colon n<\omega \rangle $ be a strictly increasing sequence of ordinals below $ \alpha_0 $, such that for every $ n<\omega $, 
$$\mu_{\alpha_0} > \mu_{\alpha^{n+1}_0} = \lambda^{n+1}_{0} = j_{W^{n}_0} \left( {\lambda^{n}_{0}} \right) > {\lambda^{n}_0} = \mu_{\alpha^{n}_{0}}  $$
and--
$$ \mu_{\alpha^{n+1}_0} = \lambda^{n+1}_0 = {k^{n+1}_0}\left( \lambda^{n+1}_{0} \right) = {k^{n+1}_{0}}\left( i^{n+1}_0 \left( g_0 \right)(\kappa) \right) = j_{\alpha^{n}_0+1}\left( g_0 \right)(\kappa) = j_{U_{ \mu_{  \alpha^{n}_{0}  } }} \left( \mu_{ \alpha^{n}_0  }  \right)  $$
let $ \alpha^{*}_{0} = \sup\{ \alpha^{n}_{0} \colon n<\omega \} \leq \alpha_0 $. Note that--
$$j_{\alpha^{*}_{0}}(g_0)(\kappa)  =j_{\alpha^{0}_0 , \alpha^{*}_{0}}\left( \mu_{\alpha^{0}_0} \right) =\sup\{ \mu_{\alpha^{n}_{0}} \colon n<\omega \} < \mu_{\alpha^{*}_{0}} $$
and thus--
$$\mu_{\alpha_0} = j_{\alpha_{0}}\left( g_0 \right)(\kappa) = j_{\alpha^{*}_0 , \alpha_0 } \left(  j_{ \alpha^{*}_{0} } \left( g_0 \right)(\kappa)  \right) = j_{ \alpha^{*}_{0} } \left( g_0 \right)(\kappa)   $$
which contradicts the fact that $  j_{ \alpha^{*}_{0} } \left( g_0 \right)(\kappa)  < \mu_{\alpha^{*}_0} \leq \mu_{\alpha_0} $.

Thus, there exists $ n_0 <\omega $ and a sequence $ \alpha^{0}_{0} < \alpha^{1}_{0} < \ldots < \alpha^{n_0}_{0} = \alpha_0 $ such that for every $ n<n_0 $,
$$ \mu_{\alpha^{n+1}_{0}} = j_{U_{ \mu_{ \alpha^{n}_{0} }  }} \left(  \mu_{ \alpha^{n}_{0} } \right) = j_{\alpha^{n+1}_{0}  } \left( g_0 \right)(\kappa) $$
where the last equality follows by induction, since--
$$  j_{\alpha^{n+1}_{0}  }(g_0)(\kappa) = j_{ \alpha^{n}_{0}, \alpha^{n+1}_{0 }   } \left(  j_{\alpha^{n}_0 }(g_0)(\kappa)    \right)  = j_{ \alpha^{n}_{0}, \alpha^{n+1}_{0 }   } \left(  \mu_{\alpha^{ n }_{0}  }   \right) = j_{U_{ \alpha^{n}_{0}   } }\left(  \mu_{\alpha^{ n }_{0}  }   \right) $$
let us justify the last equality in the above equation. If $ \mu_{\alpha_0} $ is not a limit of measurables, then $ \alpha^{n+1}_{0} = \alpha^{n}_{0}+1 $ and the equation is clear. Otherwise, $ \mu_{\alpha^{n}_{0}} $ is a limit of measurables. Therefore $  \mu_{\alpha^{n+1}_{0}} =  j_{U_{ \alpha^{n}_{0} }  } \left(\mu_{ \alpha^{n}_{0} }\right) $  is a limit of measurables, and each factor in $ j_{\alpha ^{n}_{0}+1, \alpha^{n+1}_{0} } $ is an ultrapower embedding with one of them. Thus, each such factor maps $ \mu_{\alpha^{n+1}_0} $ to itself.

This finishes the completion of the initial nice sequence below $ \alpha_0 $. Let $ N^{*}_{0} $ be the iterated ultrapower associated to the nice sequence $ \langle \alpha^{0}_0, \ldots \alpha^{n}_{0} \rangle $, with a corresponding embedding $ i^{*}_{0} \colon V\to N^{*}_{0} $. Let $ k^{*}_{0} \colon N^{*}_{0} \to M_{\alpha_0+1} $ be defined as follows: for every $ f\in V $,
$$ k^{*}_{0}\left( i^{*}_{0}(f)\left( \kappa, \mu_{\alpha^{0}_{0}}, \ldots , \mu_{\alpha^{n_0}_{0}} \right) \right) = j_{\alpha_1}(f)\left( \kappa, \mu_{\alpha^{0}_{0}}, \ldots , \mu_{\alpha^{n_0}_{0}} \right) $$
By lemma \ref{Lemma: NS, nice sequence factors to finite iteration and iterated ultrapwer}, the embedding $ k^{*}_{0} $ is an iterated ultrapower of $ N^{*}_{0} $, and $ j_{\alpha_0+1} = k^{*}_{0} \circ i^{*}_{0} $. All the ultrapowers in $ k^{*}_{0} $ are taken on measurables below $ \alpha_0 $.

Now work over $ N^{*}_{0} $, define $ \lambda^{0}_{1} = i^{*}_{0}\left( g_1 \right)\left( \kappa, \mu_{\alpha_0} \right) $. $ \lambda^{0}_1 >  \mu_{\alpha_0} $ is measurable in $ N^{*}_{0} $ and thus is not moved by $ k^{*}_{0} $. Also, it has cofinality above $ \kappa $ in $ V $. Let $ \alpha^{0}_{1} \leq \alpha_1 $ be such that $\lambda^{0}_1 = \mu_{\alpha^{0}_1} $.  If $ \lambda^{0}_1 = \mu_{\alpha_1} $, we set $ n_1 =0  $ and move on to extend the sequence below $ \alpha_2 $. Assume otherwise. Let $ W^{0}_1 = i^{*}_{0}\left( \mathcal{U} \right)\left( \lambda^{0}_1 \right) $. Let $ N^{0}_1 = \mbox{Ult}\left( N^{*}_{0}, W^{0}_1 \right) $, and $ i^{0}_1 = j^{  N^{*}_{0} }_{ W^{0}_1 } \circ i^{*}_{0} $. Let $ k^{0}_{1} \colon N^{1}_{0} \to M_{\alpha^{0}_1 +1} $ be the natural embedding, and continue the construction as above. It will stop after finitely many steps.

By repeating the same argument for $ \alpha_2, \ldots ,\alpha_m $, we generate the desired completion of $ \langle \alpha_0, \ldots, \alpha_m \rangle $.
\end{proof}

\begin{remark}
For every $0\leq i\leq m $, $ \mu_{\alpha_i} $ appears in the Prikry sequence of $ \mu^{*}_{\alpha_i} = k_{\alpha_i}\left( \mu_{\alpha_i} \right) $ in $ M\left[H\right] $. Note that in the above proof, the completion below $ \mu_{\alpha_i} $, namely the sequence $ \langle \alpha^{0}_k \ldots , \alpha^{n_i}_{i} \rangle $, is a subsequence of the Prikry sequence of $ \mu^{*}_{\alpha_i} $ below $ \mu_{\alpha_i} $. In lemma \ref{Lemma: NS, the element after mu alpha is its image} we will prove that this subsequence is actually a segment in this Prikry sequence.
\end{remark}

\begin{lemma} \label{Lemma: NS, every measurable above mu bar has cofinility above kappa in V}
Assume that $ \alpha<\kappa^* $. Denote $ \bar{\mu}= \sup\{ \mu_{\alpha'} \colon \alpha'<\alpha \} $. Let $ \lambda > \bar{\mu} $ be an inaccessible cardinal in $ M_{\alpha} $. Then $ \left(\mbox{cf}(\lambda)\right)^V > \kappa $.
\end{lemma}

\begin{proof}
Let us first consider the case where there is no $ \beta<\alpha $ and $ \lambda'< \lambda $ such that $ j_{\beta,\alpha}\left( \lambda' \right) = \lambda $. Let $ \langle \alpha_0, \ldots, \alpha_m \rangle $ be a nice sequence below $ \alpha $ such that $ \lambda = j_{\alpha}(h)\left( \kappa, \mu_{\alpha_0}, \ldots, \mu_{\alpha_m} \right) $ for some function $ h\in V $. We can assume that the sequence in complete as in lemma \ref{Lemma: NS, completing sequences}, and so $ k^* \colon N^* \to M_{\alpha} $ is an internal iterated ultrapower. Denote--
$$ \lambda^* = i^*(h)\left( \kappa, \mu_{\alpha_0}, \ldots, \mu_{\alpha_m} \right) $$
and note that $ k^*\left( \lambda^* \right) = \lambda $. It suffices to prove that $ \lambda^* = \lambda $, since every inaccessible above $ \kappa $ in a finite iteration of $ V $ has cofinality $ >\kappa $ in $ V $. Assume that $ \lambda^*< \lambda $. Because $ \lambda^* $ is inaccessible in $ N^* $, $ \lambda^* $ is one of the measurables participating in the iteration $ k $, namely $ \lambda^* = \mu_{\beta} $ for some $ \beta<\alpha $. Since $ \lambda^* > \mu_{\alpha_m} $, $ \beta>\alpha_m $. Then--
\begin{align*}
	\lambda =&\  k^*\left( \lambda^* \right) \\ =& \left( j_{\beta,\alpha} \circ  j_{\alpha_{m}+1, \beta} \circ j_{\alpha_{m-1}+1, \alpha_m} \circ \ldots \circ j_{\alpha_0+1,\alpha_1} \circ j_{1,\alpha_0}\right)\left( \lambda^* \right)\\ =& \  j_{\beta,\alpha}\left( \lambda^* \right) 
\end{align*}
where we used the fact that $ \lambda^*  = \mu_{\beta}$ is inaccessible in $ N^* $ above $ \mu_{\alpha_{m}} $, and thus is fixed by ultrapowers below $ \mu_{\alpha_m} $ and by $ j_{\alpha_m+1,\beta} $. It follows that there exists $ \beta<\alpha $ and $ \lambda^*< \lambda $ such that $ j_{\beta,\alpha}\left( \lambda^* \right) = \lambda $, which is a contradiction.

Let us now take care of the case where, for some $ \beta<\alpha $ and $ \lambda_0< \lambda $, $ j_{\beta,\alpha}\left( \lambda_0 \right) = \lambda $.  Let $ \beta<\alpha $ be the least such that such $ \lambda_0 $ exists. Since $ \lambda_0 $ is inaccessible in $ M_\beta $ and $ \lambda_0< j_{\beta,\alpha}\left( \lambda_0 \right) $, $ \lambda_0 $ is one of the measurables participating in the iteration $ j_{\beta,\alpha} $. Thus,  $ \lambda_0 = \mu_{\gamma_0} $, for some $ \gamma_0 <\alpha $.
%
%
%


 Denote $ \lambda_1 = j_{U_{ \mu_{\gamma_0} }} \left( \mu_{\gamma_0} \right) $. This is an inaccessible cardinal in $ M_{\gamma_0+1} $. Let us argue that $ \left( \mbox{cf}\left( \lambda_1 \right) \right)^V > \kappa $.

Pick a complete nice sequence $ \langle \alpha_0, \ldots, \alpha_m \rangle $ below $ \gamma_0+1 $ such that, for some function $ h\in V $,
$$ \lambda_1 = j_{\gamma_0+1}(h)\left( \kappa, \mu_{\alpha_0}, \ldots, \mu_{\alpha_m} \right) $$
we can assume that $ \alpha_m = \gamma_0 $ (else, add it. The sequence will remain complete since there is no $ \lambda'<\lambda_0 $ and $ \gamma'\leq\gamma_0 $ such that $ j_{\gamma',\gamma_0+1}\left( \lambda' \right) = \lambda_0 $). Let $ N^* $ be the associated finite iteration, with an embedding $ i^* \colon V\to N^* $. let $ k^* \colon N^* \to M_{\gamma_0+1} $ be the corresponding iterated ultrapower such that $ k^*\circ i^* = j_{\gamma_0+1} $. Denote $ \lambda^{*}_{1} = i^*(h)\left( \kappa, \mu_{\alpha_0}, \ldots, \mu_{\alpha_m} \right) $. Then $ k^*\left( \lambda^{*}_1 \right) = \lambda_1 $. Let us argue that $ \lambda^{*}_1 = \lambda_1 $. Assume that $ \lambda^{*}_1 < \lambda_1 $. Then $ \lambda^*_1 $, which is measurable in $ N^* $, is one of the measurables participating in $ k^* $. Note that--
$$\lambda_1 = k^*\left( \lambda^{*}_{1} \right) = j_{\alpha_m+1, \gamma_0+1}\circ  j_{ \alpha_{m-1}+1 , \alpha_m}\circ j_{\alpha_{m-2}+1, \alpha_{m-1}  } \circ \ldots \circ j_{1,\alpha_0} \left( \lambda^{*}_1 \right) $$
but $ \alpha_m = \gamma_0 $, so $ j_{\alpha_m+1, \gamma_0+1} $ is the identity. So $ \lambda^*_1 < \mu_{\alpha_{m}} = \mu_{\gamma_0} $. $ \mu_{\gamma_0} $ is already a non-measurable inaccessible in $ N^* $ (since we started from a complete nice sequence which includes it), and thus $ k^*\left( \mu_{\gamma_0} \right) = \mu_{\gamma_0} $. Namely $j_{U_{\gamma_0}}\left( \mu_{\gamma_0} \right) =\lambda_1= k^*\left( \lambda^*_1 \right)< \mu_{\gamma_0} $, a contradiction.

Thus $ \left( \mbox{cf}(\lambda_1) \right)^V > \kappa $. If $ \lambda_1 = \lambda$, we are done. Else, $ \lambda_1 < \lambda $ is inaccessible in $ M_{\gamma_0+1} $, and is mapped via $ j_{\gamma_0+1, \alpha} $ to $ \lambda $. Hence, arguing as before, $ \lambda_1 \leq \bar{\mu} $ is one of the measurables participating in the iterated ultrapower $ j_{\gamma_0+1, \alpha} $. Therefore, there exists $ \gamma_1 \in \left( \gamma_0, \alpha \right) $ such that $ \lambda_1 = \mu_{\alpha_{\gamma_1}} $. Denote $ \lambda_2 = j_{U_{\mu_{\gamma_1}}}\left( U_{\mu_{\gamma_1}} \right) > \lambda_1 $. As above, $ \left( \mbox{cf}\left( \lambda_2 \right) \right)^V > \kappa $. If $ \lambda_2 = \lambda $, we are done. Assume otherwise, and continue in this fashion.

Let us argue that the process stops after finitely many steps. Assume otherwise. Then we have constructed an $ \omega $-sequence of ordinals below $ \alpha $, $ \langle \gamma_n \colon n<\omega \rangle $, and an increasing sequence--
$$ \lambda_0 = \mu_{\gamma_{0}} < \lambda_1 = \mu_{\gamma_1} < \lambda_2 = \mu_{\gamma_2} <\ldots < \lambda $$
such that, for every $ n<\omega $, $ \lambda_{n+1} = \mu_{\gamma_{n+1}} = j_{U_{\mu_{\gamma_n } }}\left( \mu_{\gamma_n} \right) $. Denote $ \gamma^* = \sup\{ \gamma_n \colon n<\omega \} $ (possibly $ \gamma^* = \alpha $). Let $ \lambda^* = \sup\{ \lambda_n \colon n<\omega \} $. Then--
$$ j_{\gamma^*, \alpha}\left( \lambda^* \right) = \lambda $$
however, $ j_{\gamma^*, \alpha}\left( \lambda^* \right) = \lambda^* $: if $ \gamma^* = \alpha $ this is clear. Else, note that $ \mu_{\gamma^*}  $ is chosen strictly above $ \sup\{ \mu_{\xi} \colon \xi < \gamma^* \} = \lambda^* $. Therefore, the critical point of $ j_{\gamma^*, \alpha} $ is above $ \lambda^* $, and $ j_{\gamma^*, \alpha}\left( \lambda^* \right) = \lambda^*  $.

It follows that $ \lambda^* = \lambda $. But $ \lambda^* \leq \bar{\mu} $ (equality may hold if $ \gamma^* = \alpha $), contradicting the fact that $ \lambda > \bar{\mu} $.
\end{proof}

We now return to our context, and assume that $ \langle M_{\alpha} \colon \alpha \leq \kappa^* \rangle $ is the iteration described in the previous section, with the same notations. We can first simplify the definition of the critical points $ \mu_{\alpha} $:

\begin{corollary} \label{Corollary: NS, better definition of mu_alpha}
	Assume that $ \alpha < \kappa^* $. Let $ \bar{\mu} = \sup\{ \mu_{\alpha'} \colon \alpha'<\alpha \} $. 
	\\If $ \alpha $ is successor, $ \mu_{\alpha} $ is the first measurable above $ \bar{\mu} $ in $ M_{\alpha} $. 
	\\If $ \alpha $ is limit and $ \left(\mbox{cf}\left( \alpha \right)\right)^V \leq \kappa $, then $ \mu_{\alpha} $ is the first measurable above $ \bar{\mu} $ in $ M_{\alpha} $.
	\\ If $ \alpha $ is limit and $ \left(\mbox{cf}\left( \alpha \right)\right)^V > \kappa $, then $ \mu_{\alpha} $ is the first measurable in $ M_{\alpha} $ which is greater or equal to $ \bar{\mu} $.
\end{corollary}

\begin{proof}
	If $ \bar{\mu} $ is measurable in $ M_{\alpha} $ and $ \left(\mbox{cf}\left( \alpha \right)\right)^V > \kappa $, then $ \left(\mbox{cf}\left( \bar{\mu} \right)\right)^V > \kappa $ and thus $ \mu_{\alpha} = \bar{\mu} $ by the definition. Else, $ \mu_{\alpha} $ is chosen to be the least measurable in $ M_{\alpha} $ above $ \bar{\mu} $ with cofinality above $\kappa$  in $ V $, which is, by the previous lemma, the least measurable above $ \bar{\mu} $ in $ M_{\alpha} $.
\end{proof}

\begin{lemma} \label{Lemma: NS, the element after mu alpha is its image}
	Assume that $ \alpha<\kappa^* $ and $ \lambda $ appears after $ \mu_{\alpha} $ in the Prikry sequence of $ \mu^* = k_{\alpha}\left( \mu_{\alpha} \right) $. Then $ \lambda = j_{U_{\mu_{\alpha}}}\left( \mu_{\alpha} \right) $.
	\end{lemma}
	
\begin{proof}
Since $ j_{U_{\mu_{\alpha}}}\left( \mu_{\alpha} \right) $ is measurable in $ M_{\alpha+1} $ above $ \bar{\mu}_{\alpha+1} = \sup\{ \mu_{\alpha'} \colon \alpha' \leq \alpha \} $, it follows, by lemma \ref{Lemma: NS, every measurable above mu bar has cofinility above kappa in V}, that-- 
$$ \left( \mbox{cf} \left( j_{U_{\mu_{\alpha} }} \left(  \mu_{\alpha} \right) \right) \right)^V >\kappa$$
Thus there exists $ \beta >\alpha $ such that $ j_{U_{\mu_{\alpha} }}\left( \mu_\alpha \right)  = \mu_{\beta}$, and appears in the Prikry sequence of $ k_{\beta}\left( \mu_{\beta} \right) = j_{\beta,\kappa^*}\left( j_{\alpha,\beta}\left( \mu_{\alpha} \right) \right) = \mu^* $.

Let us prove now that $ j_{U_{\mu_{\alpha} }}\left( \mu_{\alpha} \right)  = \mu_{\beta} $ is the immediate successor of $ \mu_{\alpha} $ in the Prikry sequence of $ \mu^* $.

Assume, for contradiction, that $\mu_{\alpha}< \lambda < j_{U_{\mu_{\alpha} }} \left( \mu_{\alpha} \right)$, and $ \lambda $ appears after $\mu_{\alpha}$ in the Prikry sequence of $ \mu^* $. Assume that $ \lambda = j_{\alpha+1}(g)\left( \kappa, \mu_{\alpha_0}, \ldots ,\mu_{\alpha_k}, \mu_\alpha \right) $, for some $ g\in V $ and $ \alpha_0 <\ldots <  \alpha_k < \alpha $. Assume also that $ h\in V $ is a function such that $ \mu_{\alpha} = j_{\alpha}(h)\left( \kappa, \mu_{\alpha_0}, \ldots ,\mu_{\alpha_k} \right) $ for the same $ \alpha_0<\ldots < \alpha_k <\alpha $ (this can always be arranged by changing the sequence $ \langle \alpha_0,\ldots ,\alpha_k \rangle $). Then--
$$ j_{\alpha+1}(g)\left( \kappa, \mu_{\alpha_0}, \ldots, \mu_{\alpha_k} , \mu_{\alpha}\right) < j_{\alpha+1}(h)\left( \kappa, \mu_{\alpha_0}, \ldots , \mu_{\alpha_k} \right) $$
so we may assume that for every $ \xi, \nu_0, \ldots , \nu_k , \eta$, below $ \kappa $, $ g\left( \xi, \nu_0,\ldots , \nu_k ,\eta\right) <  h\left(  \xi, \nu_0, \ldots , \nu_k \right) $. Assume also that $ \mu_{\alpha} $ is the $ n $-th element in the Prikry sequence of $ \mu^* $. In $ V\left[G\right] $, let $ \lambda(\xi)  $ be the $ \left(n+1\right) $-th element in the Prikry sequence of $ h\left( \xi, \vec{\mu}(\xi) \right) $, so that $ \left[\xi \mapsto \lambda(\xi)\right]_W = \lambda $.


Assume that the sequence $ \langle \alpha_0, \ldots, \alpha_{k}  \rangle\subseteq \alpha $ is nice (else, add more coordinates). Now apply the Multivariable Fusion Lemma. For every $ \langle \xi, \vec{\nu} \rangle $, let--
\begin{align*}
	e\left( \xi, \vec{\nu} \right) = \{  &r\in P\setminus \nu_k+1 \colon \mbox{there exists a bounded subset } A\subseteq h\left( \xi, \vec{\nu} \right) \mbox{ such that}\\
	&r \mbox{ forces that the } \left( n+1 \right)\mbox{-th element in the Prikry sequence of } h\left( \xi, \vec{\nu} \right)\\
	&\mbox{belongs either to } A \mbox{ or to the club of closure points of the function }\\
	&\eta\mapsto g\left( \xi, \vec{\nu}, \eta \right) \}
\end{align*}
We argue that $ e\left( \xi, \vec{\nu} \right) $ is $ \leq^* $ dense open above any condition which forces that $ \lusim{\vec{\mu}}(\xi) = \vec{\nu} $. Let $ p\in P\setminus \nu_k+1 $ be such a condition. Denote for simplicity $ h = h\left( \xi, \vec{\nu} \right) $. Direct extend $ p\restriction_{h} $ such that it decides the length of $ t^{p}_{h} $; if the length is $\geq \left( n+1 \right)$, direct extend $ p\restriction_{h} $ further, so that it forces that there exists a bounded subset $ A\subseteq h $ such that the $ \left( n+1 \right) $-th element in the Prikry sequence of $ h $ belongs to it. Finally, shrink $ \lusim{A}^{p}_{h} $ by intersecting with the club of closure points of the function which maps each $ \eta<h $ to $ g\left( \xi, \vec{\nu}, \eta \right) $. The condition obtained this way indeed belongs to $ e\left( \xi, \vec{\nu} \right) $. 

Now, fix $ p\in G $ and a $ C $-tree $ T $ such that for every $ \langle \xi, \vec{\nu} \rangle\in T $ which is admissible for $ p $,
\begin{align*}
&\left(  p^{\frown} \langle \xi, \vec{\nu} \rangle \right)\restriction_{\nu_k+1} \Vdash  \mbox{there exists a bounded subset } A\subseteq h\left( \xi, \vec{\nu} \right) \mbox{ such that}\\
&p^{\frown}\langle \xi, \vec{\nu} \rangle\setminus \left(\nu_k+1\right) \mbox{ forces that the } \left( n+1 \right)\mbox{-th element in the Prikry sequence}
\\&\mbox{of } h\left( \xi, \vec{\nu} \right) \mbox{ belongs either to } A \mbox{ or to the club of closure points of the function }\\
&\eta\mapsto g\left( \xi, \vec{\nu}, \eta \right) \}
\end{align*}

Let $ \lusim{A}\left( \xi, \vec{\nu} \right) $ be a $ P_{\nu_k+1} $-name for the set $ A $ above, and set-- 
$$ A^*\left( \xi, \vec{\nu} \right) = \{ \gamma<h\left( \xi, \vec{\nu} \right) \colon \exists r\geq p^{\frown}\langle \xi, \vec{\nu} \rangle\restriction_{\nu_k+1}, \ r\Vdash \gamma\in \lusim{A}\left(\xi, \vec{\nu} \right) \} $$

It follows that for a set of $ \xi $-s in $ W $, $ \lambda(\xi) $ either belongs to $ A^*\left( \xi, \vec{\mu}(\xi) \right) $ of to the club of closure points of $ \eta \mapsto g\left( \xi, \vec{\mu}(\xi), \eta \right) $. 

However, it cannot hold that for a set of $ \xi $-s in $ W $, $\lambda(\xi)\in  A^*\left( \xi, \vec{\mu}(\xi) \right) $. Indeed assume otherwise. Denote--
$$ A^* = j_{\alpha}\left(  \langle \xi, \vec{\nu} \rangle\mapsto A^*\left( \xi, \vec{\nu} \right)  \right)\left( \kappa, \mu \right) $$ 
then $ A^* $ is bounded in $ \mu_{\alpha} $, and, under the above assumption, 
 $ \lambda \in k_{\alpha}\left(  A^*   \right) = A^* \subseteq \mu_{\alpha} $, which is a contradiction.
 
 Thus, in $ M\left[H\right] $, $ \lambda $ is a closure point of $ \eta\mapsto j_W(g)\left( \kappa, \vec{\mu}, \eta \right) $. Recall that $ \mu_{\alpha}< \lambda $, and thus $ j_W(g)\left( \kappa, \vec{\mu}, \mu_{\alpha} \right) < \lambda = j_{\alpha}(g)\left( \kappa, \vec{\mu}, \mu_{\alpha} \right) $, which is a contradiction.
\end{proof}

\begin{corollary} \label{Corollary: NS, how Prikry sequences look}
Let $ \alpha< \kappa^* $ and denote $ \mu^* = k_{\alpha}\left( \mu_{\alpha} \right) $. Then the Prikry sequence of $ \mu^* $ in $ M\left[
H\right] $ has a final segment of the form--
$$ \langle \mu_{\alpha_0}, \mu_{\alpha_1}, \mu_{\alpha_2}, \ldots, \mu_{\alpha_n}, \ldots  \rangle $$
where $ \alpha_0 = \alpha$, and for every $ n<\omega $, $ \mu_{\alpha_{n+1}} = j_{U_{\mu_{ \alpha_n }}}\left( \mu_{ \alpha_n} \right) $. Furthermore, the above sequence belongs to $ V $, namely $ \left( \mbox{cf}\left( \mu^* \right) \right)^V = \omega $.
\end{corollary}

\begin{proof}
The first part follows immediately from the previous lemma. Let us concentrate on the second part. Assume that there is no $ \beta<\alpha_0 $ and $ \mu< \mu_{\alpha_0} $ such that $ j_{\beta,\alpha}\left( \mu \right) = \mu_{\alpha_0}$ (if there is, replace $ \mu_{\alpha_0} $ with the least such $ \mu $). Let $ \beta_0, \ldots ,\beta_k $ be a complete nice sequence such that $ \mu_{\alpha_0} = j_{\alpha_0}(h)\left( \beta_0,\ldots, \beta_k \right) $ for some $ h\in V $. It follows that the sequence $ \langle \beta_0, \ldots, \beta_k, \alpha_0, \alpha_1, \ldots , \alpha_n \rangle $ is complete, for every $ n<\omega $. Then, for every $ n<\omega $, a finite iteration $ \langle N_i \colon i\leq n+1 \rangle $ can be defined as in the beginning of this section. If $ f\in V $ is a function such that $ U_{\mu_{\alpha_0} }  = j_{\alpha_0} (f)\left( \kappa, \mu_{\beta_0}, \ldots, \mu_{\beta_k} \right)$, then the sequence $ \langle N_i \colon i<\omega \rangle $ is definable in $ V $, since each step above the first $ k$-many steps in the iteration, uses a measure represented by $ f $. Because each sequence $ \langle \beta_0, \ldots, \beta_k, \alpha_0, \ldots , \alpha_n \rangle $ is complete, the sequence $ \langle \mu_{\alpha_0}, \mu_{\alpha_1}, \ldots, \mu_{\alpha_n}, \ldots \rangle $ is a final segment of the sequence of critical points in the iteration $ \langle N_i \colon i<\omega \rangle $, and thus belongs to $ V $.
\end{proof}

\begin{remark}
We would like to emphasize the point that the characterization of Prikry sequences given in the previous corollary is given only up to some finite initial segment. Let us denote $ \mu = \mu_0 $, which is the first measurable above $ \kappa $ in $ M_U $, and $ \mu^* = k_0\left( \mu_0 \right) $ which is the first measurable above $ \kappa $ in $ M $. We argue that the Prikry sequence of $ \mu^* $ in $ M\left[H\right] $ may have any prescribed finite initial segment $ t\in \left[\mu\right]^{<\omega} $. We use those notations only in the following claim:
\end{remark}

\begin{claim}
	For every finite, increasing sequence $ t\in \left[\mu\right]^{<\omega} $, there exists a condition $ p\in P_{\kappa} $ which forces that $ t $ is an initial segment of the Prikry sequence of $ \mu^* $ in $ M\left[H\right]$. 
\end{claim}

\begin{proof}
	Assume that $ \xi \mapsto t(\xi) $ is a function in $ V $ such that $ \left[\xi\mapsto t(\xi)\right]_{U} = t $. For each $ \xi<\kappa $, let $ s(\xi) $ be the first measurable strictly above $ \xi $. We can assume that for every $ \xi<\kappa $, $ \max \left(t(\xi)\right)<s(\xi) $.
	
	Note that the set $ \{ s(\xi) \colon \xi<\kappa \}\cap \lambda $ is nonstationary in any inaccessible $\lambda \leq  \kappa $: This is clear if $ \lambda $ is not a limit of measurables. If it is, $ \{ s(\xi)  \colon \xi <\kappa\} $ is disjoint to the club of limit points of $ \Delta = \{ \alpha<\kappa \colon \alpha \mbox{ is measurable} \} $ below $ \lambda $. 
	
	Now, let us define a condition $ p\in P_{\kappa} $, with $ \mbox{supp}(p) = \{ s(\xi) \colon \xi<\kappa \} $. We first choose a set $ X\in U $ on which the function $ \xi\mapsto s(\xi) $ with domain $ X $ is injective. Note that by normality of $ U $, every function is either one-to-one or constant modulo $ U $, so such a set $ X\in U $ exists. 
	
	Set, for a given $ \xi\in X $, $ p(s(\xi)) = \langle t(\xi), s\left(\xi\right) \rangle $. This is forced by any condition in $ P\restriction_{s\left(\xi\right)} $ to be a legitimate element of $ \lusim{Q}_{s\left( \xi \right) } $. The condition $ p\in P_{\kappa} $ defined in this way forces that the Prikry sequence of $ \mu^* $ starts with $ t $: Indeed, in $ V\left[G\right] $,
	$$ \{ \xi<\kappa \colon t(\xi) \mbox{ is an initial segment of the Prikry sequence of } s(\xi) \} \supseteq X \in W $$ 
	thus, in $ M\left[H\right] $, $ \left[\xi \mapsto t(\xi)\right]_W $ is an initial segment of the Prikry sequence of the measurable cardinal $ \left[\xi \mapsto s(\xi)\right]_W $. But by lemma \ref{Lemma: NS, k},
	$$\left[\xi \mapsto t(\xi)\right]_{W} = k\left(  \left[\xi \mapsto t(\xi)\right]_{U} \right)  =  k(t) = t    $$
	and clearly--
	$$ \left[\xi \mapsto s(\xi)\right]_{W} = \mu^* $$
	so in $ M\left[H\right] $, $ t $ is an initial segment of the Prikry sequence added to $ \mu^* $.\end{proof}

Let us prove now that for every measurable $ \mu^*$ above $ \kappa $ in $ M $, $ \mu^* $ has the form $ k_{\alpha}\left( \mu_{\alpha} \right) $ for some $ \alpha< \kappa^* $. In particular, in the light of corollary \ref{Corollary: NS, how Prikry sequences look}, $ \left(\mbox{cf}(\mu^*)\right)^V = \omega $. 

\begin{lemma} \label{Lemma: NS, every measurable of M is the image of some mu alpha}
Assume that $\mu^* \in \left(\kappa , \kappa^*\right) $ is measurable in $ M $. Then $ \mu^* = k_{\alpha}\left( \mu_{\alpha} \right) $ for some $ \alpha< \kappa^* $.
\end{lemma}

\begin{proof}
Let $ \beta < \kappa^* $ be the first such that, for some $ \mu\leq \mu^* $, $ \mu^* = k_{\beta}\left( \mu \right) $. Then $ \mu $ is measurable in $ M_{\beta} $. $ \beta $ is either $ 0 $ or a successor by its minimality. Assume first that $ \beta= \alpha+1 $.  $ \mu = \mu_{\alpha} $ cannot hold since $ \mu_{\alpha} $ is not measurable in $ M_{\alpha+1} $. If $ \mu < \mu_{\alpha} $ then $ j_{\alpha,\beta}\left( \mu \right) = \mu $, contradicting the minimality. Thus assume that $ \mu > \mu_{\alpha} = \bar{\mu}_{\beta} = \sup\{ \mu_{\beta'} \colon \beta'< \beta \} $. Recall that $ \mu $ is measurable in $ M_{\beta} $. By lemma \ref{Lemma: NS, every measurable above mu bar has cofinility above kappa in V}, $\left(\mbox{cf}\left( \mu \right)\right)^V > \kappa$. Therefore, for some $ \gamma\in \left[ \beta, \kappa^*\right) $, $ \mu = \mu_{\gamma} $. Hence $ k_{\gamma}\left( \mu_{\gamma} \right) = \mu^* $.

If $ \beta =0 $ then $ \mu $ is measurable in $ M_0 $ above $ \kappa $ and below $ \kappa^* $, and clearly $ \left( \mbox{cf}(\mu) \right)^V > \kappa $. So, again, there exists $ \gamma< \kappa^* $ such that $ \mu = \mu_{\gamma} $, and $ k_{\gamma}\left( \mu_{\gamma} \right) = \mu^* $.
\end{proof}

\begin{corollary} \label{Corollary: NS, best definition of mu_alpha}
Assume that $ \alpha<\kappa^* $ is limit, and denote $ \bar{\mu} = \sup\{ \mu_{\alpha'} \colon \alpha'<\alpha \} $. Assume that $ \bar{\mu} $ is measurable in $ M_{\alpha} $. Then $ \left( \mbox{cf}(\alpha) \right)^V$ is either $ \omega $ or $ \kappa^+ $. In the former case, $ \bar{\mu} $ is measurable in $ M $. In the latter case, $ \bar{\mu} = \mu_{\alpha} $ is a non-measurable inaccessible cardinal in $ M $. Moreover:
\begin{enumerate}
	\item If $ \bar{\mu} $ is not measurable in $ M_{\alpha} $ or $ \left(\mbox{cf}\left( \alpha \right)\right)^V > \kappa $, $ \mu_{\alpha} $ is the first measurable $\geq$ $ \bar{\mu} $ in $ M_{\alpha} $ (this includes the case where $ \alpha $ is successor, since, in this case, $ \mu_{\alpha-1} $ is not measurable in $ M_{\alpha} $).
	\item Else, $ \bar{\mu} $ is measurable in $ M_{\alpha} $ and $\left(\mbox{cf}(\alpha)\right)^V = \omega$, and then $ \mu_{\alpha} = \bar{\mu} $.
\end{enumerate}
\end{corollary}

\begin{proof}
Assume that $ \bar{\mu} $ is measurable in $ M_{\alpha} $. If $ \left( \mbox{cf}(\alpha) \right)^V \leq \kappa $, then $ \bar{\mu} < \mu_{\alpha} $, so $ \bar{\mu} = k_{\alpha}\left( \bar{\mu} \right) $ is measurable in $ M $. By the previous lemma, $ \bar{\mu} = k_{\gamma}\left( \mu_\gamma  \right) $ for some $ \gamma< \kappa^* $. By corollary \ref{Corollary: NS, how Prikry sequences look}, $ \left(\mbox{cf}\left( \bar{\mu} \right)\right)^V = \omega $. Hence $ \left(\mbox{cf}(\alpha)\right)^V = \omega $. 
\end{proof}

\begin{remark}
Recall that $ j_W\left( \mathcal{U} \right)\setminus \kappa\in M $ is sufficient for the definability of  $ j_W\restriction_{V} $ over $ V $. Let us argue that it is not necessary.

For every measurable $ \eta < \kappa $, let $ \langle s^{n}(\eta) \colon n<\omega \rangle $ be the increasing enumeration of the first $ \omega $-many measurables above $ \eta $ which carry at least $ \eta $-many normal measures of Mitchell order $ 0 $. For each such $ \eta  $ and $ n<\omega $, let $\vec{F}\left( s^{n}(\eta) \right)$ be an enumeration for all the normal measures of order $ 0 $ on $ s^{n}(\eta) $. Fix an unbounded nonstationary subset $ X\subseteq \Delta $ such that for every $ \eta\in X $ and $ n<\omega $, $ s^{n}(\eta)\notin X $.  Let $ P $ be the forcing notion which uses, at stage $ s^{n}(\eta) $ where $ \eta\in X $ and $ n<\omega $, the measure which extends $ \left(\vec{F}\left( s^{n}\left(\eta\right) \right)\right)\left( \eta_n \right) $.  Here, $ \eta_n<\eta $ is the $ n $-th element in the Prikry sequence of $ \eta $ in $ M\left[H\right] $. For every other measurable, use the measure chosen first with respect to a prescribed well order of $ V_{\kappa} $.

Pick a generic set $ G\subseteq P $ such that $ G $ contains a condition $ p $ such that $ X\subseteq \mbox{supp}(p) $, but for every $ \xi\in X $, $ p\restriction_{\xi} \Vdash t^{p}_{\xi} = \langle \rangle $.

Then $ j_W\left( \mathcal{U} \right)\setminus \kappa \notin M $, since the measures used in $ j_W(P) $ on $ M $-measurables above $ \kappa $ code the Prikry sequences of all the measurables in $ j_W(X)\setminus \kappa $. 

However, $ j_W\restriction_{V} $ is definable in $ V $: Assume that $ \alpha<\kappa^* $. If there is no $ \eta\in j_{\alpha}(X) $ and $ n<\omega  $ such that $ \mu_\alpha = s^{n}(\eta) $, $ U_{\mu_{\alpha}} $ is the first measure on $ \alpha $ with respect to the image under $ j_{\alpha} $ of the prescribed well order on $ V_{\kappa} $. Otherwise, assume that $ \eta \in j_{\alpha}(X) $, $ n<\omega $ and $ \mu_\alpha = s^{n}(\eta) $. Denote $ \eta^* = k_{\alpha}\left( \eta \right) $, so that $ k_{\alpha}\left( \mu_{\alpha} \right) = s^{n}\left( \eta^* \right) $. Let $ \beta = \beta_0<\kappa^* $ be the least such that $ k_{\beta}\left( \mu_{\beta} \right) = \eta^* $. We argue that the Prikry sequence of $ \eta^* $ in $ M\left[H\right] $ is the sequence of critical points taken by iteration $ U_{\mu_{\beta} } $ $ \omega  $-many times over $ M_{\beta} $. This will follow once we prove that $ \mu_{\beta} $ is the first element in the Prikry sequence of $ \eta^* $, and this is true since $ \eta^* \in j_{W}(X) $ and there exists a condition $ p\in G $ which forces that $ t^{p}_{\xi} = \langle \rangle  $ for every $ \xi\in X $. Thus, we can assume that $ \langle \mu_{\beta_0}, \mu_{\beta_1}, \ldots, \mu_{\beta_n}, \ldots  \rangle $ is the Prikry sequence of $ \eta^* $ in $ M\left[H\right] $. 

Recall that $ k_{\alpha}\left( U_{\mu_{\alpha} } \right) = j_W\left( \mathcal{U} \right)\left( k_{\alpha}\left( \mu_{\alpha} \right) \right) $; by the definition of the forcing, $ j_W\left( \mathcal{U} \right)\left( k_{\alpha}\left( \mu_{\alpha} \right) \right) $ is the Prikry forcing taken with the measure-- 
$$\left(j_W\left( \vec{F} \right)\left( k_{\alpha}\left( \mu_{\alpha} \right) \right)\right)\left(  \mu_{\beta_n} \right)$$
thus, $ U_{\mu_{\alpha} } $ can be computed in $ V $ as follows: first, calculate over $ M_{\beta_0} $ (which is already definable in $ V $ by induction) the sequence $ \langle \mu_{\beta_n} \colon n<\omega \rangle $, which are the critical points in the iteration of length $ \omega $ with $ U_{\mu_{\beta_0} } $ over $ M_{\beta_0} $ (here, $ U_{\mu_{\beta} } $ is the least measure with respect to the image under $ j_{\beta} $ of the prescribed well order on $ V_{\kappa} $); then, compute $ U_{\mu_{\alpha} }= \left(j_{\alpha}\left( \vec{F} \right)\left( \mu_{\alpha} \right)  \right)\left( \mu_{\beta_n}  \right) $.

\end{remark}

\bibliography{NonStatRestElm.bib}

\begin{thebibliography}{1}

\bibitem{ben2017homogeneous}
Omer Ben-Neria and Spencer Unger.
\newblock Homogeneous changes in cofinalities with applications to hod.
\newblock {\em Journal of Mathematical Logic}, 17(02):1750007, 2017.

\bibitem{RestElm}
Moti Gitik and Eyal Kaplan.
\newblock On restrictions of ultrafilters from generic extensions to ground
  models.

\bibitem{hamkins2001gap}
Joel~David Hamkins.
\newblock Gap forcing.
\newblock {\em Israel Journal of Mathematics}, 125(1):237--252, 2001.

\bibitem{schindler2006iterates}
Ralf Schindler.
\newblock Iterates of the core model.
\newblock {\em Journal of Symbolic Logic}, pages 241--251, 2006.

\end{thebibliography}

\end{document}